\documentclass[a4paper,11pt]{article}
\usepackage{fullpage}
\usepackage[utf8]{inputenc}
\usepackage[english]{babel} 

\usepackage{amsfonts, amssymb, amsbsy, amsmath, amsthm}
\usepackage{dsfont}
\usepackage{authblk}
\allowdisplaybreaks

\usepackage[hyphens]{url}
\usepackage[hyphenbreaks]{breakurl}

\usepackage{mathtools}
\usepackage{xcolor}
\definecolor{darkblue}{rgb}{0,0,0.7}
\usepackage[colorlinks=true, allcolors=darkblue]{hyperref}
\usepackage[capitalise]{cleveref}
\Crefname{algocf}{Algorithm}{Algorithms}
\Crefname{equation}{Equation}{Equations}
\Crefname{figure}{Figure}{Figures}

\hypersetup{
    colorlinks=true,
    breaklinks=true,            %3
}

\usepackage{pgfplots}
\usepackage{caption}
\usepackage{wrapfig}
\usepackage{graphicx}
\usepackage{tikz-cd}
\usepackage{tikz}
\usepackage{mathrsfs}
\usepackage{float}
\usetikzlibrary{shapes.geometric} 
\usetikzlibrary{positioning}

\newtheorem{theorem}{Theorem}[section]
\newtheorem{lemma}[theorem]{Lemma}
\newtheorem{definition}[theorem]{Definition}

\newtheorem{conjecture}[theorem]{Conjecture}
\newtheorem{example}[theorem]{Example}
\newtheorem{claim}[theorem]{Claim}
\newtheorem{proposition}[theorem]{Proposition}

\usepackage{hhline}
\usepackage{stmaryrd}
\usepackage{xfrac}
\usepackage{algorithm}
\usepackage{algorithmic}
\numberwithin{equation}{section}

\newcommand{\mA}{\mathcal{A}}
\newcommand{\mK}{\mathcal{K}}
\newcommand{\mB}{\mathcal{B}}

\newcommand{\mX}{\mathcal{X}}

\newcommand{\mF}{\mathcal{F}}
\newcommand{\mT}{\mathcal{T}}
\newcommand{\mG}{\mathcal{G}}
\newcommand{\mL}{\mathcal{L}}

\newcommand{\mH}{\mathcal{H}}

\newcommand{\mS}{\mathcal{S}}
\newcommand{\QQ}{\mathbb{Q}}
\newcommand{\RR}{\mathbb{R}}
\newcommand{\NN}{\mathbb{N}}

\newcommand{\bw}{\mathbf{w}}

\DeclareMathOperator*{\argmin}{arg\,min}

\usepackage[textsize=tiny]{todonotes}

\newcommand{\whiteboxedchar}[1]{
    \hspace{-2pt}\tikz[baseline=(char.base)]{
        \node[shape=rectangle, draw=black, thin, rounded corners=1.5pt, text=black, inner xsep=3pt, inner ysep=2pt] (char) {#1};
    }\hspace{-2pt}
}

\newcommand{\boxedchar}[1]{
    \hspace{-2pt}\tikz[baseline=(char.base)]{
        \node[shape=rectangle, fill=black, rounded corners=1.5pt, text=white, inner xsep=3pt, inner xsep=3pt, inner ysep=2pt] (char) {\normalfont #1};
    }\hspace{-2pt}
}

\begin{document}

\title{The Four-Color Ramsey Multiplicity of Triangles}
\author[1,2]{Aldo Kiem}
\author[1,2]{Sebastian Pokutta}
\author[1,2]{Christoph Spiegel}
\affil[1]{\small Technische Universit\"at Berlin, Institute of Mathematics}
\affil[2]{\small Zuse Institute Berlin, Department AIS2T, \emph{lastname}@zib.de}
\maketitle              % typeset the header of the contribution
\begin{abstract}
    We study a generalization of a famous result of Goodman and establish that asymptotically at least a $1/256$ fraction of all triangles needs to be monochromatic in any four-coloring of the edges of a complete graph. We also show that any large enough extremal construction must be based on a blow-up of one of the two $R(3,3,3)$ Ramsey-colorings of $K_{16}$. This result is obtained through an efficient flag algebra formulation by exploiting problem-specific combinatorial symmetries that also allows us to study some related problems.
\end{abstract}

\section{Introduction}\label{sec:introduction}

In 1959, Goodman \cite{Goodman_1959} established precisely how few monochromatic triangles any two-edge-coloring of the complete graph on $n$ vertices can contain, implying that asymptotically at least $1/4$ of all triangles need to be monochromatic as $n$ tends to infinity. Subsequently, in~\cite{Goodman_1985}, he also asked for an answer to the natural generalization of this problem to more than two colors.\footnote{He in fact calls the three-color version of this question ``an old and difficult problem'' and raises the question of more than three colors in Section~6 of~\cite{Goodman_1985}. The precise origin of this problem is unclear.} It took over 50 years and the advent of flag algebras for even the case of three colors to be settled: Cummings et al.~\cite{CummingsEtAl_2013} showed that asymptotically at least a $1/25$ fraction of all triangles need to be monochromatic in any three-edge-coloring of $K_n$. For $n$ large enough they also precisely characterize the set of extremal constructions, showing that the problem is closely linked to the Ramsey Number $R(3,3) = 6$ as previously noted by Fox~\cite[Theorem  5.2]{fox2008there}. The purpose of this paper is to study the next iteration of this problem, in particular establishing an answer in the affirmative to Question 4 in~\cite{CummingsEtAl_2013} for the case of four colors.
\begin{theorem}\label{thm:fourcolortriangles}
    Asymptotically at least a $1/256$ fraction of all triangles are monochromatic in any four-edge-coloring of $K_n$ and any sufficiently large extremal coloring must be based on one of the two $R(3,3,3)$ Ramsey-colorings of $K_{16}$.
\end{theorem}
We will give a more precise statement in the form of \cref{eq:fourcolortriangles} along with additional context in \cref{sec:background}. The proof of this result relies on the flag algebra framework of Razborov~\cite{Razborov_2007, CoreglianoRazborov_2020}. This allows one to apply a formalized double counting and Cauchy-Schwarz-type argument to obtain bounds for classic problems in Tur\'an and Ramsey theory by solving a concrete semidefinite programming (SDP) formulation. Broadly speaking, the larger this formulation, the better the derived bound becomes.

The major hurdle in establishing \cref{thm:fourcolortriangles} therefore consisted of deriving an efficient formulation by identifying and exploiting combinatorial symmetries through a parameter-dependent notion of automorphisms. The resulting proof likely constitutes the largest \emph{exact} flag algebra calculation done to date. The methods developed to derive it strengthen the previous approach of modifying the underlying notion of isomorphism and generalize Razborov's invariant-anti-invariant decomposition~\cite{Razborov_2010}. They are applicable whenever the object we are minimizing has previously ignored symmetries and we hope that they will therefore find further applications. Accompanying these computational improvements, we also give an extension of the  stability argument previously developed for the three-color case in~\cite{CummingsEtAl_2013}. We generalize it to the case of an arbitrary number of colors and establish a strong link between the problem of determining the Ramsey number and the Ramsey multiplicity problem.

\paragraph{Outline} We start with a summary of the history behind the Ramsey multiplicity problem and related literature and then give precise statements for our results in \cref{sec:background}.  A formal definition of flag algebras will be given in \cref{sec:flagalgebras} along with our proposed symmetry-based improvements in the SDP formulation in \cref{sec:improvements}. \cref{sec:stability} will contain the necessary context for deriving a stability result. The technical details behind the proof of the main statement are given in \cref{sec:proof}. Finally, \cref{sec:discussion} contains some concluding discussions and avenues for future research.

\section{The Ramsey Multiplicity Problem}\label{sec:background}

We are studying the family of $c$-colorings of the edges between a finite number of vertices, that is maps $G: \big\{ \{u, v\} \mid u,v \in V, \,  u \neq v \big\} \to [c] = \{1, \ldots, c\}$ where $V$ is any finite set. Following common notation for graphs, we write $V(G) = V$ for the \emph{set of vertices} of such a coloring $G$ as well as $v(G) = |V(G)|$ for its \emph{order}. We also write $E_i(G) = G^{-1} (\{i\})$ for the set of edges colored with color $i \in [c]$. Let $\mG^{(c)}$ denote the set of all such colorings and $\mG_n^{(c)}$ the set of all colorings of order $n$. Two colorings $G, G' \in \mG^{(c)}$ are \emph{isomorphic}, written as $G \simeq G'$, if there exists a bijection $\varphi: V(G) \overset{\sim}{\hookrightarrow} V(G')$ such that $G \equiv G' \circ \varphi$. Given $G \in \mG^{(c)}$ and $S \subseteq V(G)$, we write $G[S]$ for the restriction of $G$ to edges between vertices in $S$. We say $H$ is a \emph{sub-coloring} of $G$ if there exists a \emph{copy} of $H$ in $G$, that is if there exists $S \subseteq V(G)$ s.t. $G[S] \simeq H$.

Given colorings $H \in \mG_k^{(c)}$ and $G \in \mG_n^{(c)}$, we write $p(H; G) = |\{ S \subseteq V(G) \mid G[S] \simeq H\}| / {n \choose k}$ for the \emph{density} of $H$ in $G$. Note that $p(H;G) = 0$ if $n < k$. Denoting the monochromatic coloring of the edges between vertices in $[t]$ with color $i \in [c]$ by $K_t^i$, a multi-color version of Ramsey's theorem~\cite{ramsey1987problem} states that for any $t_1, \ldots, t_c \in \NN$ the number
\begin{equation*}
    R(t_1, \ldots, t_c) = \min \big( \big\{n \in \NN \mid \{ G \in \mG_n^{(c)} \mid p(K_t^1; G) + \ldots + p(K_t^c; G) = 0 \} = \emptyset \big\} \big)
\end{equation*}
is in fact finite. For the diagonal case, where $t_1 = \ldots = t_c$, we write $R_c(t) = R(t, \ldots, t)$. The study of the parameter
\begin{equation*}
    m_c(t; n) = \min_{G \in \mG_n^{(c)}} p(K_t^1; G) + \ldots + p(K_t^c; G)
\end{equation*}
is known as the \emph{Ramsey multiplicity problem} for cliques.  A simple double-counting argument, see \cref{lemma:double_counting} in the next section, establishes that $m_c(t; n)$ is monotonically increasing, so that the limit $m_c(t) = \lim_{n \to \infty} m_c(t; n)$ is well defined and satisfies $m_c(t) \geq m_c(t; n)$ for any $n \in \NN$. Note that $m_c(t; n) > 0$ as long as $n \geq R_c(t)$ and therefore $m_c(t) > 0$ by Ramsey's theorem.

Concerning upper bounds for $m_c(t)$, coloring the edges uniformly at random with the $c$ colors establishes that
\begin{equation} \label{eq:prob_upper}
    m_c(t) \leq c^{1 - {t \choose 2}}.
\end{equation}
Another way to obtain an upper bound is by blowing up a coloring of the edges of a \emph{looped} complete graph, that is a map $C: \big\{ \{u, v\} \mid u, v \in V \big\} \to [c]$. We use the same notation concerning the vertex and edge set as we did for unlooped colorings and write $\mL^{(c)}$ for the set of all such colorings as well as $\mL_n^{(c)}$ for colorings of order $n$. A coloring $H \in \mG^{(c)}$ \emph{embeds} into a given $C \in \mL^{(c)}_k$, if there exists a (not necessarily injective) map $\varphi: V(H) \to V(C)$ satisfying $H( \{u, v \} ) = C ( \{ \varphi(u), \varphi(v) \})$ for all  $u, v \in V(H)$.
We now let $\mB (C) = \{ H \in \mG^{(c)} \mid H \text{ embeds into } C \}$ denote the \emph{family of blow-up colorings} of $C$. Note that $\mB (C)$ contains graphs of arbitrarily large order. 
Given some probability vector $\bw = (w_v)_{v\in V(C)}$, that is $w_v \geq 0$ for $v \in V(C)$ and $\sum_{v\in V(C)}w_v = 1$, we additionally write $\mB_\bw (C)$ for the \emph{weighted family of blow-up colorings} of $C$, that is the set of all elements in $\mB (C)$ for which $\varphi$ can be chosen such that $\big| |\varphi^{-1}(\{v\})| - w_v \cdot v(G) \big| \leq 1$ for all $v\in V(C)$. Finally, let
\begin{equation*}
    \hat{p}(H; C, \bw) = \sum_{\substack{\varphi\text{ embeds}\\H \text{ into } C}} \, \prod_{v\in V(H)} w_{\varphi(v)}
\end{equation*}
denote the \emph{$\bw$-weighted embedding density}.
\begin{lemma} \label{lemma:blowup}
    Given any $C \in \mL^{(c)}_k$ and a probability vector $\bw$ for $C$, we have 
    \[
        m_c(t) \leq \hat{p}(K_t^1; C, \bw) + \ldots + \hat{p}(K_t^c; C, \bw).
    \]
\end{lemma}
\begin{proof}
    Clearly $m_c(t) \leq \lim_{n \to \infty} \min_{G \in \mG_n^{(c)} \cap \mB_\bw(C)} p(K_t^1; G) + \ldots + p(K_t^c; G)$, so the statement follows by observing that, asymptotically as $n$ tends to infinity, $p(H;G) = \hat{p}(H; C, \bw) + o(1)$ for any $H \in \mG^{(c)}$ and $G \in \mG_n^{(c)} \cap \mB_\bw(C)$.
\end{proof}
In our case, the most relevant candidates for colorings $C$ that can be used in \cref{lemma:blowup} are obtained by considering a Ramsey-coloring on $r = R_{c-1}(t) - 1$ vertices avoiding cliques of size $t$ in any of the $c-1$ colors. Coloring the loops with the additional $c$-th color and using the uniform weighting $\mathbf{1 / r} = (1/r, \ldots, 1/r)$, this implies an upper bound of
\begin{equation} \label{eq:ramsey_upper}
    m_c(t) \leq (R_{c-1}(t) - 1)^{1-t},
\end{equation}
see also Theorem 5.2 in~\cite{fox2008there}.\footnote{For an example demonstrating that considering weightings beyond the uniform along and not coloring all loops using the same color can be relevant, see the construction conjectured to be tight for $c_{4,4}$ in~\cite{pikhurko2013minimum}.}
The result of Goodman~\cite{Goodman_1959} now states that
\begin{equation}
    m_2(3; n) \cdot {n \choose 3} = \left\{\begin{array}{ll}
        2 {u \choose 3} & \text{if } n = 2u,\\
        \frac{2}{3} u (u-1) (4u+1) & \text{if } n = 4u+1,\\
        \frac{2}{3} u (u+1) (4u-1) & \text{if } n = 4u+3,
        \end{array}\right.
\end{equation}
implying that $m_2(3) = 1/4$. This aligns both with the probabilistic upper bound stated in \cref{eq:prob_upper} as well as the Ramsey upper bound stated in \cref{eq:ramsey_upper}, where for the latter we are relying on the trivial case of Ramsey's theorem, that is $R_1(t) = R(t) = t$.

Given that the former bound dominates when $t$ grows as long as $c = 2$, Erd\H{o}s suggested~\cite{Erdos_1962} that the probabilistic upper bound should always be tight in this case. 
This was disproven by Thomason~\cite{Thomason_1989} for any $t \geq 4$ and a significant number of results since then have tried to either determine improved asymptotics for $m_2(t)$ or specific values of it for small $t$~\cite{Conlon_2012,deza2012conjecture,even2015note,Franek_2002,FranekRodl_1993,Giraud_1979,GrzesikEtAl_2020,jagger1996multiplicities,Niess_2012,Sawin_2021,Sperfeld_2011,Thomason_1997, Wolf_2010,ParczykEtAl_2022}. The problem also links to Sidorenko's famous open conjecture and the search for a characterization of common graphs, cf.~\cite{BurrRosta_1980, Sidorenko_1993}. As of now, even $m_2(4)$ remains open, with the best current lower and upper bounds of $0.0296 \leq m_2(4) \leq 0.03014$ respectively due to Grzesik et al.~\cite{GrzesikEtAl_2020} as well as Parczyk et al.~\cite{ParczykEtAl_2022}. Note that we will give a slight improvement to the upper bound in \cref{sec:discussion}. For the asymptotic values, there has likewise been scarce progress, with the current best lower bound of $C^{-t^2 (1 + o(1))} \leq m_2(t)$ for $C \approx 2.18$ due to Conlon~\cite{Conlon_2012} and the best upper bound of $m_2(t) \leq 0.835 \cdot 2^{1-{t\choose 2}}$ for $t \geq 7$ due to Jagger, {\v{S}}{\'t}ov{\'\i}{\v{c}}ek, and Thomason~\cite{jagger1996multiplicities}.
Given the lack of progress on the two-color, diagonal version, there are two obvious directions to explore: the case of more colors, where $c > 2$, as well as the off-diagonal case, where $t_1 \neq t_2$.

\subsection{Increasing the number of colors}\label{sec:more_colors_intro}

Studying monochromatic triangles for more than two colors was, as already mentioned in the introduction, suggested by Goodman~\cite{Goodman_1985} and resolved for $c = 3$ by Cummings et al.~\cite{CummingsEtAl_2013}, whose result aligns with~\cref{eq:ramsey_upper} since $R_2(3) = 6$. In order to state their result in its fullest strength, let $C_{R(3,3)}$ denote the coloring in $\mL^{(3)}_{5}$ obtained by taking the unique Ramsey $2$-coloring of a complete graph on five vertices that avoids monochromatic triangles and coloring the loops with the third color, that is $E_1 (C_{R(3,3)} )$ and $E_2 (C_{R(3,3)} )$ both are $5$-cycles and $E_3 (C_{R(3,3)} )$ contains all five loops. Let 
$\mG^{(3)}_\textrm{ex} \subset \mG^{(3)}$ now consist of all colorings that can be obtained by (i) selecting an element in $\mB_\mathbf{1/5} (C_{R(3,3)} )$, (ii) recoloring some of the edges from the first or second color to instead use the third color without creating any additional monochromatic triangles, and (iii) applying any permutation of the colors. Note that the second step implies that the recolored edges must form a matching between any of the five {\lq}parts{\rq}, though not every such recoloring avoids additional triangles.
\begin{theorem}[Cummings et al.~\cite{CummingsEtAl_2013}]
    There exists an $n_0 \in \NN$ such that any 
    element in $\mG_n^{(3)}$ of order $n \geq n_0$ minimizing the number of monochromatic triangles must be in $\mG^{(3)}_\textrm{ex}$.
\end{theorem}
The result characterizes extremal constructions for large enough $n$, though more recently
%, though more recently Pikhurko et al.~\cite{PikhurkoEtAl_2019} outlined how to obtain strong forms of stability
there has been increasing interest in deriving stability results based on flag algebra calculations~\cite{PikhurkoEtAl_2019}. Let $C'_{R(3,3,3)}$ and $C''_{R(3,3,3)}$ denote the two colorings in $\mL^{(4)}_{16}$ obtained in a similar way to the previously defined $C_{R(3,3)}$ by respectively taking the two Ramsey $3$-coloring of a complete graph on $16$ vertices that avoid monochromatic triangles~\cite{greenwood1955combinatorial, kalbfleisch1968maximal, laywine1988simple, radziszowski2011small} and coloring the vertices with the fourth color. Mirroring the construction of $\mG^{(3)}_\textrm{ex}$, we let $\mG^{(4)}_\textrm{ex} \subset \mG^{(4)}$ consist of all colorings that can be obtained by

\begin{itemize} \setlength\itemsep{0em}
    \item[(i)] selecting an element in $\mB_\mathbf{1/16} (C'_{R(3,3,3)} )$ or $\mB_\mathbf{1/16} (C''_{R(3,3,3)} )$,
    \item[(ii)] recoloring some of the edges from any of the first, second or third color to instead use the fourth color without creating any additional monochromatic triangles,
    \item[(iii)] applying any permutation of the four colors.
\end{itemize}

\begin{theorem} \label{eq:fourcolortriangles}
    There exists an $n_0 \in \NN$ such that for any $\varepsilon > 0$ there exists $\delta > 0$ such that any $G \in \mG_n^{(4)}$ of order $n \geq n_0$ with $\sum_{i=1}^cp(K_3^i; G) \leq m_4(3; n) + \delta$ can be turned into an element of $\mG^{(4)}_\textrm{ex}$ by recoloring at most $\varepsilon \binom{n}{2}$ edges.
\end{theorem}
Note that this implies that any large enough element in $\mG_n^{(4)}$ minimizing the number of monochromatic triangles must be in $\mG^{(4)}_\textrm{ex}$. Our results in fact show that it likewise can be obtained for the case of three colors.

\subsection{The off-diagonal case}

The second of the previously suggested directions, that is considering the off-diagonal case, has recently started to receive some attention~\cite{ParczykEtAl_2022, behague2022common, moss2023off, Hyde2023} with two competing notions of off-diagonal Ramsey multiplicity having been suggested. The first is due to Parczyk et al.~\cite{ParczykEtAl_2022} and is concerned with determining 
\begin{equation*}
    m(t_1, \ldots, t_c; n) = \min_{G \in \mG_n^{(c)}} p(K_{t_1}^1; G) + \ldots + p(K_{t_c}^c; G).
\end{equation*}
This generalizes the previously defined $m_c(t; n)$ but does not consider the inherent imbalance when for example $c = 2$ and $t_1 \ll t_2$; minimizing $p(K_{t_1}^1; G) + p(K_{t_2}^2; G)$ in this case will be equivalent to enforcing $p(K_{t_1}^1; G) = 0$ and minimizing $p(K_{t_2}^2; G)$, a related problem previously suggested by Erd\H{o}s~\cite{Erdos_1962, nikiforov2001minimum, das2013problem, ParczykEtAl_2022}. This issue was already noted in~\cite{ParczykEtAl_2022} and subsequently addressed by Moss and Noel~\cite{moss2023off}, who instead suggested determining 
\begin{equation*}
    m_s(t_1, \ldots, t_c; n) = \min_{G \in \mG_n^{(c)}} \max_{\substack{\lambda_1, ..., \lambda_c \geq 0 \\ \lambda_1 + \ldots + \lambda_c = 1}} \lambda_1 \, p(K_{t_1}^1; G) + \ldots + \lambda_c \, p(K_{t_c}^c; G).
\end{equation*}
We will use $m(t_1, \ldots, t_c)$ as well as $m_s(t_1, \ldots, t_c)$ to respectively denote the limits of both of these functions as $n$ tends to infinity. Both notions generalize the previous diagonal definition and clearly $m_s(t_1, \ldots, t_c) \geq m(t_1, \ldots, t_c)$. Unsurprisingly, determining $m_s(t_1, \ldots, t_c)$ has proven much more difficult, with $m(3, 4)$ and $m(3, 5)$ having been settled in~\cite{ParczykEtAl_2022} and $m_s(3, 4)$ still remaining open. Here we derive the following result for the weaker of the two notions.
\begin{proposition} \label{prop:m334}
    We have $m(3, 3, 4) = 1/125$.
\end{proposition}

The upper bound follows immediately by generalizing \cref{eq:ramsey_upper} to the off-diagonal case, that is by noting that
\begin{equation} \label{eq:ramsey_upper_offdiag}
    m(t_1, \ldots, t_c) \leq (R(t_1, \ldots, t_{c-1}) - 1)^{1-t_c}
\end{equation}
and inserting $R(3, 3) = 6$. The lower bound was derived using the same improvements to the flag algebra calculus that we developed to derive our main result.

\section{A formal introduction to flag algebras} \label{sec:flagalgebras}

The following is intended as a reasonably concise introduction to Razborov's flag algebras and the associated semidefinite programming method. The purpose of this section is mainly to introduce notation that will be needed in the subsequent sections in order to establish the more efficient problem formulations as well as stability. For a more gentle introduction to the topic, we refer the reader to~\cite{SilvaEtal_2016}. For a comprehensive treatment, see~\cite{Razborov_2007, CoreglianoRazborov_2020}.

\subsection{Flags and types, and their densities}\label{sec:def_flags_and_density}

Throughout this section, we let the number of colors $c$ be arbitrary but fixed. A \emph{flag} $F = (G, \psi)$ consists of a coloring $G \in \mG^{(c)}$ along with a partial but distinct labeling of its vertex set  $\psi: [k] \hookrightarrow V(G)$ for some non-negative integer $k$. Note that we allow the empty map $\varnothing: \emptyset \hookrightarrow V(G)$ when $k = 0$. A flag inherits its order, vertex set, and colored edge sets from the underlying coloring $G$ and we also write $F(\{v,w\}) = G(\{v, w\})$ for the color of the edge between $v, w \in V(G)$ with $v \neq w$. Two flags $F = (G, \psi)$ and $F' = (G', \psi')$ are \emph{isomorphic}, denoted by $F \simeq F'$, if there exists a bijection $\varphi: V(G) \overset{\sim}{\hookrightarrow} V(G')$ such that $G \equiv G' \circ \varphi$ and $\psi' \equiv \varphi \circ \psi$. Note that this in particular necessitates that $v(G) = v(G')$ and that the domain of $\psi$ and $\psi'$ are both $[k]$ for the same $k \geq 0$. This is a natural extension of colorings that allows for partially fixing some of the vertices of underlying complete graph. If $k = 0$ then the flag is just a coloring as defined in the previous section along with an empty map. If $k = v(G)$ on the other hand, we refer to $F$ as a \emph{type} and will commonly use the notation $\tau$ for it. When $k = v(G) = 0$, we refer to it as the \emph{empty type} $\varnothing$. Given some flag $F = (G, \psi)$ and subset of its vertices $\operatorname{im} (\psi) \subseteq S \subseteq V(G)$, we write $F[S] = (G[S], \psi)$ and say $F[S]$ is  a \emph{subflag} of $F$. We say that $F$ is \emph{of type $\tau$} if its labeled part is isomorphic to $\tau$, that is if $F[\operatorname{im} (\psi)] \simeq \tau$.\footnote{Note that we are using the just introduced notion of isomorphisms for flags instead of colorings here, meaning that the type already needs to be ordered in the correct way. Using isomorphisms instead of equivalence here just allows us to ignore the exact nature of the underlying vertex sets.} We denote the set of arbitrary but canonical representatives of isomorphism classes of flags of a given type $\tau$ by $\mF^\tau$ and we use $\mF^\tau_n$ to refer to those of order $n$. Note that $\mF^\tau_n = \emptyset$ when $n < v(\tau)$ and $\mF^\tau_{v(\tau)} = \{\tau\}$.

For a given type $\tau$ and flags $F_1, \ldots, F_\ell, F \in \mF^\tau$ with $F = (G, \psi)$, we let
\begin{equation*}
    \mathbf{S}  = \{ (S_1, \ldots, S_\ell) \mid |S_i| = v(F_i), \, \operatorname{im}(\psi)  \subseteq S_i \subseteq V(G), \, S_i \cap S_j = \operatorname{im}(\psi) \text{ if } i \neq j \} %= \mathbf{S} (F_1, \ldots, F_\ell; F)
\end{equation*}
denote the family of sunflowers of size and intersection dictated by $F_1, \ldots, F_\ell, F$ and define
\begin{equation*}
    p(F_1, \ldots, F_\ell; F) = | \{ (S_1, \ldots, S_\ell) \in \mathbf{S} \mid F[S_i] \simeq F_i \text{ for all } 1 \leq i \leq \ell \} | \, / \, |\mathbf{S}|
\end{equation*}
to be the \emph{tuple density} of $F_1, \ldots, F_\ell$ in $F$. Note that $\mathbf{S} = \emptyset$ if $v(F_1) + \ldots + v(F_\ell) - (\ell-1) \, v(\tau) > v(F)$, in which case we let $p(F_1, \ldots, F_\ell; F) = 0$. We will in fact only need the cases where $\ell = 1$ (the \emph{density}) and $\ell = 2$ (the \emph{pair density}) and note that for $\tau = \varnothing$ the notion of density aligns with the density of colorings defined in the previous section, that is we have $p(H; G) = p\big((H, \varnothing); (G, \varnothing) \big)$. What follows is a fundamental averaging equality that can be proven from a simple double counting argument.

\begin{lemma}[Lemma 2.2 in~\cite{Razborov_2007}] \label{lemma:double_counting}
    For a given type $\tau$, flags $F_1, \ldots, F_\ell, F \in \mF^\tau$ and integer $k$ satisfying $v(F_1) + \ldots + v(F_\ell) - (\ell-1) \, v(\tau) \leq k \leq v(F)$, we have
    \begin{equation*}
        p(F_1, \ldots, F_\ell; F) = \sum_{F' \in \mF^{\tau}_k} p(F_1, \ldots, F_\ell; F') \, p(F'; F).
    \end{equation*}
    \begin{proof}
        For each $S \in\mathbf S(F'; F)$ let $\iota_S: V(F') \overset{\sim}{\hookrightarrow} S \subseteq V(G)$ be an arbitrary but fixed bijection.
        There is a natural map $\mathbf{S}(F_1,\dots,F_\ell; F') \times \mathbf{S}(F'; F) \to \mathbf{S}(F_1,\dots,F_\ell;F)$ given by mapping $(S_1,\dots,S_\ell)\in \mathbf S(F_1,\dots,F_\ell;F')$ and $S \in \mathbf{S}(F'; F)$ to $(\iota_S(S_1),\dots,\iota_S(S_\ell))$. This map is surjective and the pre-image of every point is of equal cardinality.
        %(namely $\binom{v(F)-m}{k-m}$ for $m=v(F_1) + \ldots + v(F_\ell) - (\ell-1) \, v(\tau)$).
        Since any $(S_1,\dots, S_\ell)\in\mathbf S(F_1,\dots,F_\ell;F)$ satisfies $F[S_i]\simeq F_i$  for all $1\le i\le \ell$ if and only if all elements $(S_1',\dots,S_\ell)'\in \mathbf S(F_1,\dots,F_\ell;F')$ and $S \in \mathbf{S}(F'; F)$ in its pre-image also satisfy $F[S] \simeq F'$ as well as $F'[S_i']\simeq F_i$ for all $1\le i\le \ell$,  the desired statement follows.
    \end{proof}
\end{lemma}

\subsection{The actual flag algebras}

We can now formally define the actual flag algebras. For a given type $\tau$, we consider elements in $\RR \mF^\tau$, that is arbitrary finite linear combinations of elements in $\mF^\tau$, and let
\begin{equation}\label{eq:definition_of_kernel}
    \mK^\tau = \operatorname{span} \left\{ F'  - \sum_{F \in \mF^\tau_{ n}} p(F';F) \, F  \mid F' \in \mF^\tau, \, n \geq v(F') \right\} \subset \RR \mF^\tau
\end{equation}
denote the subspace of $\RR \mF^\tau$ spanned by the kernel of the density equalities given by \cref{lemma:double_counting}.
\begin{definition} \label{def:flagalgebra}
    The \emph{flag algebra $\mA^\tau$ of type $\tau$} is given by furnishing the quotient space $\mathbb R\mF^\tau / \mK^\tau$ with a product through the bilinear extension of $F \cdot F' = \sum_{G \in \mF^\tau_{ n}} p(F, F'; G) \cdot G$, defined for any $F, F' \in \mF^\tau$ and $n = v(F) + v(F') - v(\tau)$.
\end{definition}
Note that by \cref{lemma:double_counting}, the bilinear extension in the definition of the product is well defined and in fact independent of $n$ as long as $n \geq v(F) + v(F') - v(\tau)$. We can also extend our notion of density so that the expression $p(f_1,\dots,f_\ell;F)$ is well-defined as along as $f_1, \ldots, f_\ell$ can be expressed in $\mA^\tau$ as the sum of cosets corresponding to flags of order at most $v(F)$.

% The next remark is only valid when $\mH = \emptyset$. Let $F\in \mF_n^\tau$ and define $\mA_n^\tau$ as the subspace of $\mA^\tau$, generated by all flags $F\in\mF^\tau_n$. Then when $(f_1,\dots,f_\ell)\in (\mA_n^\tau)^{\times \ell},$ the expression $p(f_1,\dots,f_\ell;F)$ is well-defined.

We identify flags $F \in \mF^\tau$ with their coset $F + \mK^\tau$ in $\mA^\tau$ for notational convenience, i.e., we treat $\mF^\tau$ as a subset of $\mA^\tau$, and whenever an element from $\mA^\tau$ does not necessarily correspond to (the coset of) a flag, we will use lower case letters to denote it. Note that $\tau$ is the identity element of the algebra $\mA^\tau$, that is $f \cdot \tau = f$ for any $f \in \mA^\tau$. Also note that for any $f \in \mA^\tau$ there exists some $n_0 = n_0(f) \in \NN$, likewise by \cref{lemma:double_counting}, such that $f$ can be expressed as a finite linear combination of (the cosets of) the elements in $\mF^\tau_{n}$ as long as $n \geq n_0$. More precisely, we have
\begin{equation*}
    f = \sum_{F \in \mF^\tau_{ n}} p(f; F) \, F,
\end{equation*}
for any $n \geq n_0$, since for any coefficients $(c_F)_{F \in \mF^\tau_{ n}}$ satisfying  $f = \sum_{F \in \mF^\tau_{ n}} c_F \, F$, we have
\begin{equation*}
    p(f; F') = p\Big( \sum_{F \in \mF^\tau_{ n}} c_F \, F; F' \Big) =  \sum_{F \in \mF^\tau_{ n}} c_F \, p(F; F') = c_{F'}
\end{equation*}
for any $F' \in \mF^\tau_{ n}$.

The decisive property of flag algebras is that they algebraically encapsulate some of the combinatorial behavior of sequences of flags (and therefore of sequences of colorings). Let $(F_n)_{n \in \NN}$ be a sequence of flags in $\mF^\tau$ of increasing order and assume that $\lim_{n \to \infty} p(F; F_n)$ exists for any $F \in \mF^\tau$. We call such a sequence \emph{convergent} and note that by compactness any sequence of flags of increasing order contains a convergent subsequence. On the other side, let $\operatorname{Hom}^+(\mA^\tau, \RR) = \{\phi \in \operatorname{Hom} (\mA^\tau, \RR) \mid \phi(F) \geq 0 \text{ for all } F \in \mF^\tau \}$ denote the set of \emph{positive homomorphisms}, that is all real algebra homomorphisms mapping the cosets corresponding to flags in $\mF^\tau$ to positive values.
\begin{theorem}[Theorem 3.3 in ~\cite{Razborov_2007}] \label{thm:homfunctional}
    Let a type $\tau$ be given. For any $\phi \in \operatorname{Hom}^+(\mA^\tau, \RR)$ there exists a convergent sequence $(F_n)_{n \in \NN}$ in $\mF^\tau$ satisfying $\lim_{n \to \infty} p(f; F_n) = \phi(f)$ for all $f \in \mA^\tau$ and vice-versa.
\end{theorem}
%
% In fact, $\phi$ is the linear extension of $\lim_{n \to \infty} p(\cdot; G_n)$ to $\mA^\tau$.
This means the goal becomes to identify elements in the \emph{semantic cone} $\mS^\tau = \{f \in \mA^\tau \mid \phi(f) \geq 0 \text{ for all } \phi \in \operatorname{Hom}^+(\mA^\tau, \RR)\}$, as these correspond to combinatorial relations that are {\lq}true{\rq} for any sequence. 
For notational convenience, we will let $\mS^\tau$ define a partial order on $\mA^\tau$ by saying that $f \geq g$ for $f, g \in \mA^\tau$ if and only if $f - g \in \mS^\tau$. 

\begin{example}
    To prove that $\lim_{n \to \infty} p(K_3^1; G_n) + \ldots + p(K_3^4; G_n) \geq 1/256$ holds for any sequence of colorings $G_n \in \mG^{(4)}$ of increasing order, as stated in \cref{thm:fourcolortriangles}, we need to show that the element $f_0 = K_3^1 + K_3^2 + K_3^3 + K_3^4 - 1/256$ is in $\mS^{\varnothing}$ or, using the partial order, $f_0 \geq 1/256$. Here we have identified the elements $K_3^i \in \mG^{(4)}$ first with flags in $\mF^\varnothing$ and then with the corresponding cosets in the flag algebra $\mA^\varnothing$ and we have also used $1/256$ as a short hand notation for $1/256 \cdot \varnothing \in \mA^\varnothing$.  Note that we can also pictorially express $f_0$ as
    \begin{equation*}
        f_0 = \begin{tikzpicture}[line width=0.5mm,baseline=1.0ex, scale=0.5]
            \node[circle,fill=black,draw,scale=0.4] (A) at (0,0) {};
            \node[circle,fill=black,draw, scale=0.4] (B) at (1,0) {};
            \node[circle,fill=black,draw, scale=0.4] (C) at (0.5,0.85) {};
        
            \draw[red, dash dot] (A) -- (B) -- (C) -- (A);
        \end{tikzpicture}
        + 
        \begin{tikzpicture}[line width=0.5mm,baseline=1.0ex, scale=0.5]
            \draw[blue, dotted] (A) -- (B) -- (C) -- (A);
        
            \node[circle,fill=black,draw,scale=0.4] (A) at (0,0) {};
            \node[circle,fill=black,draw, scale=0.4] (B) at (1,0) {};
            \node[circle,fill=black,draw, scale=0.4] (C) at (0.5,0.85) {};
        \end{tikzpicture}
        +
        \begin{tikzpicture}[line width=0.5mm,baseline=1.0ex, scale=0.5]
            \draw[green, dashed] (A) -- (B) -- (C) -- (A);
        
            \node[circle,fill=black,draw,scale=0.4] (A) at (0,0) {};
            \node[circle,fill=black,draw, scale=0.4] (B) at (1,0) {};
            \node[circle,fill=black,draw, scale=0.4] (C) at (0.5,0.85) {};
        \end{tikzpicture}
        +
        \begin{tikzpicture}[line width=0.5mm,baseline=1.0ex, scale=0.5]
            \draw[yellow] (A) -- (B) -- (C) -- (A);
        
            \node[circle,fill=black,draw,scale=0.4] (A) at (0,0) {};
            \node[circle,fill=black,draw, scale=0.4] (B) at (1,0) {};
            \node[circle,fill=black,draw, scale=0.4] (C) at (0.5,0.85) {};
        \end{tikzpicture}
        - 1/256.
    \end{equation*}
\end{example}

Finally, we say that a positive element $f_0 \in \mA^\tau$ is \emph{tight} if there exists a positive homomorphism that is \emph{extremal} for $f_0$, that is some $\phi \in \operatorname{Hom}^+(\mA^\tau, \RR^n)$ satisfying $\phi(f_0) = 0$. By \cref{thm:homfunctional} this is of course equivalent to requiring the existance of a convergent sequence in $\mF^\tau$ that is \emph{extremal} for $f_0$, that is some $(F_n)_{n \in \NN}$ satisfying $\lim_{n \to \infty} p(f_0; F_n) = 0$.

\subsection{Linear maps between algebras}

While the relevant algebra for our and most other applications is $\mA^\varnothing$, i.e., the flag algebra of the empty type, it is often more expressive to construct an argument for some other type $\tau$ and then {\lq}pull{\rq} this argument from the algebra $\mA^\tau$ into the relevant algebra $\mA^\varnothing$.
\begin{definition} \label{def:downward}
    Let some type $\tau$ be given. For any flag $F = (G, \psi) \in \mF^\tau$, let $F \vert_\tau = (G, \varnothing) \in \mF^\varnothing$ denote the flag of the empty type obtained by forgetting the labeling of the vertices of $F$ and
    \begin{equation*}
        q_\tau(F) = \frac{|\{\varphi : [v(\tau)] \hookrightarrow V(G) \mid (G, \varphi) \simeq F \}|}{v(G)! \, / \, (v(G)-v(\tau))!}
    \end{equation*}
    the probability that a randomly chosen injection from $[v(\tau)]$ to $V(G)$ turns $G$ into a flag of type $\tau$ that is isomorphic to $F$. Writing $\llbracket F \rrbracket_\tau = q_\tau(F) \cdot F \vert_\tau \in \mA^\varnothing$, the \emph{downward operator} is now given by the linear extension of $\llbracket \cdot \rrbracket_\tau$ to $\mA^\tau$.
\end{definition}
Note that $\llbracket \cdot \rrbracket_\varnothing = \operatorname{id}$. One key feature of the downward operator, apart from being linear~\cite[Theorem 2.5]{Razborov_2007}, is its \emph{order-preserving} nature, i.e., $\llbracket \mS^\tau \rrbracket_{\tau} \subseteq \mS^\varnothing$ as given by the following lemma.

\begin{lemma}\label{lemma:orderpreserving}
    If $f \in \mA^\tau$ is a positive element, then its image $\llbracket f \rrbracket_{\tau}$ in $\mA^\varnothing$ is as well.
\end{lemma} 
\begin{proof}
    For each $\varepsilon>0$ there exists an integer $n_0 = n_0(f, \varepsilon)$ such that $f$ can be expressed as a linear combination of elements in $\mF^\tau_{n}$ while also satisfying
    $p(f,F)\ge -\varepsilon$ and therefore $p(f+\varepsilon,F)\ge 0$ for  any $F\in \mF^\tau_{n}$ with $n \geq n_0$. 
    Otherwise, there would exist a positive homomorphism $\phi$ by \cref{thm:homfunctional} satisfying $\phi(f)<-\varepsilon$, contradicting the assumption that $f$ is a positive element. We now have
    \[
    \llbracket f\rrbracket_\tau+\varepsilon \cdot q_{\tau}(\tau)
    \ge
    \llbracket f\rrbracket_\tau +\varepsilon \cdot q_{\tau}(\tau) \cdot \tau\vert_\tau 
    =
    \llbracket f+\varepsilon \rrbracket_\tau
    =
    \sum_{F\in\mF^\tau_{n}} p(f+\varepsilon,F) \, \llbracket F\rrbracket_\tau
    \ge 0,
    \]
    where in the first and last step we have used that $0 = 0 \cdot \varnothing \leq \tau\vert_\tau  \leq 1 \cdot \varnothing = 1$, which is in fact true for any flag in $\mF^\varnothing$, and then for the second and third step, respectively, the linearity of the downward operator and the fact that f, and therefore also $f + \varepsilon$, is expressible as a linear combination of elements in $\mF_{n}^\tau$. Since $q_\tau(\tau)$ is a constant, it follows that $\llbracket f\rrbracket_\tau \geq -\varepsilon$ for any $\varepsilon > 0$ and hence $\llbracket f\rrbracket_\tau \geq 0$ as desired.
\end{proof}

Note that Razborov provides an alternative proof of this statement in ~\cite[Theorem 3.1]{Razborov_2007} using the notion of an ensemble of random homomorphisms ~\cite[Definition 10]{Razborov_2007}.
Likewise important is the following Cauchy-schwarz-type statement, though we will only need the implication that the downward operator applied to squares produces positive elements, which is already implied by \cref{lemma:orderpreserving}.

\begin{lemma}[Theorem 3.14 in~\cite{Razborov_2007}] \label{lemma:cauchyschwarz}
    For any type $\tau$ and $f \in \mA^\tau$, we have $\llbracket f^2 \rrbracket_\tau \geq \llbracket f \rrbracket_\tau^2$.
\end{lemma}
% \begin{proof}[Proof sketch.]
%     $\operatorname{Hom}^+(\mA^\tau,\mathbb R)$ can be injectively mappend into $[0,1]^{\mF^\tau}$ by sending a $\phi$ to the point $(\phi(F))_{F\in\mF}$. $\operatorname{Hom}^+(\mA^\tau,\mathbb R)$ therefore inherits the structure of a compact Hausdorff space, so for each $\phi_0\in\operatorname{Hom}^+(\mA^\varnothing,\mathbb R)$ there exists a probability measure $\mu = \mu(\varphi_0)$ on $\operatorname{Hom}^+(\mA^\tau,\mathbb R)$ satisfying
%     \[
%         \int_{\operatorname{Hom}^+(\mA^\tau,\mathbb R)} \psi(f)d\mu(\psi) 
%         =
%         \phi_0(\llbracket f \rrbracket_\tau).
%     \]
%     for all $g\in\mA^\tau$ so that the claim $\phi_0(\llbracket f^2 \rrbracket_\tau) \geq \phi_0(\llbracket f \rrbracket_\tau)^2$ follows from the ordinary Cauchy–Schwarz inequality for ...\todo[color=red]{Clarify what version of CS is used exactly}
% \end{proof}

\subsection{The semidefinite programming method}

One way to establish that $f_0 \geq 0$ for some particular $f_0 \in \mA^\varnothing$
%, for example $f_0 = K_3^1 +  K_3^2 +  K_3^3 +  K_3^4 - 1/256$ 
is to lower bound it by some known positive element. One obvious source of such positive elements, besides the cosets of flags, are sums of squares (SOS). Through \cref{lemma:orderpreserving} or \cref{lemma:cauchyschwarz}, we can also incorporate squares from algebras of other types, so that we are looking to establish an expression of the type
\begin{equation} \label{eq:basic-sos}
    f_0 \geq \sum_{\tau \in \mT} \sum_{g \in \mB^\tau} \llbracket g^2 \rrbracket_\tau
\end{equation}
for some finite set of types $\mT$ and finite sets of algebra elements $\mB^\tau \subset \mA^\tau$.

In practice, we fix the size of environment $N$ in which all expressions must fit, that is $f_0$ must be expressible as a linear combination of elements in $\mF^\varnothing_N$ and the squares $g^2$ must likewise be expressible as a linear combination of elements in $\mF^\tau_N$ for each of the elements $g \in \mB^\tau$, as this implies that $\llbracket g^2 \rrbracket_\tau$ can also be expressed as a linear combination of elements in $\mF^\varnothing_N$. This means we consider types $\mT \subseteq \{ \tau \mid v(\tau) \leq N-2, \, v(\tau) \, \operatorname{mod} N \equiv 2 \}$ and finitely generated sets of algebra elements $\mB^\tau = \operatorname{span} \big( \mF^\tau_{ \ell_\tau } \big)$ where $\ell_\tau = (N - v(\tau)) / 2$. Using the fact that there are at most $|\mF^\tau_{ \ell_\tau }|$ linearly independent expressions in $\mB^\tau$, \cref{eq:basic-sos} therefore turns into
\begin{equation} \label{eq:sos_certificate}
    f_0 = \sum_{\tau \in \mT} \sum_{i = 1}^{|\mF^\tau_{ \ell_\tau }|} \left\llbracket \left( \sum_{F \in \mF^\tau_{ \ell_\tau }} c^\tau_{i, F} \cdot F \right)^2 \right\rrbracket_\tau + \sum_{F \in \mF^\varnothing_N} s_F \cdot F
\end{equation}
for some coefficients $c^\tau_{i, F}$ for $\tau \in \mT$, $1 \leq i \leq \ell_\tau$ and $F\in\mF^\tau_{\ell_\tau}$ as well as $s_F \geq 0$ for all $F \in \mF^\varnothing_N$. We refer to the collection of variables
\begin{equation*}
    N, \, \mT, \, (\{c_{i, F}^\tau \mid 1 \leq i \leq |\mF^\tau_{\ell_\tau}|, F \in \mF^\tau_{\ell_\tau} \} )_{\tau \in \mT}, \, \{ s_F \mid F \in \mF_N^\varnothing \}
\end{equation*}
satisfying \cref{eq:sos_certificate} as an \emph{(SOS) certificate} of the positivity of $f_0$ and to the coefficients $s_F$ in particular as \emph{slack} as they are responsible for turning the inequality in \cref{eq:basic-sos} into an equality. We say a flag $F \in \mF_N^\varnothing$ is \emph{sharp} in a given certificate if $s_F = 0$. The following lemma captures a property of cerficates that will be fundamental in order to establish stability.
\begin{lemma}
    If a flag $F \in \mF_N^\varnothing$ is not sharp in a given certificate of the positivity of some tight $f_0$, then $\phi (F) = 0$ for any $\phi \in \operatorname{Hom}^+(\mA^\varnothing, \RR^n)$ that is extremal for $f_0$.
\end{lemma}
\begin{proof}
    Since $\phi$ is extremal for $f_0$ as well as linear and since both $F$ and the squares in \cref{eq:sos_certificate} are positive elements, we have
    \begin{align*}
        0 = \phi(f_0) =  \sum_{\tau \in \mT}\sum_{i = 1}^{|\mF^\tau_{ \ell_\tau }|}  \phi \left(  \left\llbracket \left( \sum_{F \in \mF^\tau_{ \ell_\tau }} c^\tau_{i, F} \cdot F \right)^2 \right\rrbracket_\tau \right) + \sum_{F \in \mF^\varnothing_N} s_F \cdot \phi \left( F \right) \geq 0,
    \end{align*}
    so that whenever $s_F > 0$ for some $F \in \mF^\varnothing_N$ we must have $\phi(F) = 0$.
\end{proof}

\begin{example}\label{ex:Goodman_SOS} Goodman's computation of the two color Ramsey multiplicity of triangles has the following compact SOS certificate, in which all the slack variables are $0$:
    \[
        \begin{tikzpicture}[line width=0.5mm, baseline=1ex,
        scale = 0.5]
            \draw [red] (0,0) -- (1,0);
            \draw [red] (0,0) -- (0.5,0.75);
            \draw [red] (0.5,0.75) -- (1,0);
            \fill [black] (0,0) circle (4pt);
            \fill [black] (1,0) circle (4pt);
            \fill [black] (0.5,0.75) circle (4pt);
        \end{tikzpicture}
        +   
        \begin{tikzpicture}[line width=0.5mm, baseline=1ex,
        scale = 0.5]
            \draw [blue, dashed] (0,0) -- (1,0);
            \draw [blue, dashed] (0,0) -- (0.5,0.75);
            \draw [blue, dashed] (0.5,0.75) -- (1,0);
            \fill [black] (0,0) circle (4pt);
            \fill [black] (1,0) circle (4pt);
            \fill [black] (0.5,0.75) circle (4pt);
        \end{tikzpicture}
        -\frac{1}{4}\cdot \varnothing
        =\Bigg \llbracket \Bigg(\frac{\sqrt{3}}{2} \,
        \begin{tikzpicture}[line width=0.5mm, baseline=1ex,
        scale = 0.5]
            \draw [red] (0,0) -- (0,1);
            \fill [black] (0,0) circle (4pt);
            \fill [black] (0,1) circle (4pt);
            \node at (0.30,-0.30) {\emph{1}};
        \end{tikzpicture}
        -\frac{\sqrt{3}}{2} \, 
        \begin{tikzpicture}[line width=0.5mm, baseline=1ex,
        scale = 0.5]
            \draw [blue, dashed] (0,0) -- (0,1);
            \fill [black] (0,0) circle (4pt);
            \fill [black] (0,1) circle (4pt);
            \node at (0.30,-0.30) {\emph{1}};
        \end{tikzpicture}
        \Bigg)^2\Bigg
        \rrbracket 
            \begin{tikzpicture}[line width=0.5mm, baseline=1ex,
        scale = 0.5]
            \fill [black] (0,0) circle (4pt);
            \node at (0.3,-0.3) {\emph{1}};
            \end{tikzpicture}.
\]

\end{example}

There is a well known connection between SOS optimization and semidefinite programming, see for example Gatermann and Parrilo~\cite{gatermann2004symmetry}. Define the \emph{data matrices}
\begin{equation*}\label{eq:density_matrix}
    D_F^\tau = \left( p \big( \left\llbracket F_1 \cdot F_2 \right\rrbracket_\tau; F \big) \right) \in \QQ^{\mF^\tau_{\ell_\tau} \times \mF^\tau_{\ell_\tau}}
\end{equation*}
for any $\tau \in \mT$ and $F \in \mF^\varnothing_N$ and write $\langle \cdot, \cdot \rangle$ for the Frobenius inner product of two matrices. $N, \, \mT, \, (\{c_{i, F}^\tau \mid 1 \leq i \leq | \mF^\tau_{\ell_\tau}|, F \in \mF^\tau_{\ell_\tau} \} )_{\tau \in \mT}, \, \{ s_F \mid F \in \mF_N^\varnothing \}$ now is an SOS certificate if and only if the positive semidefinite matrices
\begin{equation*}
    \QQ^{\mF^\tau_{\ell_\tau} \times \mF^\tau_{\ell_\tau}} \ni Q^\tau = (C^\tau)^\top \cdot C^\tau  \succeq 0,
\end{equation*}
where $C^\tau \in \QQ^{\ell_\tau \times \mF^\tau_{\ell_\tau}}$ is the matrix with element $c_{i,F}^\tau$ in row $i$ and column $F$, satisfy
\begin{equation} \label{eq:sdp_certificate}
    p(f_0;F)  = \sum_{\tau \in \mT} \left\langle Q^\tau, D_F^\tau \right\rangle + s_F,
\end{equation}
for all $F \in \mF^\varnothing_N$. We will refer to
\begin{equation*}
    N, \, \mT, \, (Q^\tau )_{\tau \in \mT}, \, \{ s_F \mid F \in \mF_N^\varnothing \}
\end{equation*}
satisfying \cref{eq:sdp_certificate} as an \emph{(SDP) certificate} of the positivity of $f_0$. This is the formulation in which we will derive our improvements. Let us remark that we will refer to $Q^\tau$ and $D_F^\tau$ as \emph{blocks}, since in the context of semidefinite programming one would consider them part of larger block-diagonal matrices $Q \succeq 0$ and $D_F$, so that \cref{eq:sdp_certificate} simply becomes $p(f;F)  = \left\langle Q, D_F \right\rangle + s_F$ for all $F \in \mF^\varnothing_N$.

\begin{example}\label{ex:Goodman_SDP}
    Mirroring \cref{ex:Goodman_SDP}, the SDP formulation of Goodman's problem with $N=3$ has data matrices
    \begin{align*}
        D^{\bullet_1}_{\begin{tikzpicture}[line width=0.3mm, baseline=0ex,
        scale = 0.4]
            \draw [red] (0,0) -- (1,0);
            \draw [red] (0,0) -- (0.5,0.75);
            \draw [red] (0.5,0.75) -- (1,0);
            \fill [black] (0,0) circle (4pt);
            \fill [black] (1,0) circle (4pt);
            \fill [black] (0.5,0.75) circle (4pt);
        \end{tikzpicture}}
        =
        \begin{pmatrix}
            1 & 0 \\ 0 & 0
        \end{pmatrix}, \quad
        D^{\bullet_1}_{\begin{tikzpicture}[line width=0.3mm, baseline=0ex,
        scale = 0.4]
            \draw [red] (0,0) -- (1,0);
            \draw [red] (0,0) -- (0.5,0.75);
            \draw [blue, dashed] (0.5,0.75) -- (1,0);
            \fill [black] (0,0) circle (4pt);
            \fill [black] (1,0) circle (4pt);
            \fill [black] (0.5,0.75) circle (4pt);
        \end{tikzpicture}}
        =
        \begin{pmatrix}
            \frac{1}{3} & \frac{1}{3} \\ \frac{1}{3} & 0
        \end{pmatrix}, \quad
        D^{\bullet_1}_{\begin{tikzpicture}[line width=0.3mm, baseline=0ex,
        scale = 0.4]
            \draw [red] (0,0) -- (1,0);
            \draw [blue, dashed] (0,0) -- (0.5,0.75);
            \draw [blue, dashed] (0.5,0.75) -- (1,0);
            \fill [black] (0,0) circle (4pt);
            \fill [black] (1,0) circle (4pt);
            \fill [black] (0.5,0.75) circle (4pt);
        \end{tikzpicture}}
        =
        \begin{pmatrix}
            0 & \frac{1}{3} \\ \frac{1}{3} & \frac{1}{3} 
        \end{pmatrix}, and \quad
                D^{\bullet_1}_{\begin{tikzpicture}[line width=0.3mm, baseline=0ex,
        scale = 0.4]
            \draw [blue, dashed] (0,0) -- (1,0);
            \draw [blue, dashed] (0,0) -- (0.5,0.75);
            \draw [blue, dashed] (0.5,0.75) -- (1,0);
            \fill [black] (0,0) circle (4pt);
            \fill [black] (1,0) circle (4pt);
            \fill [black] (0.5,0.75) circle (4pt);
        \end{tikzpicture}}
        =
        \begin{pmatrix}
            0 & 0 \\ 0 & 1 
        \end{pmatrix},
    \end{align*}
    where the flags are sorted with \begin{tikzpicture}[line width=0.3mm, baseline=0ex, scale = 0.4]
            \draw [red] (0,0) -- (0,1);
            \fill [black] (0,0) circle (4pt);
            \fill [black] (0,1) circle (4pt);
            \node at (0.40,-0.30) {\emph{1}};
        \end{tikzpicture}  coming first and then \begin{tikzpicture}[line width=0.3mm, baseline=0ex, scale = 0.4]
            \draw [blue, dashed] (0,0) -- (0,1);
            \fill [black] (0,0) circle (4pt);
            \fill [black] (0,1) circle (4pt);
            \node at (0.40,-0.30) {\emph{1}};
        \end{tikzpicture}. The certificate for Goodman's theorem is now completed through the semidefinite matrix
    \[
        Q^{\bullet_1} = \begin{pmatrix} \frac{3}{4} & -\frac{3}{4} \\ -\frac{3}{4} & \frac{3}{4}\end{pmatrix}
    \]
    and by setting all slack variables to $0$.
\end{example}

\subsection{Forbidden subcolorings}

All definitions and statements made in this section can also be conditioned on any particular  set $\mH\subseteq\mF^\tau$ of flags not appearing as an induced subflag. For this, we denote by $\mF^\tau_\mH = \{ F \in \mF^\tau \mid p(H; F)=0 \}$ the collection of all flags $F$ satisfying not containing any $H\in\mH$. Note that for every $F\in\mF^\tau_\mH$ and $S\subseteq V(F)$, the flag $F[S]$ is again in $\mF^\tau_\mH$. This defines an algebra $\mA_\mH$ as before. We will omit $\mH$ from notation whenever it is the empty set. This will in fact be the case for the vast majority of the subsequent sections, with the only exception being the proof of \cref{th:extremals_are_Ramsey} where it be convenient to briefly switch the an algebra with a particular set of forbidden subflags.

% \todo[inline]{Add proof / reference / intuition for statement $\mA_\mH^\varnothing \simeq \mA^\varnothing / \langle \mH \rangle $}

\section{Using symmetries in flag algebras} \label{sec:improvements}

This section explores the use of problem specific symmetries to 
reduce the overall size 
of the SDP given in \cref{eq:sdp_certificate}. In \cref{eq:generalized_isomorphism}, we first define a generalized notion of isomorphism that will be required for both methods. In \cref{sec:color_invariance}, we use color permutations to reduce the number of constraints and blocks. Afterwards, in \cref{sec:block_diagonalization}, we apply the theory of block-diagonalization to break down the remaining blocks into several smaller ones, therefore reducing the overall number of variables. These methods were crucial for deriving \cref{eq:fourcolortriangles}.

\subsection{Generalized flag isomorphisms}\label{eq:generalized_isomorphism}

We begin by defining a class of homomorphisms between two flag algebras $\mA^\tau$ and $\mA^\sigma$ in the particular case where $\sigma$ is derived from $\tau$ through color and vertex permutations. To be more precise, let $\tau$ be a given type and $(\alpha, \beta)$ a pair consisting of a color permutation $\alpha: [c] \overset{\sim}{\hookrightarrow} [c]$ and a permutation of the type indices $\beta: [v(\tau)] \overset{\sim}{\hookrightarrow} [v(\tau)]$. For every $\tau$-flag $F = (G, \psi)$, we now denote by $\pi^{\alpha,\beta}(F) = (\alpha \circ F, \psi \circ \beta)$ the flag obtained by permuting the colors and type indices according to $(\alpha, \beta)$. Note that the flag $\pi^{\alpha,\beta}(F)$ is of type $\pi^{\alpha,\beta}(\tau)$, so that $\pi^{\alpha,\beta}$ defines a bijection from $\mF^\tau$ to $\mF^{\pi^{\alpha,\beta}(\tau)}$.
The following two statements show that we can in fact extend $\pi^{\alpha, \beta}$ to form an isomorphism between the respective algebras that commutes with our notion of density as well as the downward operator in a natural way.

\begin{lemma}\label{lemma:extend_color_vertex_permutations}
    The linear extension of $\pi^{\alpha, \beta}$ induces an algebra homomorphism from $\mA^\tau$ to $\mA^{\pi^{\alpha,\beta}(\tau)}$ that satisfies
    \begin{equation} \label{eq:product_commute}
        \pi^{(\alpha,\beta)}(f_1 \cdot f_2) = \pi^{(\alpha,\beta)}(f_1) \cdot \pi^{(\alpha,\beta)}(f_2)
    \end{equation}
    for any $f_1, f_2 \in \mA^\tau$ as well as
    \begin{equation} \label{eq:density_permutation}
        p \big( \pi^{(\alpha,\beta)}(f_1), \ldots, \pi^{(\alpha,\beta)}(f_\ell); \pi^{(\alpha,\beta)}(F) \big) = p(f_1, \ldots, f_\ell; F).
    \end{equation}
    for any $f_1, \ldots, f_\ell \in \mA^\tau$ and $F \in \mF^\tau$.
\end{lemma}
\begin{proof}
    We start by noting that \cref{eq:density_permutation} holds for \emph{flags} $F_1, \ldots, F_\ell \in \mF^\tau$,
    since for any $i\in \{1,\dots,\ell\}$ and $(S_1, \ldots, S_\ell) \in {\bf S}(F_1, \ldots, F_\ell; F)$ we have $\pi^{\alpha,\beta}(F)\vert_{S_i} \simeq \pi^{\alpha,\beta}(F_i)$ if and only if $F\vert_{S_i} \simeq F_i$. 
    Using this, it immediately follows that $\pi^{\alpha,\beta}(\mK^\tau) \subseteq \mK^{\pi^{\alpha,\beta}(\tau)}$ and that $\pi^{\alpha,\beta}$ likewise commutes with the multiplication of flags, since for $F_1, F_2\in\mF^\tau$ we have
    \begin{align*}
        \pi^{\alpha,\beta}(F_1\cdot F_2) & = \sum_{H\in\mF_n^\tau} p(F_1,F_2;H) \, \pi^{\alpha,\beta}(H) \\
        & = \sum_{H\in\mF_n^\tau} p(\pi^{\alpha,\beta}(F_1),\pi^{\alpha,\beta}(F_2);\pi^{\alpha,\beta}(H)) \, \pi^{\alpha,\beta}(H) \\
        & = \sum_{H \in \mF_n^{\pi^{\alpha, \beta}(\tau)}}  p(\pi^{\alpha,\beta}(F_1),\pi^{\alpha,\beta}(F_2);H) \, H \\
        & = \pi^{\alpha,\beta}(F_1)\cdot\pi^{\alpha,\beta}(F_2)
    \end{align*}
    for some large enough $n$. It follows that the linear extension of $\pi^{\alpha,\beta}$ induces an algebra homomorphism from $\mA^\tau$ to $\mA^{\pi^{\alpha,\beta}(\tau)}$ and that \cref{eq:product_commute} as well as \cref{eq:density_permutation} hold for arbitrary algebra elements by linearity.
\end{proof}
\begin{lemma}\label{lemma:harmonic_downward}
    We have $\llbracket \pi^{\alpha,\beta}(f) \rrbracket_{\pi^{\alpha,\beta}(\tau)} = \pi^{\alpha,\varnothing}(\llbracket f \rrbracket_{\tau})$ for any $f \in\mF^\tau$.
\end{lemma}
\begin{proof}
    Let us first show that the statement is true for any flag $F = (G,\psi) \in \mF^\tau$. We note that $\pi^{\alpha,\beta}(F)\vert_{\pi^{\alpha,\beta}(\tau)} = (\alpha\circ G,\varnothing) = \pi^{\alpha,\varnothing}(F\vert_\tau)$. For the coefficient, we likewise have
    \begin{eqnarray*}
        q_{\pi^{\alpha,\beta}(\tau)}(\pi^{\alpha,\beta}(F))
        &=&
        \frac{|\{\varphi : [v(\tau)] \hookrightarrow V(G) \mid (\alpha\circ G, \varphi\circ \beta) \simeq (\alpha\circ G,\psi\circ \beta) \}|}{v(G)! \, / \, (v(G)-v(\tau))!},
    \end{eqnarray*}
    since $(\alpha\circ G,\varphi\circ \beta) \simeq (\alpha\circ G,\psi\circ \beta)$ if and only if $(G,\varphi) \simeq F$.
    It follows that 
    \begin{align*}
        \llbracket \pi^{\alpha,\beta}(F) \rrbracket_{\pi^{\alpha,\beta}(\tau)} & = q_{\pi^{\alpha,\beta}(\tau)}(\pi^{\alpha,\beta}(F)) \cdot \pi^{\alpha,\beta}(F)\vert_{\pi^{\alpha,\beta}(\tau)} \\
        %
        % & = q_\tau(F) \cdot (\alpha \circ G, \varnothing) \\
        %
        & = q_\tau(F) \cdot \pi^{\alpha,\varnothing}( F \vert_\tau )
        = \pi^{\alpha, \varnothing}\big( q_\tau(F) \cdot F \vert_\tau \big) \\
        & = \pi^{\alpha,\varnothing}(\llbracket F \rrbracket_{\tau}).
    \end{align*}
    The statement now follows for all elements in $\mA^\tau$ by linearity. 
\end{proof}

Let us note that we are essentially stacking another notion of isomorphism on top of our previously defined notion of flag isomorphism by now taking permutations of the type index sets as well as permutations of colors into account. The former has always implicitly been considered in flag algebra certificates by not adding two types that are isomorphic when considered as graphs (even though they are non-isomorphic when considered as flags) to the set $\mT$.  For the latter we need to be more nuanced: the element $f_0 \in \mA^\varnothing$, for which we want to establish positivity, may not be invariant under all permutations of colors, motivating the following definition.
\begin{definition}
    Given some $f_0\in\mA^\varnothing$, we let 
    \begin{equation*}
        \Upsilon = \Upsilon(f_0) = \{ \alpha:[c]\to[c] \mid \pi^{\alpha,\varnothing}(f_0) = f_0 \}
    \end{equation*}
    denote the group of all \emph{admissible color permutations} for $f_0$. Given some type $\tau$ and flags $F_1, F_2 \in \mF^\tau$, we say they are \emph{generalized isomorphic} with respect to $f_0$, denoted by $F_1 \simeq_g F_2$, if there exists $\alpha \in \Upsilon$ and $\beta: [v(\tau)] \overset{\sim}{\hookrightarrow} [v(\tau)]$ satisfying $\pi^{(\alpha,\beta)}(F_1) \simeq F_2$. For each flag $F$, we will let $\mu(F)$ denote an arbitrary but canonical representative of the equivalence class induced by $\simeq_g$.  We also let
    \begin{equation*}
        \Lambda_F =  \Lambda_F(f_0) = \{ (\alpha,\beta) \mid \alpha \in \Upsilon, \, \beta: [v(\tau)]\overset{\sim}{\hookrightarrow} [v(\tau)], \, \pi^{(\alpha,\beta)}(F) \simeq F \}
    \end{equation*}
    denote the \emph{generalized automorphism group} of a given flag $F \in \mF^\tau$ with respect to $f_0$.
\end{definition}

%
% Later on in \cref{def:color_invariant_automorphisms_types}, we will consider the automorphisms of a type $\tau$ relative to the automorphisms $\Upsilon$ of an element $f_0\in\mA^\varnothing$.\todo[color=red]{Das muss weg}

% Suppose now we are given some algebra element $f_0\in\mA^\tau$ for which we want to establish  positivity. Let us establish some relevant notation for the group of color permutations under which $f_0$ is invariant, since these are the symmetries that we can take advantage of when looking for a certificate.
%

%
% 
\begin{example}
    In the case of $f_0 = K_3^1 + K_3^2 + K_3^3 + K_3^4 - 1/256$, we have $\Upsilon = S_{[4]}$.
\end{example}

% \begin{itemize}
%     \item Define admissible color permutations
%     \item Generalized isomorphism notion for flags and types
%     \item Define canonical representative $\mu$ for flags and types
%     \item Define matrix permutation notation
%     \item already remark on how this is used to reduce types when ignoring color permutations
% \end{itemize}

\subsection{Using color-invariance}\label{sec:color_invariance}

%For the particular case of  \cref{eq:fourcolortriangles}.

We can average both sides of a given certificate in \cref{eq:sos_certificate}  over all $\alpha \in \Upsilon$ by applying $\pi^{(\alpha,\varnothing)}$ and invoking \cref{lemma:extend_color_vertex_permutations} as well as \cref{lemma:harmonic_downward}. On the left hand side we again get $f_0$, but on the right hand side we obtain a new certificate of positivity for $f_0$ that is now invariant under the map $\pi^{(\alpha,\varnothing)}$ for any $\alpha \in \Upsilon$. The idea is to directly search for such an invariant certificate and therefore reducing the size of the problem formulation.
 In order to formally state this result in a form corresponding to  \cref{eq:sdp_certificate}, let us write 
\begin{equation*}
    \hat{D}_{F }^{\tau } = \frac{1}{|\Upsilon|}\sum_{\alpha \in \Upsilon} D^\tau_{\pi^{(\alpha,\varnothing)} (F)}
\end{equation*}
for the \emph{$\Upsilon$-averaged data matrices}. We will also need some additional notation to correctly orient the index sets of matrices.
Given any $\mF^\tau_N \times \mF^\tau_N$ matrix $M$, we let $M^{(\alpha, \beta)}$ denote the $\mF^{\pi^{(\alpha,\beta)}(\tau)}_N \times \mF^{\pi^{(\alpha,\beta)}(\tau)}_N$ matrix satisfying
\begin{equation*}
    M^{(\alpha, \beta)}_{F_1, F_2} = M_{\pi^{(\alpha^{-1},\beta^{-1})}(F_1),\pi^{(\alpha^{-1},\beta^{-1})}(F_2)}
\end{equation*}
for any $F_1, F_2 \in \mF^{\pi^{(\alpha,\beta)}(\tau)}_N$. Note that clearly  $\langle L^{(\alpha,\beta)}, M^{(\alpha,\beta)}\rangle = \langle L, M\rangle$ as well as $(L + M)^{(\alpha,\beta)} = L^{(\alpha,\beta)} + M^{(\alpha,\beta)}$ for any two $\mF^\tau_N \times \mF^\tau_N$ matrices $L$ and $M$.

\begin{lemma}\label{lem:averaged_are_invariant}
    Given any type $\tau$, $\alpha \in \Upsilon$, $\beta: [v(\tau)] \to [v(\tau)]$, and $F \in \mF^\varnothing$, we have
    \begin{align*}
        (D_F^\tau)^{(\alpha,\beta)} = (D_{\pi^{(\alpha,\varnothing)}(F)}^{\pi^{(\alpha,\beta)}(\tau)}) \quad \text{as well as} \quad (\hat D^\tau_{F})^{(\alpha,\beta)} = \hat D_F^{\pi^{(\alpha,\beta)}(\tau)}.
    \end{align*}
\end{lemma}
\begin{proof}
    Let $F_1, F_2 \in \mF^{\pi^{(\alpha,\beta)}(\tau)}_N$. For the first equality, we have
    \begin{align*}
        (D_F^\tau)^{(\alpha, \beta)}_{F_1,F_2} &=
        (D_F^\tau)_{\pi^{(\alpha^{-1}, \beta^{-1})} (F_1), \pi^{(\alpha^{-1}, \beta^{-1})} (F_2)} \\
        & = p \left( \big\llbracket \pi^{(\alpha^{-1}, \beta^{-1})} (F_1) \cdot \pi^{(\alpha^{-1}, \beta^{-1})}(F_2) \big\rrbracket_{\tau} ; F \right) \\
        & = p \left( \pi^{(\alpha, \varnothing)} \left( \big\llbracket \pi^{(\alpha^{-1}, \beta^{-1})} (F_1 \cdot F_2) \big\rrbracket_{\tau} \right); \pi^{(\alpha, \varnothing)} (F) \right) \\
        & = p(\left\llbracket F_1 \cdot F_2 \right\rrbracket_{\pi^{(\alpha, \beta)}(\tau)}; \pi^{(\alpha, \varnothing)} (F)) \\
        & = \big( D_{\pi^{(\alpha,\varnothing)}(F)}^{\pi^{(\alpha,\beta)}(\tau)} \big)_{F_1,F_2}.
    \end{align*}
    For the second equality, we have
    \begin{align*}
        (\hat D^\tau_{F})^{(\alpha,\beta)}_{F_1,F_2} 
        & = \frac{1}{|\Upsilon|}\sum_{\alpha' \in \Upsilon} \left( D^\tau_{\pi^{(\alpha',\varnothing)} (F)} \right)^{(\alpha, \beta)}_{F_1, F_2} \\
        & = \frac{1}{|\Upsilon|}\sum_{\alpha' \in \Upsilon} p(\llbracket \pi^{(\alpha^{-1}, \beta^{-1})}(F_1 \cdot F_2) \rrbracket_\tau; \pi^{(\alpha', \varnothing)}(F)) \\
        & = \frac{1}{|\Upsilon|}\sum_{\alpha' \in \Upsilon}  p(\pi^{(\alpha^{-1}, \varnothing)} (\llbracket (F_1 \cdot F_2) \rrbracket_{\pi^{(\alpha, \beta)} (\tau)} ); \pi^{(\alpha', \varnothing)}(F)) \\
        & = \frac{1}{|\Upsilon|}\sum_{\alpha' \in \Upsilon} p(\left\llbracket F_1 \cdot F_2 \right\rrbracket_{\pi^{(\alpha,\beta)}(\tau)}; \pi^{(\alpha', \varnothing)}(F)) \\
        & = \frac{1}{|\Upsilon|}\sum_{\alpha' \in \Upsilon} \left( D^{\pi^{(\alpha,\beta)}(\tau)}_{\pi^{(\alpha', \varnothing)}(F)} \right)_{F_1, F_2} \\
        & = \big( \hat D_F^{\pi^{(\alpha,\beta)}(\tau)} \big)_{F_1, F_2}.
    \end{align*}
    Here we have made use of \cref{lemma:extend_color_vertex_permutations} as well as \cref{lemma:harmonic_downward} for both arguments.
\end{proof}

\begin{proposition}\label{lem:one_direction_of_equiv}
    Assume we are given matrices $Q^\tau\succeq 0$ that satisfy
    \begin{equation} \label{eq:reduced-sdp-0}
        p(f_0; F) = \sum_{\tau \in \mT}\left\langle Q^{\tau}, \hat{D}_{F }^{\tau } \right\rangle + s_{ F }
    \end{equation}
    for all $F \in \mu(\mF_n^\varnothing)$. With
    \begin{equation*}
        \tilde{\mT} = \{ \tilde \tau \mid \tilde{\tau} = \pi^{(\alpha,\beta)}(\tau) \text{ for some } \tau\in\mT, \alpha:[c]\to [c], \beta:[v(\tau)]\to [v(\tau)] \} \supseteq \mT
    \end{equation*}
    and
    \begin{align*}
        \tilde Q^{\tilde \tau} &= \sum_{\tau\in \mT} \sum_{\substack{\alpha\in\Upsilon \\ \beta:[v(\tau)]\to [v(\tau)] \\ \pi^{(\alpha,\beta)}(\tau) = \tilde\tau}} 
        \frac{1}{|v(\tau)|!|\Upsilon|}  \, (Q^{\tau})^{(\alpha,\beta)} \succeq 0
    \end{align*}
    for all $\tilde{\tau} \in \tilde{\mT}$, we have that $N, \, \tilde{\mT}, \, (\tilde{Q}^{\tilde{\tau}} )_{\tilde {\tau} \in \tilde{\mT}}, \, \{ s_{\mu(F)} \mid F \in \mF_N^\varnothing \}$ is a certificate of positivity for $f_0$.
\end{proposition}
\begin{proof}
    Note that the $Q_{\tilde{\tau}}$ are positive semidefinite as a positive combination of positive semidefinite matrices. Given any $F \in \mF_N^\varnothing$, we have
    \begin{align*}
        p(f_0;F) &= p(f_0;\mu(F)) \\
        &= \sum_{\tau \in \mT}\left\langle Q^{\tau}, \hat{D}_{\mu(F) }^{\tau } \right\rangle + s_{ \mu(F) } \\
        &= \sum_{\tau \in \mT}\left\langle Q^{\tau}, \frac{1}{|\Upsilon|}\sum_{\alpha\in\Upsilon}D_{\pi^{(\alpha^{-1},\varnothing)}(F)}^{\tau } \right\rangle + s_{ \mu(F) } \\
        &= \sum_{\tau \in \mT}\sum_{\alpha\in\Upsilon}
        \sum_{\beta:[v(\tau)]\to [v(\tau)]}
        \left\langle \frac{1}{|v(\tau)|!|\Upsilon|} (Q^{\tau})^{(\alpha,\beta)}, D_{F}^{\pi^{(\alpha,\beta)}(\tau)} \right\rangle + s_{ \mu(F) } \\
       &= \sum_{\tilde\tau\in\tilde{\mT}} \left\langle \sum_{\tau\in \mT} \sum_{\substack{\alpha \in \Upsilon \\ \beta:[v(\tau)]\to [v(\tau)] \\ \pi^{(\alpha,\beta)}(\tau) = \tilde\tau}}
        \frac{1}{|v(\tau)|!|\Upsilon|} (Q^{\tau})^{(\alpha,\beta)}, D_{F}^{\tilde\tau} \right\rangle + s_{ \mu(F) } \\
        &= \sum_{\tilde\tau\in\tilde{\mT}} \left\langle \tilde{Q}^{\tilde{\tau}}, D_{F}^{\tilde\tau} \right\rangle + s_{ \mu(F) },
    \end{align*}
    establishing that $N, \, \tilde{\mT}, \, (\tilde{Q}^{\tilde{\tau}} )_{\tilde {\tau} \in \tilde{\mT}}, \, \{ s_{\mu(F)} \mid F \in \mF_N^\varnothing \}$ is in fact an SDP certificate of positivity for $f_0$.
\end{proof}

\cref{lem:one_direction_of_equiv} suffices for our applications, because it tells us that any solution of  \cref{eq:reduced-sdp-1} gives us a valid certificate of positiveness for $f_0$. For completeness, we include a simple proof that a certificate of positivity for the SDP from \cref{eq:reduced-sdp-0} produces a certificate of positivity for \cref{eq:sdp_certificate}. In order to state it, let $\eta(\tau)$ an arbitrary but fixed permutation pair $(\alpha,\beta)$ with$\alpha\in\Upsilon$ and satisfying $\pi^{\eta(\tau)}(\tau) = \mu(\tau)$ for a given $\tau\in \mT$.

\begin{proposition}
    \label{thm:color_invariant_sdp}
    If $N, \, \mT, \, (Q^\tau )_{\tau \in \mT}, \, \{ s_F \mid F \in \mF_N^\varnothing \}$ is a certificate of positivity for $f_0$, then
    \[
        \hat{s}_{ F } = \frac{1}{|\Upsilon|}\sum_{\alpha \in \Upsilon} s_{\pi^{(\alpha,\varnothing)}(F)}
    \quad \text{and} \quad
        \hat Q^{\hat \tau} = \sum_{\substack{\tau \in \mT \\ \mu(\tau) = \hat\tau}} (Q^\tau)^{\eta(\tau)}
    \]
    is an SDP certificate of the averaged SDP, that is for all $F \in \mu (\mF_n^\varnothing)$
    \begin{equation} \label{eq:reduced-sdp-1}
        p(f_0; F) = \sum_{\tau \in \mu(\mT)}\left\langle \hat{Q}^{\tau}, \hat{D}_{F }^{\tau } \right\rangle + \hat{s}_{ F }.%  
    \end{equation}
\end{proposition}
\begin{proof}
By assumption, for all  $F \in \mF_n^\varnothing$
     \begin{equation*}
         p(f_0; F) = \sum_{\tau \in \mT}\left\langle Q^{\tau},   D_{F }^{\tau } \right\rangle + s_{ F }.
    \end{equation*}
    Then, in particular
     \begin{align*}
          \sum_{\alpha \in \Upsilon} p(f_0; \pi^{(\alpha,\varnothing)} (F)) 
          &= \sum_{\alpha \in \Upsilon} \left(\sum_{\tau \in \mT}\left\langle Q^{\tau},   D_{\pi^{(\alpha,\varnothing)} (F) }^{\tau } \right\rangle + s_{ \pi^{(\alpha,\varnothing)} (F) } \right) \\
        |\Upsilon| p(f_0; F) &= \sum_{\tau \in \mT}\left\langle Q^{\tau}, \sum_{\alpha \in \Upsilon}  D_{\pi^{(\alpha,\varnothing)} (F) }^{\tau } \right\rangle + \sum_{\alpha \in \Upsilon} s_{ \pi^{(\alpha,\varnothing)} (F) } \\
        p(f_0; F) &= \sum_{\tau \in \mT}\left\langle Q^{\tau}, \hat{D}_{F }^{\tau } \right\rangle + \hat{s}_{ F } \\
        &= \sum_{\tau \in \mT}\left\langle Q^{\eta(\tau)}, (\hat{D}_{F }^{\tau})^{\eta(\tau)} \right\rangle + \hat{s}_{ F } \\
        &= \sum_{\sigma\in \mu(\mT)}\left\langle 
        \sum_{\substack{\tau \in \mT \\ \mu(\tau) = \sigma}} (Q^\tau)^{\eta(\tau)}, \hat{D}_{F }^{\sigma} \right\rangle + \hat{s}_{ F },
    \end{align*}
    establishing the claim.
\end{proof}

The SDP from \cref{eq:reduced-sdp-0} is relevant because $\mu(\mF_N^\varnothing)$ is (possibly significantly) smaller than $\mF_N^\varnothing$, so we have reduced the number of constraints, and $\mu(\mT)$ is smaller than $\mT$, so we have reduced the number of blocks. As previously mentioned, it is notable that implicitly this was previously already done for vertex permutations of types, i.e., one does not reuse types that are graph isomorphic even though they are not flag isomorphic. \cref{lem:one_direction_of_equiv} is the natural generalization of this for colors, where we now need to take the invariance of $f_0$ w.r.t. color permutations into account (whereas the invariance of $f_0$ with respect to graph isomorphisms is a given).

\begin{example}\label{ex:Goodman_SDP_averaged}
    Continuing \cref{ex:Goodman_SDP}, we now use black edges to denote one color class and white ones to denote the other, so as to emphasize the symmetry between them. We have $\Upsilon = \{(\operatorname{id},\varnothing), (0\mapsto 1, 1\mapsto 0,\varnothing)\}$ and we choose $\mu$ so that say $\mu(\mF_3^\varnothing) = \{\begin{tikzpicture}[line width=0.3mm, baseline=0.3ex,
        scale = 0.4]
            \draw [black] (0,0) -- (1,0);
            \draw [black] (0,0) -- (0.5,0.75);
            \draw [black] (0.5,0.75) -- (1,0);
            \fill [black] (0,0) circle (4pt);
            \fill [black] (1,0) circle (4pt);
            \fill [black] (0.5,0.75) circle (4pt);
        \end{tikzpicture}
        ,
        \begin{tikzpicture}[line width=0.3mm, baseline=0.3ex,
        scale = 0.4]
            \draw [black] (0,0) -- (1,0);
            \draw [black] (0,0) -- (0.5,0.75);
            \fill [black] (0,0) circle (4pt);
            \fill [black] (1,0) circle (4pt);
            \fill [black] (0.5,0.75) circle (4pt);
        \end{tikzpicture}\}$.
    The color-invariant version from \cref{eq:reduced-sdp-0} has averaged data matrices
     \begin{align*}
        D^{\bullet_1}_{\begin{tikzpicture}[line width=0.3mm, baseline=0ex,
        scale = 0.4]
            \draw [black] (0,0) -- (1,0);
            \draw [black] (0,0) -- (0.5,0.75);
            \draw [black] (0.5,0.75) -- (1,0);
            \fill [black] (0,0) circle (4pt);
            \fill [black] (1,0) circle (4pt);
            \fill [black] (0.5,0.75) circle (4pt);
        \end{tikzpicture}}
        =
        \begin{pmatrix}
            1 & 0 \\ 0 & 1
        \end{pmatrix} \quad \text{and} \quad
        D^{\bullet_1}_{\begin{tikzpicture}[line width=0.3mm, baseline=0ex,
        scale = 0.4]
            \draw [black] (0,0) -- (1,0);
            \draw [black] (0,0) -- (0.5,0.75);
            \fill [black] (0,0) circle (4pt);
            \fill [black] (1,0) circle (4pt);
            \fill [black] (0.5,0.75) circle (4pt);
        \end{tikzpicture}}
        =
        \begin{pmatrix}
            \frac{1}{3} & \frac{2}{3} \\ \frac{2}{3} & \frac{1}{3}
        \end{pmatrix}.
    \end{align*}
    Note that in this case, the number of symmetries of the SDP increase when averaging, as compared to the constraints in \cref{ex:Goodman_SDP}.
\end{example}

We note that Balogh et al.~\cite{Balogh_Rainbow_Triangles_2017} previously already used a similar approach in order to reduce the size of the SDP search space by taking the color invariance of $f_0$ into account. Instead of coloring the edges of a complete graph, they use a theory where the edges are partitioned according to whether they have the same color or not, i.e., they use $\simeq_g$ as their underlying notion of graph (or flag) isomorphism. In practice this leads to the same reduction in the number of constraints and blocks but it additionally reduces number of flags. This however comes at a cost, since reducing the number of flags does not guarantee the same expressiveness of the SDP that our approach guarantees through \cref{thm:color_invariant_sdp}.\footnote{In practice, once $N$ is large enough there is often still a sufficient amount of redundancies in the formulation given by \cref{lem:one_direction_of_equiv} for the additional removal of flags through the approach used in~\cite{Balogh_Rainbow_Triangles_2017} to have not negative effect, at least when considering colorings of the edges of the complete graph. We still believe that the lossless nature motivates our approach.} This is easiest to illustrate with  
Goodman's theorem which can be proven with a sum of squares expression using the minimum graph size of $3$ as in \cref{ex:Goodman_SOS}.
If we try to prove the same result using edge-partitioned graphs with up to two partition classes, we see that graph size $3$ does not suffice anymore because there is only one flag of size 2 on the point type:
   \[
        \begin{tikzpicture}[line width=0.5mm, baseline=1ex,
        scale = 0.5]
            \draw [black] (0,0) -- (1,0);
            \draw [black] (0,0) -- (0.5,0.75);
            \draw [black] (0.5,0.75) -- (1,0);
            \fill [black] (0,0) circle (4pt);
            \fill [black] (1,0) circle (4pt);
            \fill [black] (0.5,0.75) circle (4pt);
            \node [isosceles triangle, minimum width=1cm, minimum height=0.75cm, draw, rounded corners, shape border rotate=90,  isosceles triangle stretches] (t) at (0.475,0.23) {};
        \end{tikzpicture}
        -0\cdot \varnothing
        \ge 0\cdot\Bigg \llbracket \ 
        \begin{tikzpicture}[line width=0.5mm, baseline=1ex,
        scale = 0.5]
            \fill [black] (0,0) circle (4pt);
            \fill [black] (0,1) circle (4pt);
            \node[scale=0.7] at (0.5,-0.3) {1};
            \draw (0,0) -- (0,1);
            \draw[rounded corners] (-0.3, -0.3) rectangle (0.3, 1.3) {};
        \end{tikzpicture}^2
        \Bigg\rrbracket 
            \begin{tikzpicture}[line width=0.5mm, baseline=2ex,
        scale = 0.5]
            \fill [black] (0,0) circle (4pt);
            \node[scale=0.7] at (0.3,-0.3) {1};
            \end{tikzpicture}
    \]
Here, the circled triangle means that all edges of the triangle are in the same partition class. Similarly the circled flag means that its only edge is in a partition class.

\subsection{Block-diagonalization}\label{sec:block_diagonalization}

    The group $\Lambda_\tau$ permutes the set of flags $\mF^\tau_{\ell_\tau}$. 
    This permutation extends to a group action of $\Lambda_\tau$ on $\RR^{\mF^\tau_{\ell_\tau}}$, so $\RR^{\mF^\tau_{\ell_\tau}}$ is an $\mathbb R[\Lambda_\tau]$-module. Given $x\in \RR^{\mF^\tau_{\ell_\tau}}$ and $g\in \Lambda_\tau$, it is customary to denote by $g \, x$ the image of $x$ under the linear map corresponding to $g$. Similarly, $g$ denotes the linear map $\RR^{\mF^\tau_{\ell_\tau}}\to \RR^{\mF^\tau_{\ell_\tau}}$ as well.

    \begin{example}\label{ex:group_action}
        Continuing \cref{ex:Goodman_SDP_averaged}, note that 
        \[
            \Lambda_{\bullet_1} = \{(\operatorname{id},\operatorname{id}_1), ((0\mapsto1,1\mapsto0),\operatorname{{id}_1}) \}.
        \]
        The identity element of $\Lambda_{\bullet_1}$ acts as an identity on $\mF^{\bullet_1}_{2}$ whereas the non-identity of $\Lambda_{\bullet_1}$ transposes the two flags in $\mF^{\bullet_1}_{2}$. When we view $\Lambda_{\bullet_1}$ as acting on the vector space $\mathbb R^{\mF_2^{\bullet_1}}$ then the two elements of $\Lambda_{\bullet_1}$ correspond to the linear maps given by the matrices
        \[
            \begin{pmatrix}
                1 & 0 \\ 0 & 1
            \end{pmatrix}
            \quad \text{and} \quad
            \begin{pmatrix}
                0 & 1 \\ 1 & 0
            \end{pmatrix}.
        \]
    \end{example}
    
    By \cref{lem:averaged_are_invariant}, the constraint matrices $\hat D^\tau_F$ are invariant when permuting the flag indices through elements of $\Lambda_\tau$. We have already observed this in \cref{ex:Goodman_SDP_averaged}.
    For the remainder of this section, we view the matrix $\hat D^\tau_F$ as an $\mathbb R$-linear map $\mathbb \RR^{\mF^\tau_{\ell_\tau}} \to \RR^{\mF^\tau_{\ell_\tau}}$. The group-invariance from \cref{lem:averaged_are_invariant} can be rephrased as saying that for each $g\in \Lambda_\tau$ we have 
    \[
        \hat D^\tau_F \, g = g \, \hat D^\tau_F.
    \]
    Moreover, we may equivalently say that $\hat D_F^\tau$ is an endomorphism of the module $\RR^{\mF^\tau_{\ell_\tau}}$, i.e.,
    \[
        \hat D^\tau_F \in \operatorname{End}_{\mathbb R[\Lambda_\tau]}(\RR^{\mF^\tau_{\ell_\tau}}).
    \]
The invariance under group actions of constraint matrices, means that the constraint matrices have many duplicate entries. Our goal in this section is to \emph{block-diagonalize} the matrices in such a way that all of the symmetries are removed and the resulting problem formulation therefore relies on fewer variables.

The problem of block-diagonalizing group symmetric SDPs has been extensively studied~\cite{Babai_1989,BachocEtAl_2012,Murota_2010,gatermann2004symmetry}. There is a theoretical best possible diagonalization, i.e., change of basis resulting in a more granular block diagonal structure of the matrices, that is given by the constraint matrices $(\hat D_F^\tau)_{F\in\mF^\tau_{\ell_\tau}}$ or, more specifically, their symmetries. The goal is to find a single base change matrix $B^\tau$ over $\mF^\tau_{\ell_\tau}\times [j_\tau]$, ideally with $j_\tau\ll |\mF^\tau_{\ell_\tau}|$, so that all matrices $(B^\tau)^T D_F^\tau B^\tau$ for $F\in\mF^\tau$ are minimally block-diagonal. Some approaches such as~\cite{Murota_2010} determine the symmetries present in the original constraints $D^\tau_F$ by itself, while most approaches assume prior knowledge of (at least part of) the symmetries. Given the results stated in this section, we are in the latter case. Here, methods such as \cite{Babai_1989}, allow one to determine a block-diagonalization by finding simple subspaces in $\mathbb R^{\mF^\tau_{\ell_\tau}}$

\begin{example}\label{ex:concrete_base_change}
    For the two constraint matrices from \cref{ex:Goodman_SDP_averaged}, the following matrix
    \[
        B^{\bullet_1} = \frac{1}{\sqrt{2}}
        \begin{pmatrix}
            1 & -1 \\
            1 & 1
        \end{pmatrix}
    \]
    simultaneously block-diagonalizes both constraints. Instead of using the constraints from \cref{ex:Goodman_SDP}, we may use the constraints:
    \[
        (B^{\bullet_1})^T
        D^{\bullet_1}_{\begin{tikzpicture}[line width=0.3mm, baseline=0ex,
        scale = 0.4]
            \draw [black] (0,0) -- (1,0);
            \draw [black] (0,0) -- (0.5,0.75);
            \draw [black] (0.5,0.75) -- (1,0);
            \fill [black] (0,0) circle (4pt);
            \fill [black] (1,0) circle (4pt);
            \fill [black] (0.5,0.75) circle (4pt);
        \end{tikzpicture}}
        B^{\bullet_1}
        =
        \begin{pmatrix}
            1 & 0 \\
            0 & 1
        \end{pmatrix}
        \quad
        (B^{\bullet_1})^T
        D^{\bullet_1}_{\begin{tikzpicture}[line width=0.3mm, baseline=0ex,
        scale = 0.4]
            \draw [black] (0,0) -- (1,0);
            \draw [black] (0,0) -- (0.5,0.75);
            \fill [black] (0,0) circle (4pt);
            \fill [black] (1,0) circle (4pt);
            \fill [black] (0.5,0.75) circle (4pt);
        \end{tikzpicture}}
        B^{\bullet_1}
        =
        \begin{pmatrix}
            1 & 0 \\
            0 & -\frac{1}{3}
        \end{pmatrix}.
    \] 
    Note that this reduced the problem structure from a single $2 \times 2$ block to two $1 \times 1$ blocks, i.e., in this example we obtained a linear program (LP).
\end{example}

We briefly sketch the theory underlying the theoretical best possible decomposition given known symmetries.
Let $R$ denote the $\mathbb R$-subalgebra of the matrix ring $\operatorname{End}_{\mathbb R}(\RR^{\mF^\tau_{\ell_\tau}})$ that is generated by all matrices $\hat D^\tau_F$ as $F$ varies in $\mu(\mF)$. Since $R$ is generated by symmetric matrices, it follows that the vector space $\mathbb R^{\mF^\tau_{\ell_\tau}}$ is a semi-simple $R$-module. Indeed, the orthogonal complement of any $R$-invariant subspace is again $R$-invariant. By \cite[pp. 654, Proposition 4.7]{SergeLang2002}, $R$ is semi-simple, so $R = \prod_{j\in J} R_j$ with each $R_j$ a simple ring.

We can find in $\mathbb R^{\mF^\tau_{\ell_\tau}}$ a simple $R_j$-subspace for every $j\in J$. This is because the unit $e_j$ of $R_j$ is non-zero in $\operatorname{End}_{\mathbb R}(\RR^{\mF^\tau_{\ell_\tau}})$ and so $R_j$ acts faithfully on $\mathbb R^{\mF^\tau_{\ell_\tau}}$. Such a subspace $V_j$ may be found in the isotypic subspace $e_j\RR^{\mF^\tau_{\ell_\tau}}$. By Wedderburn's Theorem \cite[pp. 649, Corollary 3.5]{SergeLang2002}:
\[
    R_j \simeq \operatorname{End}_{\operatorname{End}_{\mathbb R}(R_j)}(V_j)
\]
and by Frobenius's theorem we know that $\operatorname{End}_{\RR}(R_j)\in \{\mathbb R,\mathbb C, \mathbb H\}$.
In other words, there exists a unique sequence $(n_j,K_j)\in \mathbb N\times \{\mathbb R, \mathbb C, \mathbb H\}$ so that
\begin{equation}\label{eq:ring_minimal_decomposition}
    R_j \simeq \operatorname{Mat}^{n_j\times n_j}(K_j)
    \quad \text{ and }
    \quad 
    R \simeq \bigoplus_{j\in J} \operatorname{Mat}^{n_j\times n_j}(K_j).
\end{equation}
The matrices $B^\tau$ that we are looking for, are those base change matrices that witness the second isomorphism in \cref{eq:ring_minimal_decomposition}.

\begin{example}
    Continuing \cref{ex:concrete_base_change}, the matrix $B^{\bullet_1}$ is normal and so we see that:
    \[
        R = \{B^{\bullet_1}\begin{pmatrix}
            r & 0 \\ 0 & s
        \end{pmatrix}(B^{\bullet_1})^T \mid r,s\in\mathbb R \}.
    \]
    As a ring, this is isomorphic to the product of fields $\mathbb R\times\mathbb R$, each field of course being a simple ring. The units are:
    \[
        B^{\bullet_1}\begin{pmatrix}
            1 & 0 \\ 0 & 0
        \end{pmatrix}(B^{\bullet_1})^T
        =
        \begin{pmatrix}
            \frac{1}{2} & \frac{1}{2} \\ \frac{1}{2} & \frac{1}{2}
        \end{pmatrix}
        \quad
        B^{\bullet_1}\begin{pmatrix}
            0 & 0 \\ 0 & 1
        \end{pmatrix}(B^{\bullet_1})^T
        =
        \begin{pmatrix}
            \frac{1}{2} & -\frac{1}{2} \\ -\frac{1}{2} & \frac{1}{2}
        \end{pmatrix}        
    \]
    The simple $R_j$-spaces are then the images of the units because these are already simple (as invariant one dimensional vector spaces)
    \[
        \operatorname{span}_{\mathbb R}
        \begin{pmatrix}
            1 \\ 1
        \end{pmatrix}
        \quad  \text{and} \quad
        \operatorname{span}_{\mathbb R}
        \begin{pmatrix}
            1 \\ -1
        \end{pmatrix}.
    \]
    Finally, $B^{\bullet_1}$ is the base change matrix for the sum of these two simple spaces.
\end{example}

For flag algebra SDPs, numerical evidence suggests that $R=\operatorname{End}_{\mathbb R[\Lambda_\tau]}(\mathbb R^{\mF^\tau_{\ell_\tau}})$. Unfortunately, the randomized algorithms from \cite{Murota_2010} and \cite{Babai_1989} were problematic for large flag algebra SDPs. The matrices $\hat D^\tau_F$ are extremely sparse and multiplying them with dense floating point matrices produces many near zero matrix entries. 
We therefore derived our own, simplified partial block-diagonalization, that we will refer to as a \emph{rational isotypic decomposition}, in order to provide an efficient diagonalization involving rational base change matrices only. The algorithm is summarized in  \cref{alg:Rational_Isotypic}.

    Letting $R = \operatorname{End}_{\mathbb R[\Lambda_\tau]}(\RR^{\mF^{\tau}_{\ell_\tau}})$, there is a bijection between the simple rings $R_j$ and the irreducible characters $\chi$ of $\Lambda_\tau$. Given an irreducible character $\chi$,
    the unit $e_\chi$ of the simple ring $R_\chi$ is
    \[
        e_\chi = \frac{\operatorname{deg}(\chi)}{|\Lambda_\tau|}\sum_{g\in \Lambda_\tau} \chi(g^{-1}) g.
    \]
    The idea of the isotypic decomposition is to simply take the isotypic component $e_\chi\mathbb R^{\mF^\tau_{\ell_\tau}} = \operatorname{im}(e_\chi)$ itself instead of searching for simple subspaces $V_\chi$ in $e_\chi\mathbb R^{\mF^\tau_{\ell_\tau}}$.
    %Aside from this, \cref{alg:Rational_Isotypic}
    This ensures two things: firstly, $\chi$ may have complex or non-rational entries possibly making $\operatorname{im}(e_\chi)$ non-rational and even complex valued. Therefore, in lines 4 to 9 of \cref{alg:Rational_Isotypic}, if $\chi$ is not rational, but $\chi+\overline{\chi}$ is, then we take the invariant space $\operatorname{im}(e_\chi) + \operatorname{im}(e_{\overline{\chi}})$. If neither $\chi$ nor $\chi+\overline{\chi}$ is rational then we add $\chi$ to the class function $\chi^\bot$ which is initialized to $0$ and accumulates over time. Then we take $\operatorname{im}(e_{\chi^\bot})$.

\begin{example}
    We begin again with \cref{ex:group_action}. The group $\Lambda_\tau$ has two irreuducible characters, namely $\chi_1 = (1,1)$ and $\chi_2 = (1,-1)$. In this case, both are already rational. The isotypic components for each character are the images of the units:
    \[
        e_{\chi_1} = \frac{1}{2}\left(
        \begin{pmatrix}
            1 & 0 \\ 0 & 1
        \end{pmatrix}
        +
        \begin{pmatrix}
            0 & 1 \\ 1 & 0
        \end{pmatrix}
        \right)
        =
        \begin{pmatrix}
            \frac{1}{2} & \frac{1}{2}  \\ \frac{1}{2}  & \frac{1}{2} 
        \end{pmatrix}
        \quad \implies
        \quad
        \operatorname{span}_{\mathbb R}
        \begin{pmatrix}
            1 \\ 1
        \end{pmatrix}
    \]
    \[
        e_{\chi_2} = \frac{1}{2}\left(
        \begin{pmatrix}
            1 & 0 \\ 0 & 1
        \end{pmatrix}
        -
        \begin{pmatrix}
            0 & 1 \\ 1 & 0
        \end{pmatrix}
        \right)
        =
        \begin{pmatrix}
            \frac{1}{2} & -\frac{1}{2}  \\ -\frac{1}{2}  & \frac{1}{2} 
        \end{pmatrix}
        \quad \implies
        \quad 
        \operatorname{span}_{\mathbb R}
        \begin{pmatrix}
            1 \\ -1
        \end{pmatrix}.
    \]
\end{example}

\begin{algorithm}
    \caption{Rational Isotypic Decomposition}
    \label{alg:Rational_Isotypic}
    \begin{algorithmic}[2]
        \REQUIRE A group $\Lambda_\tau$, an action $\Lambda_\tau\times [\mF^\tau_{\ell_\tau}]\to [\mF^\tau_{\ell_\tau}]$
        \ENSURE Every $V_\chi$ has a rational basis and $\RR^{\mF^\tau_{\ell_\tau}} = \bigoplus_{\chi\in C\cup\{\chi^\bot\}} V_\chi$
        \STATE $ C \gets \varnothing$
        \STATE $ \chi^\bot \gets (0,\dots,0) $
        \FOR{$\chi$ an irreducible character of $\Lambda_\tau$}
            \IF{ $\chi$ is rational}
                \STATE $C \gets C \cup \{\chi\}$
            \ELSIF{$\chi+\overline{\chi}$ is rational}
                \STATE $C \gets C \cup \{\chi+\overline{\chi}\}$
            \ELSE
                \STATE $\chi^\bot = \chi^\bot + \chi$
            \ENDIF
        \ENDFOR
        \FOR{$\chi$ in $C\cup \{\chi^\bot\}$}
            \STATE $V_\chi \gets \operatorname{im}(\frac{\operatorname{deg}(\chi)}{|\Lambda_\tau|}\sum_{g\in \Lambda_\tau} \chi(g^{-1}) g)$ 
        \ENDFOR
        \RETURN $\{V_\chi\}$
    \end{algorithmic}
\end{algorithm}

\medskip

Finally, let us note that Razborov proposed an invariant-anti-invariant-split in \cite{Razborov_2010}. This can also be considered within the framework of block diagonalization, by viewing it as a further simplification of our rational isotypic decomposition. In Razborov's case, one only computes the isotypic component of $\chi = \chi_{\operatorname{triv}}$, which is called the invariant part, where $\chi_{\operatorname{triv}} = (1,\dots, 1)$ is the trivial character. Afterwards, we take the complement of the invariant space $V_{\chi_{\operatorname{triv}}}$, which is called the anti-invariant part. Our method results in a more significant reduction in the number of variables, while still being relatively easy to implement and taking a reasonable amount of compute power to run even for larger problems.

\section{Stability results from flag algebra certificates}\label{sec:stability}

If one can show that a positive element $f_0$ is tight, then it is of interest to determine whether this is accompanied by a stability statement, i.e., if we can say that any flag $F \in \mF^\varnothing$ with $p(f_0; F)$ small must also be {\lq}close{\rq} to an element of some $\mS \subseteq \mF^\varnothing$.
 The most common measure of closeness between two flags $F$ and $F'$ is the \emph{edit distance} $\Delta_{\operatorname{edit}}(F, F')$ defined as the minimum number of edges in $F$ that need to be recolored in order to obtain a flag isomorphic to $F'$.
\begin{definition}\label{def:stability}
    $f_0$ is \emph{stable} with respect to some set $\mS\subseteq \mF^\varnothing$, if there exists an integer $n_0$ so that for any $\varepsilon > 0$ there exists some $\delta > 0$ such that for all $G\in\mF^\varnothing$ with $v(G) = n\ge n_0$ and $p(f_0; G) \leq \min_{F\in\mF^\varnothing_{n}}p(f_0;F) + \delta$ there exists $H\in\mS$ with $v(H)=v(G)$ and $\Delta_{\operatorname{edit}}(G,H) \le \varepsilon \binom{n}{2}$.
\end{definition}

One could obviously show stability with respect to $\mS = \mF^\varnothing$, so the goal is to choose $\mS$ as small as possible. Note that $\lim_{n \to \infty} \min_{F\in\mF^\varnothing_n}p(f_0;F) = 0$ since we are assuming $f_0$ a tight positive element. %Also note that for $p(f_0;G) = \min_{F\in\mF^\varnothing_{n}}p(f_0;F)$ we actually get $G\in\mS$, that is we can also get an exact description of extremal constructions.

The general argument made by Cummings et al.~\cite{CummingsEtAl_2013} to establish exactness in the three-color case is also applicable to the four-color case. It can in fact also be formulated for the case of an arbitrary number of colors as well as strengthened to yield a stability result. The former of these two aspects is significant in that it theoretically implies that an equivalence of the two problems, that is determining the Ramsey number $R_{c-1}(3)$ and determining the Ramsey multiplicity $m_c(3)$, could be established without having solved both or even either of them. We will discuss this in more detail in \cref{sec:discussion}.
Most major proof ideas of this section will therefore build on~\cite{CummingsEtAl_2013}, but we provide a self-contained and slightly more verbose proof for the convenience of the reader.

In \cref{sec:formal_statement_theorem} we will introduce some notation necessary to then formally state the main result from which we will derive stability. It essentially says that if a particular set of colorings can be assumed to not occur in any extremal construction, we obtain stability. The fact that these colorings have zero density will be derived from the flag algebra certificate. In \cref{sec:existence_standard_graphs} we state the notion of a standard coloring, slightly modifying it from~\cite{CummingsEtAl_2013} in order to accommodate some new additional graphs introduced in the previous section. 
 \cref{sec:standard_graphs_basic} contains some results regarding these standard colorings.
 %, in particularly showing that the forbidden graphs fully describe the structure of extremal homomorphisms.
 %, in particular showing  that the forbidden graphs already imply a complete structure theorem for extremal homomorphisms.
 The contents of \cref{sec:Ramsey_mult_graphs} are new, necessitated by the fact that we state the proof without requiring explicit knowledge of the extremal Ramsey colorings underlying the construction. Finally, \cref{sec:lifting_standard_subcoloring} 
 contains only minor modifications of the arguments in~\cite{CummingsEtAl_2013},  while
 most of \cref{sec:remove_error_edges} has been strengthened to yield stability 
%We must show that every edge color recoloring reduces the number of triangles. 
while also significantly simplifying the proof of \cref{lem:V_star_is_zero} by isolating and consistently applying an argument that decreases the number of triangles through \cref{alg:decrease_triangles}. A summary of the structure of the proof of our stability statement can be found in \cref{fig:proof_outline}.

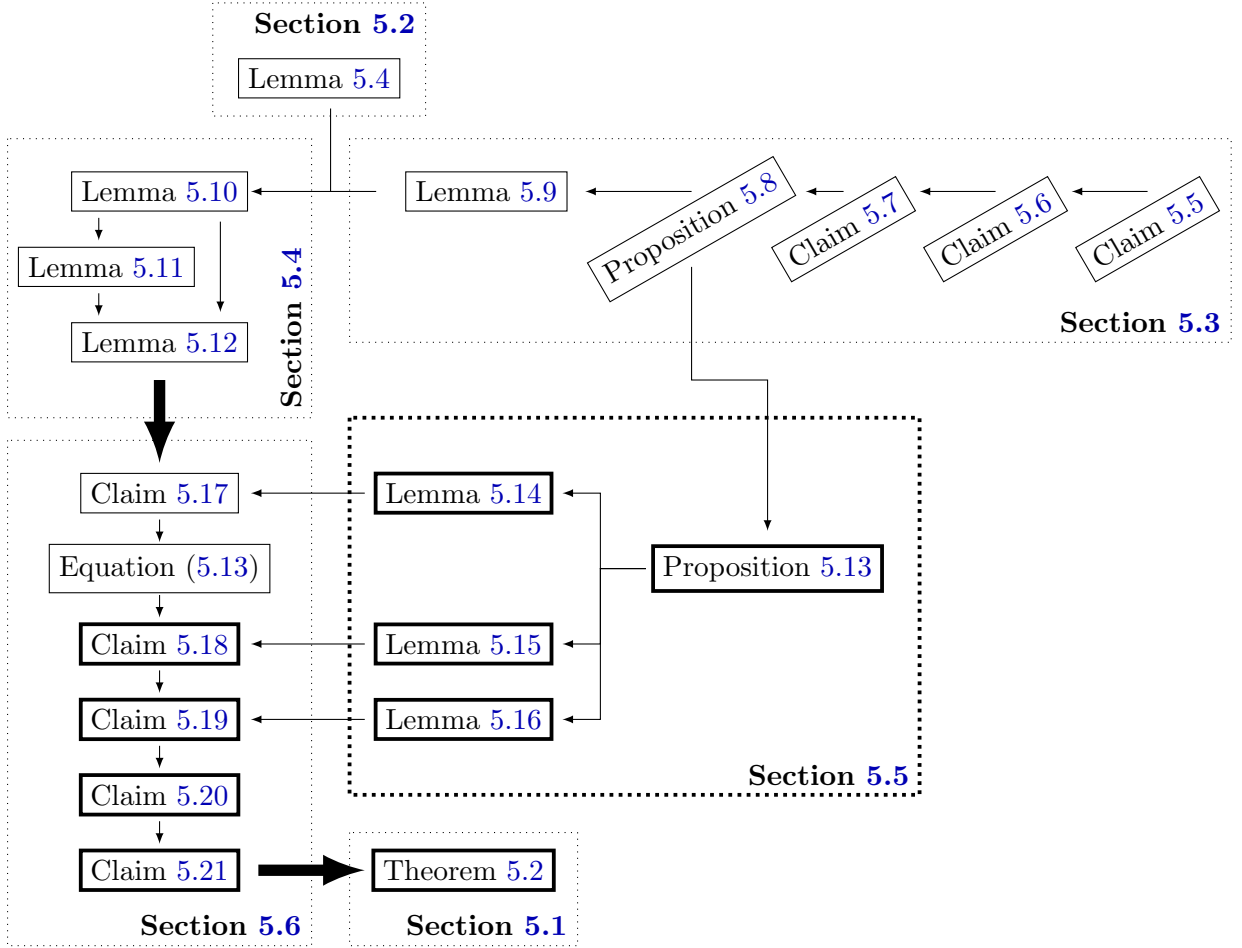
\begin{figure}[h]
    \bigskip
    \centering
    \begin{tikzpicture}

    % EXISTENCE STANDARD

    \draw[dotted] (0.7,2.5) rectangle (3.5, 1.0);
    \node[anchor=north east] at (3.5, 2.5){\textbf{\cref{sec:existence_standard_graphs}}};
    
    \node[draw, rectangle] at (2.1,1.5) {\cref{lem:Claim_1}};
    \draw  (2.25, 1.1) -- (2.25, 0) ;

    % standard coloringS
    \draw[dotted] (2.5,0.7) rectangle (14.1,-2);
    \node[anchor=south east] at (14.1,-2){\textbf{\cref{sec:standard_graphs_basic}}};
    
    \node[draw, rectangle, rotate=30] at (13,-0.5) {\cref{lem:Claim_2}};
    \draw[-latex]  (13.0,0) --  (12,0);
    \node[draw, rectangle, rotate=30] at (11,-0.5) {\cref{lem:Claim_4}};
    \draw[-latex]  (11,0) --  (10,0);
    \node[draw, rectangle, rotate=30] at (9,-0.5) {\cref{lem:standard_has_partition}};
    \draw[-latex]  (9.0,0) --  (8.5,0);
    \node[draw, rectangle, rotate=30] at (7,-0.5) {\cref{lem:standard_is_Ramsey}};
    \draw[-latex]  (7,0) --  (5.6,0);
    \node[draw, rectangle, rotate=0] at (4.3,0) {\cref{lem:standard_has_uniform_subgraph}};

    % RAMSEY PROPERTIES

    \draw[dotted, line width=0.5mm] (2.5,-3) rectangle (10.0,-8);
    \node[anchor=south east] at (10.0,-8) {\textbf{\cref{sec:Ramsey_mult_graphs}}};

    \draw[-latex]  (7.0,-1) -- (7.0,-2.5) -- (8.0,-2.5) --  (8.0,-4.5);
    \node[draw, rectangle, line width=0.5mm] at (8,-5) {\cref{th:extremals_are_Ramsey}};
    
    \draw[-latex]  (6.4, -5) -- (5.8, -5) -- (5.8, -4) -- (5.3, -4) ;
    \node[draw, rectangle, line width=0.5mm] at (4,-4) {\cref{lem:Degree_at_lest_two}};
    \draw[-latex]  (2.7, -4) -- (1.2, -4) ;

    \draw[-latex]  (6.4, -5) -- (5.8, -5) -- (5.8, -6) -- (5.3, -6) ;
    \node[draw, rectangle, line width=0.5mm] at (4,-6) {\cref{lem:two_disjoint_monochrom}};
    \draw[-latex]  (2.7, -6) -- (1.2, -6) ;

    \draw[-latex]  (6.4, -5) -- (5.8, -5) -- (5.8, -7) -- (5.3, -7) ;
    \node[draw, rectangle, line width=0.5mm] at (4,-7) {\cref{lem:exists_cherry_in_Ramsey}};
    \draw[-latex]  (2.7, -7) -- (1.2, -7) ;

    % RECONSTRUCTOR
    
    \draw[dotted] (-2,0.7) rectangle (2,-3);
    \node[anchor=south east] at (2,-3) 
    {\rotatebox{90}{{\textbf{\cref{sec:lifting_standard_subcoloring}}}}};

    \node[draw, rectangle] at (0,0) {\cref{lem:Claim_7}};
    \draw[-latex]  (2.85, 0) -- (1.2, 0) ;

    \draw[-latex]  (0.8, -0.4) -- (0.8, -1.6) ;
    \node[draw, rectangle] at (-0.7,-1) {\cref{lem:bound_for_zerosets}};

    \draw[-latex]  (-0.8, -0.35) -- (-0.8, -0.65) ;
    \draw[-latex]  (-0.8, -1.35) -- (-0.8, -1.65) ;
    
    \node[draw, rectangle] at (0,-2) {\cref{lem:Claim_8}};

    \draw[-latex, line width=1.5mm]  (0, -2.5) -- (0, -3.6) ;

    % RECOLORING
    \draw[dotted] (-2,-3.3) rectangle (2,-10);
    \node[anchor=south east] at (2,-10) {\textbf{\cref{sec:remove_error_edges}}};

    \node[draw, rectangle] at (0,-4) {\cref{lem:Claim_10}};
    \draw[-latex]  (-0, -4.35) -- (-0, -4.65) ;
    
    \node[draw, rectangle] at (0,-5) {\cref{eq:f2}};
    \draw[-latex]  (-0, -5.35) -- (-0, -5.65) ;

    \node[draw, rectangle, line width=0.5mm] at (0,-6) {\cref{lem:V_star_is_zero}};
    \draw[-latex]  (-0, -6.35) -- (-0, -6.65) ;
    
    \node[draw, rectangle, line width=0.5mm] at (0,-7) {\cref{claim:f4}};
    \draw[-latex]  (-0, -7.35) -- (-0, -7.65) ;
    
    \node[draw, rectangle, line width=0.5mm] at (0,-8) {\cref{claim:f5}};
    \draw[-latex]  (-0, -8.35) -- (-0, -8.65) ;
    
    \node[draw, rectangle, line width=0.5mm] at (0,-9) {\cref{claim:f6}};

    % THEOREM
    
    \draw[dotted] (2.5,-8.5) rectangle (5.5,-10);
    \node[anchor=south east] at (5.5,-10)
    {\rotatebox{0}{{\textbf{\cref{sec:formal_statement_theorem}}}}};

    \node[draw, rectangle, line width=0.5mm] at (4,-9) {\cref{th:stability}};
    \draw[-latex, line width=1.5mm]  (1.3, -9) -- (2.7, -9) ;

    \end{tikzpicture}
    \caption{Outline of the proof of \cref{th:stability} with significant modifications from~\cite{CummingsEtAl_2013} highlighted.}
    \label{fig:proof_outline}
\end{figure}

\subsection{Formal statement of the stability result}\label{sec:formal_statement_theorem}

Let the number of colors $c$ be arbitrary but fixed throughout the rest of this section and exclusively consider the problem of minimizing the density of monochromatic triangles. %We see no reason why the arguments should not work for minimizing arbitrary sums of cliques. We have not formulated our arguments in this generality.
Note that we already introduced the sets $\mG_{\operatorname{ex}}^{(3)}$ and $\mG_{\operatorname{ex}}^{(4)}$ in \cref{sec:background}. We will now generalize this for arbitrary $c$. Let $\mL_{\operatorname{ex}}^{(c)}$ denote all looped colorings in $\mL^{(c)}_{R_{c-1}(3)-1}$ where all loops get the same color and the non-loop edges form a $R_{(c-1)}(3)$-Ramsey graph with the remaining $c-1$ colors. Note that $|\mL_{\operatorname{ex}}^{(3)}| = 1 \cdot 3 = 3$ and $|\mL_{\operatorname{ex}}^{(4)}| = 2 \cdot 4 = 8$ since the respective Ramsey colorings are self-complementary. Next, we define $\mG^{(c)}_{\operatorname{ex}}$ as the set of all colorings
that can be obtained by selecting an element $G\in\mB_\mathbf{1/(R_{c-1}(3)-1)} (E)$ for some $E \in\mL_{\operatorname{ex}}^{(c)}$ and recoloring some edges of $G$ to the loop color of $C$ without creating any additional monochromatic triangles. Note that throughout, we will treat the elements of $\mG^{(c)}_{\operatorname{ex}}$ both as colorings and as flags of the empty type.

%
%Let us introduce some necessary notation.
We will write $\hat K_3 = \sum_{i \in [c]} K_3^i$ for
the flag algebra element corresponding to the sum of all monochromatic triangles. Likewise, let $\hat K_{3,3}$ denote the sum of all flags consisting of two disjoint $c_1$- and $c_2$-colored triangles with the remaining edges colored with $c_3$, where $c_1,c_2,c_3\in [c]$ are distinct colors. The summands of $\hat K_{3,3}$ can pictorially be expressed as 
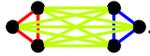
\[   
   \begin{tikzpicture}[line width=0.4mm, baseline=1ex,
        scale = 0.5]
        \node[circle,fill=black,draw,scale=0.4] (A) at (0,0) {};
        \node[circle,fill=black,draw, scale=0.4] (B) at (0,1) {};
        \node[circle,fill=black,draw, scale=0.4] (C) at (-0.66,0.5) {};
        \draw[red] (A) -- (B);
        \draw[red] (A) -- (C);
        \draw[red] (B) -- (C);
        \node[circle,fill=black,draw,scale=0.4] (E) at (2,0) {};
        \node[circle,fill=black,draw, scale=0.4] (F) at (2,1) {};
        \node[circle,fill=black,draw, scale=0.4] (G) at (2.66,0.5) {};
        \draw[blue] (E) -- (F);
        \draw[blue] (E) -- (G);
        \draw[blue] (F) -- (G);

        \draw[lime] (A) -- (E);
        \draw[lime] (B) -- (E);
        \draw[lime] (C) -- (E);
        \draw[lime] (A) -- (F);
        \draw[lime] (B) -- (F);
        \draw[lime] (C) -- (F);
        \draw[lime] (A) -- (G);
        \draw[lime] (B) -- (G);
        \draw[lime] (C) -- (G);
        
    \end{tikzpicture}.
\]

Given integers $k_1$ and $k_2\ge k_3$ with $k_1 + k_2 + k_3 = 3$, let $\mH(k_1,k_2,k_3)$ denote the family of (isomorphism classes of) flags of order $4$ that contain a subset of cardinality $3$ inducing a monochromatic triangle, where we denote the color of this triangle by $c_1$, and the fourth vertex is incident to $k_1,k_2,k_3$ edges that respectively have the distinct colors $c_1,c_2,c_3$. Using this, we let
\[
    \mH =
    \mH(2,1,0)\cup \mH(1,1,1)\cup \mH(0,2,1),
\]
the elements of which can pictorially be represented as
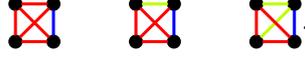
\[
        \begin{tikzpicture}[line width=0.4mm, baseline=1ex,
        scale = 0.5]
            \node[circle,fill=black,draw,scale=0.4] (A) at (0,0) {};
            \node[circle,fill=black,draw, scale=0.4] (B) at (1,0) {};
            \node[circle,fill=black,draw, scale=0.4] (C) at (1,1) {};
            \node[circle,fill=black,draw, scale=0.4] (D) at (0,1) {};
            \draw[red] (A) -- (B);
            \draw[red] (A) -- (C);
            \draw[red] (A) -- (D);
            \draw[blue] (B) -- (C);
            \draw[red] (B) -- (D);
            \draw[red] (C) -- (D);
        \end{tikzpicture}
    \hspace{9mm} 
    \begin{tikzpicture}[line width=0.4mm, baseline=1ex,
        scale = 0.5]
        \node[circle,fill=black,draw,scale=0.4] (A) at (0,0) {};
        \node[circle,fill=black,draw, scale=0.4] (B) at (1,0) {};
        \node[circle,fill=black,draw, scale=0.4] (C) at (1,1) {};
        \node[circle,fill=black,draw, scale=0.4] (D) at (0,1) {};
        \draw[red] (A) -- (B);
        \draw[red] (A) -- (C);
        \draw[red] (A) -- (D);
        \draw[blue] (B) -- (C);
        \draw[red] (B) -- (D);
        \draw[lime] (C) -- (D);
    \end{tikzpicture}
    \hspace{9mm} 
    \begin{tikzpicture}[line width=0.4mm, baseline=1ex,
        scale = 0.5]
        \node[circle,fill=black,draw,scale=0.4] (A) at (0,0) {};
        \node[circle,fill=black,draw, scale=0.4] (B) at (1,0) {};
        \node[circle,fill=black,draw, scale=0.4] (C) at (1,1) {};
        \node[circle,fill=black,draw, scale=0.4] (D) at (0,1) {};
        \draw[red] (A) -- (B);
        \draw[lime] (A) -- (C);
        \draw[red] (A) -- (D);
        \draw[blue] (B) -- (C);
        \draw[red] (B) -- (D);
        \draw[lime] (C) -- (D);
    \end{tikzpicture}.
    \]
    We will write $\hat K_{3,1}$ for the sum of all flags in $\mH$. 
    Using this notation, we can now formally state the main stability result.
    
\begin{theorem}\label{th:stability}
    If $\hat K_3 - \varsigma$ is a tight positive algebra element and there exists some $s>0$ such that $\hat K_3-\varsigma \ge s(\hat K_{3,1} + \hat K_{3,3})$, then $\hat K_3 - \varsigma$ is stable with respect to $\mG^{(c)}_{\operatorname{ex}}$ and therefore in particular $\varsigma = 1 / (R_{c-1}(3)-1)^2$.
\end{theorem}
Let us introduce a shorthand notation for the assumptions of \cref{th:stability} by defining
\begin{itemize} \itemsep0em 
    \item[\boxedchar{A1}] we know that $\hat K_3 - \varsigma$ is a tight positive algebra element,
    \item[\boxedchar{A2}] there exists some $s > 0$ satisfying $\hat K_3-\varsigma \ge s(\hat K_{3,1} + \hat K_{3,3})$, and
    \item[\boxedchar{A3}] we are given some large enough $G \in \mF^\varnothing$ whose order we will denote by $n$.
\end{itemize}
The flag $G$ will serve as a fixed but arbitrary instance for which the properties of \cref{def:stability} need to be shown. Throughout the following sections, we will also use constants $\varepsilon_0,...,\varepsilon_8,\delta_0\in\mathbb R$ and $n_1,n_2\in \mathbb N$ that depend on the number of colors $c$ and the constant $s$ as given by \boxedchar{A2} along with both $C > 0$ and $n_0 \in\mathbb N$ from \cref{def:stability}. They will satisfy
\[
    0 < \frac{1}{n_0} < \varepsilon_0 < \frac{1}{n_1} < \varepsilon_1 <
    \varepsilon_2 < \varepsilon_3 < \frac{1}{n_2} < \varepsilon_4
    < \varepsilon_5 < \varepsilon_6 < \varepsilon_7 <
    1-\delta_0 < \varepsilon_8 < 1
\]
along with certain constraints that are not further specified here but will be \whiteboxedchar{highlighted} throughout for emphasis.
% We will will also need to specify the constraints for both $C > 0$ and $n_0 \in\mathbb N$ in \cref{def:stability}. In particular, we will assume that
% \[
%     n_0 > 1 / \varepsilon_0
% \]
% along with additional assumption likewise \whiteboxedchar{highlighted} throughout.
Note that we will assume $n = v(G) \geq n_0$ throughout with $G$ given by \boxedchar{A3}. 

\subsection{Definition and existence of standard flags}\label{sec:existence_standard_graphs}

The following definition now introduces a notion of flags that have sufficiently low $\hat K_3$-density as well as zero $\hat K_{3,1} + \hat K_{3,3}$ density.

\begin{definition}
    A flag $F\in\mF^\varnothing$ is \emph{standard} if $p(\hat K_{3,1} + \hat K_{3,3}; F)= 0$. It is \emph{$\alpha$-standard} for some $\alpha>0$ if additionally  $p(\hat K_3;F) \le  \alpha$.
\end{definition}

\begin{lemma}\label{lem:Claim_1} Assume \boxedchar{A1}, \boxedchar{A2}, and \boxedchar{A3}. For a subset of $\mathscr M \subseteq V(G)$ of size $n_1$ chosen uniformly at random, we have 
\[
    \mathbb P[G[\mathscr M] \text{ is }
    (\varsigma+\varepsilon_1)\text{-standard}]
    \ge 1 - 
    \left(\frac{1}{\varepsilon_1(1+\varepsilon_1)} + \frac{\binom{n_1}{6}}{s} \right)
    \left(p(\hat K_3;G)-\varsigma\right)
    -
    \left( 
    \frac{1}{1+\varepsilon_1}+\frac{1}{s}
    \right)
    \varepsilon_1.
\]

\end{lemma}
\begin{proof}
    By \cref{lemma:double_counting}, we have both
    \begin{align*}
        \mathbb E[p(\hat K_3;G[\mathscr M])]
        =
        p(\hat K_3;G) \quad \text{and} \quad
        \mathbb E[p(\hat K_{3,1} + \hat K_{3,3};G[\mathscr M])]
        =
        p(\hat K_{3,1} + \hat K_{3,3};G).
    \end{align*}
    Let $x=\mathbb P[p(\hat K_3;G[\mathscr M])>\varsigma+\varepsilon_1]$ and note that, for \whiteboxedchar{sufficiently large $n_0$}, we have by \boxedchar{A1} that $p(\hat K_3; G)\ge \varsigma-\varepsilon_1^2$.
    By approximating $p(\hat K_3;G)$ through a step function, we have
    \[
    (\varsigma-\varepsilon_1^2) \, (1-x)
    +
    (\varsigma+\varepsilon_1) \, x  
    \le
    p(\hat K_3;G),
    \]
    and therefore
    \begin{equation} \label{eq:k3bnd}
        \mathbb P[p(\hat K_3;G[\mathscr M] > \varsigma + \varepsilon_1) = x
        \le
        \frac{p(\hat K_3;G)-\varsigma+\varepsilon_1^2}{\varepsilon_1(1+\varepsilon_1)}.
    \end{equation}
    Using Markov's inequality, we get
    \begin{align*}
        \mathbb P[ 
        p(\hat K_{3,1} + \hat K_{3,3}; G[\mathscr M]) > 0] &\leq \mathbb P \left[p(\hat K_{3,1};G[\mathscr M]) \ge \binom{n_1}{4}^{-1} \right]
        +\mathbb P \left[p(\hat K_{3,3};G[\mathscr M]) \ge \binom{n_1}{6}^{-1} \right] \\ 
        &\le \binom{n_1}{4} \, p(\hat K_{3,1};G)
        +\binom{n_1}{6} \, p(\hat K_{3,3};G).
    \end{align*}
    Using \boxedchar{A2}, we have for \whiteboxedchar{sufficiently large $n_0$}
    \[
        p(\hat K_3;G)-\varsigma
        \ge
        s\left(p(\hat K_{3,1};G) + p(\hat K_{3,3}; G)\right) - \varepsilon_1 \binom{n_1}{6}^{-1}.
    \]
    Assuming that \whiteboxedchar{$n_1$ is large enough} for ${n_1 \choose 4} \leq {n_1 \choose 6}$ to hold, we therefore get
    \begin{equation}
        \mathbb P[ 
        p(\hat K_{3,1} + \hat K_{3,3}; G[\mathscr M]) > 0] \le \frac{\binom{n_1}{6}(p(\hat K_3;G)-\varsigma)+\varepsilon_1}{s}. \label{eq:k31k33bnd}
    \end{equation}
    Finally, through the union bound and by combining \cref{eq:k3bnd} as well as \cref{eq:k31k33bnd}, we therefore get
    \begin{align*}
        & \quad \mathbb P[G[\mathscr M] \text{ is } (\varsigma+\varepsilon_1)\text{-standard}] \\
        & = 1 - \mathbb P[p(\hat K_3;G[\mathscr M]) >  \varsigma+\varepsilon_1 \vee 
        p(\hat K_{3,1} + \hat K_{3,3}; G[\mathscr M]) > 0] \\
        & \geq 1 - \mathbb P[p(\hat K_3;G[\mathscr M]) >  \varsigma+\varepsilon_1] - \mathbb P[ 
        p(\hat K_{3,1} + \hat K_{3,3}; G[\mathscr M]) > 0] \\
        & \geq 1 - \frac{p(\hat K_3;G)-\varsigma+\varepsilon_1^2}{\varepsilon_1(1+\varepsilon_1)} - \frac{\binom{n_1}{6}(p(\hat K_3;G)-\varsigma)+\varepsilon_1}{s} \\
        & = 1 - 
            \left(\frac{1}{\varepsilon_1(1+\varepsilon_1)} + \frac{\binom{n_1}{6}}{s} \right)
            \left(p(\hat K_3;G)-\varsigma\right)
            -
            \left( 
            \frac{1}{1 + \varepsilon_1}+\frac{1}{s}
            \right)
            \varepsilon_1,
    \end{align*}
    implying the statement of the lemma.
\end{proof}

\subsection{Properties of standard flags \label{sec:standard_graphs_basic}}

Throughout, we will use the following notation: given two disjoint sets $S_1, S_2$, we write $S_1 \ast S_2 = \big\{ \{ v_1, v_2 \} \mid v_1 \in S_1, v_2 \in S_2\big\}$ for the set of all edges between them. If $S_1=\{v\}$, we simply write $v \ast S_2$.
%If $F$ is looped, we abbreviate $F(v,v)$ to $F(v)$.
Given a set of edges $M \subseteq S_1 \ast S_2$ and a vertex $v\in S_1$,
we let $M(v)$ denote the set $\{ u \mid \{u,v\} \in M\}$. Note that $|M(v)| \in \{0,1\}$ when $M$ is a matching.
For an $\alpha$-standard  graph $F$ let $\mX_F$ denote the set of all maximal sets of vertices inducing a monochromatic clique of size at least $\max(c+1,6)$ in $F$. Note that by Ramsey's theorem, $\mX_F$ is non-empty when \whiteboxedchar{$n_1$ is large enough}.

\begin{claim} \label{lem:Claim_2} Let $F$ be standard. Let $X \in \mX_F$ and $c_1$ the color of the monochromatic clique $F[X]$. For any
$y \in V(F)\setminus X$
there exists at most one edge $e_0 \in y \ast X$ with $F(e_0) = c_1$ as well as a color $c_2 \neq c_1$ such that $F(e) = c_2$ for all $e \in y \ast X \setminus \{e_0\}$.
\end{claim}
\begin{proof}
    If there exists an edge $e_0 \in y \ast X$ satisfying $F(e_0) = c_1$, then all other edges in $y \ast X$ must be monochromatic with a color distinct from $c_1$; if they were monochromatic with $c_1$, this would contradict the maximality of $X$, and if they were not monochromatic, then $F$ would contain an element in $\mH(2,1,0)$ or $\mH(1,1,1)$ as a subcoloring, contradicting the assumption that $F_1$ is $\varepsilon_1$-standard.
    
    If there does \emph{not} exists an edge $e_0 \in y \ast X$ satisfying $F(e_0) = c_1$, then all edges in $y \ast X$ must be monochromatic with a color distinct from $c_1$; if they were monochromatic with $c_1$, this would again contradict the maximality of $X$, and if they were not monochromatic, then by the pigeonhole principle $F$ would contain an element in $\mH(0,2,1)$ since the cardinality of $X$ is at least $c+1$, again contradicting the assumption that $F_1$ is $\varepsilon_1$-standard.
\end{proof} 

\begin{claim}\label{lem:Claim_4}
Let $F$ be standard. 
For given distinct sets $X_1,X_2\in\mX_F$ respectively corresponding to maximal cliques that are monochromatic with colors $c_1$ and $c_2$, the following holds:
    \begin{itemize}
        \item If $X_1\cap X_2 = \varnothing$, then $c_1 = c_2$ and there exists a matching  $M\subseteq X_1\ast X_2$ such that all edges in $M$ have color $c_1=c_2$ and all edges in $X_1\ast X_2 \setminus M$ have the same color $c_3$ that is distinct from $c_1$ and $c_2$.
        \item If $X_1\cap X_2= \{v\}$, then $c_1 = c_2$ and all edges in $(X_1\setminus \{v\})\ast (X_2\setminus \{v\})$ have the same color $c_3$ that is distinct from $c_1$ and $c_2$.
        \item $|X_1\cap X_2| \geq 2$ cannot occur.
    \end{itemize}
\end{claim}

\begin{proof}    
    Assume that $X_1\cap X_2 = \varnothing$. By \cref{lem:Claim_2}, there exists two sets $M_1, M_2\subseteq X_1 \ast X_2$ of edges satisfying $|M_2(x_1)| \le 1$ for any $x_1 \in X_1$ and $|M_1(x_2)| \le 1$ for any $x_2 \in X_2$  that are, respectively, all colored with color $c_1$ and $c_2$.

    % If $c_1\ne c_2$, then consequently $M_1$ and $M_2$ need to be disjoint.\todo{Needed?}
    If $c_1\ne c_2$, let us assume that $M_1$ is non-empty and pick some arbitrary edge $\{v_1,v_2\}\in M_1$ with $v_1\in X_1$ and $v_2\in X_2$. Applying \cref{lem:Claim_2} to $X_2$ and $v_1$ implies that that
    \[
        M_1 \supseteq v_1 \ast (X_2\setminus M_2(v_1)).
    \]
    with $|M_2(v_1)| \le 1$.
    Since $|M_1(x_2)|\le 1$ for all $x_2\in X_2$ by \cref{lem:Claim_2}, it follows that there can exist at most one additional edge in $M_1$ that is not in $v_1 \ast (X_2\setminus M_2(v_1))$, namely $\{M_2(v_1), z\}$ for some $z\in X_1\setminus v_1$. This however would likewise imply that $M_1 \supseteq z \ast (X_2 \setminus M_2(z))$ with $|M_2(z)| \le 1$, contradicting that $|M_1| \leq |X_2|$. Hence either $M_1 = \emptyset$ or
    \[
        M_1 = v_1 \ast (X_2\setminus M_2(v_1))
    \]
    and by symmetry either $M_2 = \emptyset$ or 
    \[
        M_2 = (X_1\setminus M_1(v_2)) \ast v_2,
    \]
    for some $v_2 \in X_2$.
    It follows that there exists $v_ \in X_1$ and $v_2 \in X_2$ such that
    \begin{equation*}
        (X_1\setminus\{v_1\})\ast (X_2\setminus \{v_2 \}) \subseteq X_1 \ast X_2 \setminus (M_1 \cup M_2)
    \end{equation*}
    and all edges in $(X_1\setminus\{v_1\})\ast (X_2\setminus \{v_2 \})$ need to have the same color $c_3$ distinct from $c_1$ and $c_2$ by \cref{lem:Claim_2} and symmetry. Then however, since we are assuming that $c_1 \neq c_2$, we get the contradiction that $p(\hat K_{3,3}; F) > 0$.
    
    If $c_1=c_2$, then $M_1 = M_2$ needs to be a matching by symmetry and $(X_1\ast X_2)\setminus (M_1\cup M_2)$ is again monochromatic with a color distinct from $c_1=c_2$.
    
    \medskip
    
    Assume now that $|X_1\cap X_2| = \{v\}$ and pick an arbitrary $w \in X_2 \setminus X_1$. Then $\{v,w\}$ has color $c_2$ and so all edges in $w \ast (X_1 \setminus \{v\})$ have the same color $c_3$ distinct from $c_2$ by \cref{lem:Claim_2}. By symmetry, the same argument is true for edges going from any $w\in X_1\setminus X_2$ to the set $X_2\setminus \{v\}$ with the color being distinct from $c_1$. Since $(X_1\setminus \{v\})\ast (X_2\setminus \{v\})$ is a complete bipartite graph, the symmetry also implies that the color $c_3$ is the same for each choice of $w$, so that all edges in $(X_1\setminus \{v\})\ast (X_2\setminus \{v\})$ have the same color $c_3$ distinct from $c_1$ and $c_2$. Finally, if $c_1 \ne c_2$, we would get a contradiction to the assumption that $p(\hat K_{3,3};F) = 0$.

    \medskip
    
    Finally, if $|X_1\cap X_2| \ge 2$, then there exist $v_1, v_2 \in X_1\cap X_2$ and the edge $\{v_1, v_2\}$ therefore has both colors $c_1$ and $c_2$, implying that $c_1 = c_2$. By assumption $X_1 \neq X_2$, so there exists $v\in X_1\Delta X_2$. Let us w.l.o.g. say that $v\in X_2$. Then, for each $w \in X_1 \cap X_2$, of which there are at least two, the edge $\{v,w\}$ as color $c_1$, contradicting \cref{lem:Claim_2}.
\end{proof}

For a standard $F$, let us therefore construct
a $c$-coloring $E_F$ of the complete and partially looped graph on vertex set 
$\mX_F\sqcup Y_F$, where $Y_F= \{ \{y\} \mid y \in V(F)\setminus 
\bigcup_{X \in \mX_F} X\}$. Given $X, X_1, X_2 \in \mX_F$ and $y \in Y_F$, we note that the edges in $y \ast X$ and $X_1 \ast X_2$ always have a clearly defined majority color in $F$ by \cref{lem:Claim_2} and \cref{lem:Claim_4}. The edges and their color in $E_F$ are now defined through 
\begin{equation}\label{eq:def_of_looped_graph_E}
    E_F(\{v_1, v_2\}) = \left\{\begin{array}{lr}
        \text{no edge} & \text{if } v_1 = v_2 \in Y_F, \\
        F(\{v_1, v_2\}) & \text{if } v_1 \neq v_2 \in Y_F, \\
        \text{color of } F[v_1]  & \text{if } v_1 = v_2 \in \mX_F, \\
        \text{majority color in } v_1 \ast v_2 & \text{otherwise}. 
        \end{array}\right.
\end{equation}

\begin{claim}\label{lem:standard_has_partition}
    Let $F$ be standard. Given any $i\in [c]$, the only embeddings of $K_{\max(c+1,6)}^i$ in $E_F$ map everything to a single vertex. The same holds for $K_3^i$ if the embedding contains at least one vertex in $\mX_F$.
\end{claim}
\begin{proof}
    We begin with the proof of the second claim.
    Assume $v_1, v_2, v_3\in V(E)$ with $v_1\in\mX_F$ forms a monochromatic triangle in $E_F$. Denote the color of this triangle by $c_\triangle \in [c]$. The case 
    $|\{v_1,v_2,v_3\}|= 2$ is not possible \cref{lem:Claim_4} contradicts the case that $v_1,v_2,v_3 \in \mX_F$ and \cref{lem:Claim_2} contradicts the case where one of $v_1,v_2,v_3$ is in $Y$. Note that we cannot have two of the vertices in $Y$ as $E_F$ does not have loops in $Y$.
    
    We may therefore assume that $|\{v_1,v_2,v_3\}|= 3$. Choose any two vertices $u_2 \in v_2$, $u_3 \in v_3$ in $F$ satisfying $F(\{u_2,u_3\}) = E(\{v_2,v_3\}) = c_\triangle$. Note that such an edge exists by definition of $E_F$. Since $v_1 \in \mX_F$ by assumption, we know that $|v_1| \ge 4$ and we can by the pigeonhole principle pick two distinct vertices $w_1,w_2 \in v_1$ by \cref{lem:Claim_2} satisfying $F(\{u_2, w_1\}) = F(\{u_2, w_2\}) = F(\{u_3, w_1\}) = F(\{u_3, w_2\}) = F(\{u_2,u_3\}) = c_\triangle$.  We know that $F(\{w_1,w_2\}) = E_F(\{v_1\}) \neq c_\triangle$ by construction as well as \cref{lem:Claim_2}. With these however, we get the contradiction that $F[\{u_2,u_3,w_1,w_2\}]$ is in $\mH(2,1,0)$.
    $$
    
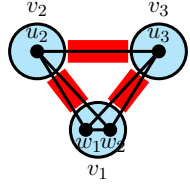
\begin{tikzpicture}[line width=0.4mm,baseline=2ex, scale=0.8]
        \node[circle,fill=cyan,fill opacity=0.3,draw,scale=2] (A) at (0,0) {};
        \node[circle,fill=black, draw,scale=0.4] (Ap1) at (-0.2,0) {};
        \node[circle,fill=black,draw,scale=0.4] (Ap2) at (0.2,0) {};
        \node[circle,fill=cyan,fill opacity=0.3, draw, scale=2] (B) at (-1,1.3) {};
        \node[circle,fill=black,draw,scale=0.4] (Bp) at (-1,1.3) {};
        \node[circle,fill=cyan,fill opacity=0.3, draw, scale=2] (C) at (1,1.3) {};
        \node[circle,fill=black,draw,scale=0.4] (Cp) at (1,1.3) {};
        
        \draw[red, line width=3mm] (A) -- (B);
        \draw[red, line width=3mm] (B) -- (C);
        \draw[red, line width=3mm] (A) -- (C);

        \draw[black] (Ap1) -- (Ap2);
        \draw[black] (Ap1) -- (Bp);
        \draw[black] (Ap1) -- (Cp);
        \draw[black] (Ap2) -- (Bp);
        \draw[black] (Ap2) -- (Cp);
        \draw[black] (Bp) -- (Cp);

        \node[scale=0.8] at (0,-0.7) {$v_1$};
        \node[scale=0.8] at (-0.15,-0.25) {$w_1$};
        \node[scale=0.8] at (0.25,-0.25) {$w_2$};
        \node[scale=0.8] at (-1,2) {$v_2$};
        \node[scale=0.8] at (-1,1.57) {$u_2$};
        \node[scale=0.8]  at (1,2) {$v_3$};
        \node[scale=0.8]  at (1,1.57) {$u_3$};
    \end{tikzpicture}
    $$
    
    Now let us prove the first claim.
    Assume that $E_F$ contains an embedding of $K_{\max(c+1,6)}^{c'}$ at vertices $v_1,\dots,v_{\max(c+1,6)}$ for some $c' \in [c]$. By the just established claim, $v_i \notin \mX_F$ for all $1 \leq i \leq \max(c+1,6)$. This however implies that $Y$ would contain a copy of $K^{c'}_{\operatorname{max(c+1,6)}}$, giving us a contradiction.
\end{proof}

\begin{proposition}\label{lem:standard_is_Ramsey}
   Assume \boxedchar{A1} and that $F \in \mF^\varnothing$ is a $(\varsigma+\varepsilon_1)$-standard flag of order $n_1 = v(F)$. We have $E_F\in\mL_{\operatorname{ex}}^{(c)}$ and $Y_F=\varnothing$ for \whiteboxedchar{$n_1$ sufficiently large} and for every $X\in\mX_F$ we have
    \begin{equation}\label{eq:first_bound_on_cliques}
    (1-\varepsilon_3) \, \frac{n_1}{R_{c-1}(3)-1}\le |X|
    \le  (1+\varepsilon_3) \, \frac{n_1}{R_{c-1}(3)-1}.      
    \end{equation}\label{lem:XisR}
\end{proposition}
\begin{proof}
    Because $E_F$ without its loops contains no $K^i_{\operatorname{max(c+1,6)}}$, $v(E)$ is bounded independently of $n_1$. The elements of $\mX_F$ intersect pairwise in at most one point. Therefore, 
\begin{equation}\label{eq:first_bound_on_triangle_density}
    p(\hat K_3; F)
    \ge
    \frac{\sum_{X\in\mX_F}\binom{X}{3}}{\binom{n_1}{3}} \geq ....
\end{equation}
Now, if $|\mX_F|\le R_{c-1}(3)-2$ then we would further get
\begin{equation*}
    ... \ge
    \frac{|\mX_F|
    \binom{\lfloor (n_1-|Y|)/|\mX_F| \rfloor}{3}}
    {\binom{n_1}{3}}
    =
    \frac{1}{(R_{c-1}(3)-2)^2} + o(n_1).
\end{equation*}
Note that for the last step, we have used that $|Y_F| \leq v(E_F)$ is bounded independent of $n_1$. On  the other hand, by \cref{eq:ramsey_upper} and \boxedchar{A1} we know that 
\begin{equation*}
    \varsigma \leq \frac{1}{(R_{c-1}(3)-1)^2}.
\end{equation*}
So for \whiteboxedchar{$\varepsilon_1$ sufficiently small}, \whiteboxedchar{$n_1$ sufficiently large}, we get a contradiction from
\begin{align*}
    \frac{1}{(R_{c-1}(3)-2)^2} + o(n_1) \leq p(\hat K_3; F) \leq \varsigma + \varepsilon_1 \leq \frac{1}{(R_{c-1}(3)-1)^2} + \varepsilon_1.
\end{align*}
Since $E_F[\mX_F]$ is triangle-free, we cannot have $|\mX_F| \geq R_{c-1}(3)$ and therefore
$|\mX_F|= R_{c-1}(3)-1$ follows.

    To prove that $Y_F=\varnothing$, assume that there existed $v\in Y_F$. Then, since we know that $|\mX_F| = R_{c-1}(3)-1$, there needs to exist two distinct $w_1,w_2\in \mX_F$ so that $v,w_1,w_2$ is a monochromatic triangle in $E_F$. This however would contradict \cref{lem:standard_has_partition}. Therefore, $E_F\in\mL_{\operatorname{ex}}^{(c)}$.

    Finally, assume now that the bound in \cref{eq:first_bound_on_cliques} was not fulfilled, say because $|X|$ was too large. Then, we could bound $p(\hat K_3;F)$ just as in \cref{eq:first_bound_on_triangle_density} and would similarly arrive at a contradiction.
\end{proof}

\begin{lemma}\label{lem:standard_has_uniform_subgraph}
    Assume \boxedchar{A1} and that $F \in \mF^\varnothing$ is a $(\varsigma+\varepsilon_1)$-standard flag of order $n_1 = v(F)$. With probability $1-\varepsilon_4§$ a uniformly chosen $n_2$-sized subset $\mathscr W \subseteq V(F)$ has an embedding $\varphi:V(\mathscr W)\to V(E_F)$ such that $|\varphi^{-1}(v)|/|\varphi^{-1}(w)| \ge 1-\varepsilon_4$ for all $v,w\in V(E_F)$.
\end{lemma}

\begin{proof}
    First, let us note that $\mX_F = \{X_1,\dots, X_{r}\}$ is a covering of $V(F)$ by \cref{lem:standard_is_Ramsey}. The same therefore holds for $C_i = \mathscr W \cap X_i$ with respect to $W$. We can lower bound the probability $P$ of $\mathscr W$ having the desired properties by
    \begin{equation} \label{eq:p}
        P \geq 1 - P_1 - P_2 - P_3,
    \end{equation}
    where $P_1$ denotes that probability that the $C_i$ do not form a partition of $\mathscr W$, $P_2$ the probability that there exists $i,j$ s.t. $|C_i| / |C_j| \le 1- \varepsilon_4$, and $P_3$ the probability that there does not exist an embedding $\varphi$ with the desired properties conditioned on $C_i$ being a partition of $\mathscr W$.

    Writing $P_I$ for the probability that a random function $[n_2] \to [n_1]$ is injective, we note that 
    \[
        \mathbb P[I] = \frac{n_1!}{(n_1-n_2)! \, n_1^{n_2}} \geq 1 - \varepsilon_4
    \]    
    for \whiteboxedchar{$n_1$ sufficiently large}.
    To upper bound $P_1$, let $U = \bigcup_{X,X'\in\mX_F, \, X\ne X'} X\cap X'$ and note that by \cref{lem:Claim_4}  $|U| < {r \choose 2}$. By Boole's inequality, it follows that
    \begin{equation*}
        P_1 \leq n_2 \, |U| / ( n_1 \, P_I) \leq \varepsilon_4/3
    \end{equation*}
    for \whiteboxedchar{$n_1$ sufficiently large}.
    By Chernoff-Hoeffding, we also have that
    \begin{equation*}
        P_2 \leq \varepsilon_4 /3
    \end{equation*}
    for \whiteboxedchar{$n_1$ and $n_2$ large enough} and \whiteboxedchar{$\varepsilon_3$ small enough}. Finally, to upper bound $P_3$, we note that by \cref{lem:standard_is_Ramsey} such an embedding exists whenever the matchings with non-majority color between $X_i$ and $X_j$ are avoided. Let us therefore write
    \begin{equation*}
        M = \bigcup_{i\ne j\in [r]} \{(u,v)\in X_i\ast X_j \mid F(\{u,v\}) = E_F(\{X_i\}) = E_F(\{X_j\})\}
    \end{equation*}
    and note that by \cref{lem:standard_is_Ramsey}
    \begin{equation*}
        |M| \le \binom{{r}}{2}\max_i |X_i|
        \le
        \binom{{r}}{2}\frac{1+\varepsilon_3}{{r}}n_1.
    \end{equation*}
    By Boole's inequality, we therefore have
    \begin{equation} \label{eq:p3}
        P_3 \leq \frac{1}{P_I}\binom{n_2}{2}\frac{|M|}{\binom{n_1}{2}} \leq \varepsilon_4/3
    \end{equation}
    for \whiteboxedchar{$n_1$ large enough}.
    Combining these bounds establishes that  $P \geq 1 - \varepsilon_4$.
\end{proof}

\subsection{Lifting a standard subcoloring}\label{sec:lifting_standard_subcoloring}

Assume that \whiteboxedchar{$\varepsilon_1$ is small enough} to satisfy $(1 + 1/s) \, \varepsilon_1 \le \varepsilon_4 / 2$. Let us also write $r=R_{c-1}(3)-1$ and add an additional assumption to the list in the form of
\begin{itemize} \itemsep0em 
    \item[\boxedchar{A4}] $p(\hat K_3; G)$ is small enough such that the bound in \cref{lem:Claim_1} is at least $1 - \varepsilon_4$.
\end{itemize}

\begin{lemma}\label{lem:Claim_7}
    Assume \boxedchar{A1}, \boxedchar{A3}, and \boxedchar{A4}. 
    There exists a set $W \subseteq V(G)$ with $|W| = n_2$ with a partition $W = C_1 \cup \ldots \cup C_{r}$ 
   satisfying 
    $|C_i| / |C_j| \ge 1-\varepsilon_4$ 
    as well as some unique $E\in\mL_{\operatorname{ex}}^{(c)}$ 
    such that for any $x,y\in V(G)$ chosen uniformly at random,  with probability at least $1-\varepsilon_5$ the induced sub-coloring $G[W \cup \{x, y\}]$ embeds into $E$.
\end{lemma}

% \[
% \begin{tikzpicture}[line width=0.4mm,baseline=-0.6ex, scale=0.8]
%     \node[circle,draw,scale=8] (A) at (-0.4,0.4) {};
%     \node[circle,draw,scale=5] (A) at (-0.2,0.2) {};
    
%     \node[circle,opacity=0.2, fill=white,draw,scale=1.4, line width=1mm] (Ap1) at (-0.75,0.55) {};
%     \node[circle,opacity=0.2, fill=white,draw,scale=1.4, line width=1mm] (Ap2) at (0.4,0.45) {};
%     \node[circle,opacity=0.2, fill=white,draw,scale=1.4, line width=1mm] (Ap3) at (-0.2,-0.5) {};
    
%     \draw[lime, line width=1mm] (Ap1) -- (Ap2);
%     \draw[red, line width=1mm] (Ap1) -- (Ap3);
%     \draw[cyan, line width=1mm] (Ap2) -- (Ap3);

%     \node[scale=0.8] at (-0.4,2.5) {$G$};
%     \node[scale=0.8] at (-0.2,1.6) {$W$};
%     \node[scale=0.8] at (-1.55,0.95) {$C_1$};
%     \node[scale=0.8] at (1.15,0.65) {$C_2$};
%     \node[scale=0.8] at (-0.3,-1.15) {$C_3$};

%     \node[circle,fill=black,draw,scale=0.4] (Q1) at (-0.9,1.6) {};
%     \node[circle,fill=black,draw,scale=0.4] (Q2) at (-1.8,-0.3) {};

%     \draw[red] (Q1) -- (Q2);
%     \draw (Q1) -- (Ap1);
%     \draw[lime] (Q1) -- (Ap2);
%     \draw[red] (Q1) -- (Ap3);
%     \draw[red] (Q2) -- (Ap1);
%     \draw[cyan] (Q2) -- (Ap2);
%     \draw (Q2) -- (Ap3);

%     \node[scale=0.8] at (-0.9,1.9) {$x$};
%     \node[scale=0.8] at (-1.9,0) {$y$};
% \end{tikzpicture}
% \]

\begin{proof}
    % We start by establishing that with probability $???$ a subset $W\subseteq V(G)$ chosen uniformly at random has a partition $W = C_1 \cup \ldots \cup C_r$ with $|C_i|/|C_j|\ge 1-\varepsilon_4$ s.t. that $G[W]$ embeds into $E_G$ as defined in ... with the pre-images of each vertex in this embedding being a distinct part of $W$.
    %
    % \begin{enumerate} \itemsep0em 
    %     \item\label{item:is_blow_up} $G[W]$ partitions into sets $C_1,\dots,C_r$ with $|C_i|/|C_j|\ge 1-\varepsilon_4$ for all $i,j\in [r]$ and $G[W]$ is a blow-up on $C_1,\dots,C_r$
    % \end{enumerate}
    %

    % Finally, we prove the statement of the theorem.

    If $\mathscr U$ is a uniform random subsets of $G$ of size $n_2$, and for each $F\in\mF^\varnothing_{n_1}$, $\mathscr W_F$ is a uniform random subset of $F$ of size $n_2$ then we get
    \begin{align*}
        \mathbb P[G[\mathscr U]\text{ embeds into some $\mL^{(c)}_{\operatorname{ex}}$}]
        &\ge
        \sum_{\substack{F\in\mF_{n_1}^\varnothing\\ F\text{ is $\varepsilon_1$-standard}}} \mathbb P[G[\mathscr W_F]\text{ embeds into } E_F] \, p(F;G) \\
        &\ge(1-\varepsilon_4)^2 \ge 1-2\varepsilon_4,
    \end{align*}
    where the first inequality follows by \cref{lemma:double_counting}, the second inequality follow from \cref{lem:standard_has_uniform_subgraph}, \cref{lem:Claim_1} and \boxedchar{A4}, and the last inequality simply omits the quadratic term. Note that the same also holds if $\mathscr U$ is of size $n_2 +1$ and $n_2 + 2$, assuming \whiteboxedchar{$n_1$, $n_2$, and $\varepsilon_3$} are adjusted accordingly.

    Let $\mathscr W$ be an  $n_2$-sized subset of $G$ chosen uniformly at random and let $\mathscr P_1$ and $\mathscr P_2$ be two vertices in $V(G)$ also chosen uniformly at random. Note that $\mathscr W$, $\mathscr P_1$, and $\mathscr P_2$ are independent and let the following five events partition our probability space
    \begin{align*}
        I_1 &= \{\mathscr P_1\in \mathscr W\land \mathscr P_2 \in \mathscr W\}, \\
        I_2 &= \{\mathscr P_1\in \mathscr W\land \mathscr P_2 \notin \mathscr W\}, \\
        I_3 &= \{\mathscr P_1\notin \mathscr W\land \mathscr P_2 \in \mathscr W\}, \\
        I_4 &= \{\mathscr P_1\notin \mathscr W\land \mathscr P_2 \notin \mathscr W\land \mathscr P_1 = \mathscr P_2\}, \text{ and} \\
        I_5 &= \{\mathscr P_1\notin \mathscr W\land \mathscr P_2 \notin \mathscr W\land \mathscr P_1 \ne \mathscr P_2\} .
    \end{align*}
    Then $\mathscr W \cup \{\mathscr P_1, \mathscr P_2\}$ is a uniform random $n_2$-sized set when conditioned on $I_1$, uniform random $(n_2+1)$-sized set when conditioned on $I_2,I_3$ or $I_4$ and a random $(n_2+2)$-sized subset when conditioned on $I_5$.
    %We see this by taking a set $W\subseteq V(G)$ and computing the probability $\mathbb P[(\mathscr W\cup \{\mathscr P_1,\mathscr P_2\} = W) \cap I_i]$ for some $i=1,\dots, 5$.
    It follows that
    \begin{align*}
    &\mathbb P[\mathscr W \cup \{\mathscr P_1, \mathscr P_2\}\text{ embeds into some }\mL_{\operatorname{ex}}^{(c)}]\\
    &=
    \sum_{i=1,\dots,5} \mathbb P[(\mathscr W \cup \{\mathscr P_1, \mathscr P_2\} \text{ embeds into some } \mL_{\operatorname{ex}}^{(c)}) \text{ and } I_i \text{ holds} \} \ge 1 - 10\varepsilon_4
    \end{align*}
    due to the oberservation at the beginning of this proof. It follows that there
    exists at least one set $W$ with $|W|=n_2$ so that
    \begin{align*}
    \mathbb P[\mathscr W \cup \{\mathscr P_1, \mathscr P_2\}\text{ embeds into some }\mL_{\operatorname{ex}}^{(c)}, \mathscr W = W] \ge 1 - 10\varepsilon_4   
    \end{align*}
    We assume that \whiteboxedchar{$\varepsilon_4$ is sufficiently small} so that we get the lower bound $\ge 1-\varepsilon_5$. When \whiteboxedchar{$n_2$ is sufficiently large}, there exists exactly one $E\in\mL_{\operatorname{ex}}^{(c)}$ into which $G[W]$ embeds, proving the claim.
\end{proof}

Assuming \boxedchar{A1}, \boxedchar{A3} and \boxedchar{A4}, fix a set $W_G \subseteq V(G)$ as guaranteed by \cref{lem:Claim_7} together with a partition $C_1,...,C_r$ and $E \in \mL_{\operatorname{ex}}^{(c)}$ be the looped graph associated with $G[W]$. Let us w.l.o.g. assume that $V(E) = [r]$ and that each of the components $C_i$ correspond to the vertex $i$ in $E$. Now let
\begin{equation}
    \begin{aligned}
        E_{0} &= 
        \{(u,v)\in V(G)\times V(G) \mid G[W_G\cup\{u,v\}]\text{ does not embed into } E\}, \\
        V_{0} &= \{v\in V(G)  \mid v\text{ is incident to }\ge
        \varepsilon_6n_0 \text{ edges in }E_0\}, \\
        F_i &= \{v\in V(G)\setminus V_0  \mid W_G\cup \{v\} \text{ embeds into }E \}, \text{ and} \\
        V_i & = \left\{  v\in V(G)  \mid \ v\ast F_j
        \text{ has }\ge\delta_0 |F_j|\text{ edges with color }E(\{i,j\}) \text{ for each }j\in [r]  \right\}.
    \end{aligned}  
    \label{eq:E0V0FiVi}
\end{equation}

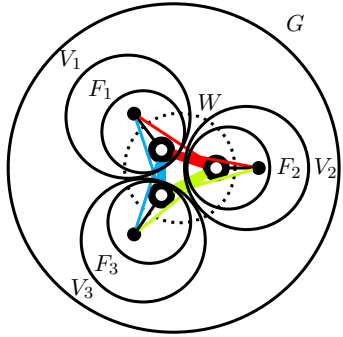
$$
\begin{tikzpicture}[line width=0.4mm,baseline=0ex, scale=0.8]
    \node[circle,,draw,scale=4, dotted] (A) at (0,0) {};
    \node[circle,,draw,scale=12] (Z) at (-0.1,0) {};
    \node[circle,fill=white,draw,scale=0.7, line width=1mm] (Ap1) at (-0.3,0.3) {};
    \node[circle,fill=white,draw,scale=0.7, line width=1mm] (Ap2) at (0.6,0) {};
    \node[circle,fill=white,draw,scale=0.7, line width=1mm] (Ap3) at (-0.3,-0.45) {};
    \node[circle,fill=black,draw,scale=0.4] (V1) at (-0.75,0.9) {};
    \node[circle,fill=black,draw,scale=0.4] (V2) at (-0.75,-1.1) {};
    \node[circle,fill=black,draw,scale=0.4] (V3) at (1.3,0) {};

    \node[scale=0.8] at (-1.3,1.3) {$F_1$};
    \node[scale=0.8] at (-1.8,1.8) {$V_1$};
    \node[scale=0.8] at (1.8,0) {$F_2$};
    \node[scale=0.8] at (2.4,0) {$V_2$};
    \node[scale=0.8] at (-1.2,-1.6) {$F_3$};
    \node[scale=0.8] at (-1.6,-2) {$V_3$};
    \node[scale=0.8] at (0.5,1.1) {$W$};
    \node[scale=0.8] at (1.9,2.4) {$G$};

    \draw[red, line width=1.2mm] (Ap1) -- (Ap2);
    \draw[cyan, line width=1.2mm] (Ap1) -- (Ap3);
    \draw[lime, line width=1.2mm] (Ap2) -- (Ap3);

    \draw[black] (V1) -- (Ap1);
    \draw[red] (V1) -- (Ap2);
    \draw[cyan] (V1) -- (Ap3);

    \draw[cyan] (V2) -- (Ap1);
    \draw[lime] (V2) -- (Ap2);
    \draw[black] (V2) -- (Ap3);

    \draw[red] (V3) -- (Ap1);
    \draw[black] (V3) -- (Ap2);
    \draw[lime] (V3) -- (Ap3);

    \node[circle,,draw,scale=3] (C1) at (-0.6,0.6) {};
    \node[circle,,draw,scale=3] (C2) at (-0.5,-0.9) {};
    \node[circle,,draw,scale=3] (C3) at (0.8,0) {};
       
    \node[circle,,draw, scale=4.5]  at (-0.85,0.85) {};
    \node[circle,,draw,scale=4.5]  at (-0.6,-1.2) {};
    \node[circle,,draw,scale=4.5]  at (1.1,0) {};
\end{tikzpicture}
$$

\begin{lemma} \label{lem:bound_for_zerosets}
    Assume \boxedchar{A3} and \boxedchar{A4}.
    We have $|E_0|<\varepsilon_5 n_0^2$ and $|V_0|\le \varepsilon_6n_0$.
\end{lemma}
\begin{proof}    
    By \cref{lem:Claim_7} we have $|E_0|\le \varepsilon_5n_0^2$.
    Furthermore, the definition of $V_0$ implies the bound $\varepsilon_6n_0|V_0|\le 2|E_0|$ which gives the claim for $V_0$ when \whiteboxedchar{$\varepsilon_5$ is sufficiently small}.
\end{proof}

\begin{lemma}\label{lem:Claim_8}
    Assume \boxedchar{A3} and \boxedchar{A4}. For every $i\in [r]$ we have $|F_i| \geq (1-\varepsilon_7) \, n /r$, the $V_i$ are pairwise disjoint and we have $F_i\subseteq V_i$.
    % so that\todo{fix upper bound}
    %  \[
    %     \frac{1-\varepsilon_7}{{r}} \, n \le
    %     |V_i|
    %     \le \left(\frac{1+\varepsilon_7}{{r}} +\varepsilon_6 + \sqrt{\varepsilon_5}\right) n.
    % \]   
\end{lemma}
\begin{proof}  
    First, note that
    \begin{equation}\label{eq:crude_upper_bound_on_remainder}
        \Big|V(G) \setminus \big(V_0 \cup \bigcup_{i=1}^r F_i \big)\Big|^2\le |E_0|.
    \end{equation}
    Now, assume that for some $i$ we have
    $|F_i|<(1-\varepsilon_7) \, n / r$.
    Then by \cref{eq:crude_upper_bound_on_remainder}, the union $\bigcup_{i\ne j} F_j$ has size
    at least
    \[
       \Big|\bigcup_{i\ne j} F_j \Big| \ge n -  |V_0| - \sqrt{|E_0|} - |F_i|
       \ge
       \left( 1-\varepsilon_6 - \sqrt{\varepsilon_5} - \frac{1-\varepsilon_7}{{r}} \right)\, n
    \]
    by \cref{lem:bound_for_zerosets}.
    The number of monochromatic triangles of $G$ is at least
    \[
        \binom{|F_i|}{3}
        - n|E_0|
        + \sum_{j \neq i} \binom{|F_j|}{3} \geq \binom{|F_i|}{3}
        - n|E_0|
        + (r-1)
        \binom{\frac{1}{r-1}|\bigcup_{i \ne j} F_j|}{3}
    \]
    and therefore for \whiteboxedchar{$\varepsilon_4$, $\varepsilon_5$ and $\varepsilon_6$ sufficiently small } as well as \whiteboxedchar{$n_0$ sufficiently large} and $n\ge n_0$ we would get a contradiction to \boxedchar{A4} since $p(\hat K_3; G)\ge  \varsigma+\varepsilon_4/2$.
    %
    % Proving that $|F_i| \le (1+\varepsilon_7) n_0 / r$ follows the same argument.
    Note that in particular the $F_i$ are non-empty for \whiteboxedchar{$n_0$ sufficiently large}.

    Now let us show that the sets $V_i$ are disjoint when \whiteboxedchar{$\delta_0 > 1/2$}. Otherwise, let $v \in V_i \cap V_j$ for some $i \neq j$ and note that would have at least $\delta_0|F_i|$ edges with color $E(\{i,j\})$ and $\delta_0|F_i|$ edges with color $E(\{i\})$ in $v\ast F_i$ with colors. However, by \cref{lem:standard_is_Ramsey}, loops in $E$ have a different color from their incident edges in $E$, a contradiction, since $F_i$ is non-empty and $\delta_0 > 1/2$. %This proves that the $V_i$ are pairwise disjoint.

    Now let us show that $F_i\subseteq V_i$. For any $v\in F_i$ we know that $v\not\in V_0$, so that $|E_0(v)|<\varepsilon_6n_0$. Therefore, for each $j$ we have
    \[
        \frac{|\{w\in F_j\mid G(\{v,w\})\ne E(\{i,j\}) \}|}{|F_i|}
        \le 
        \frac{|E_0(v)|}{|F_i|}
        \le
        \varepsilon_6 \left(\frac{1-\varepsilon_7}{r}\right)^{-1}.
    \]
    This upper bound is $\le \delta_0$ for \whiteboxedchar{$\varepsilon_6$ sufficiently small} and therefore $v \in V_i$ by definition. Finally, the lower bound for $|V_i|$ follows immediately from that of $|F_i|$.
    % and the upper bound follows from inserting the lower bounds for $|F_i|$ into $|V_i| \leq n - \sum_{j \neq i} |F_j|$.
\end{proof}

\subsection{Properties of Ramsey constructions}\label{sec:Ramsey_mult_graphs}

The following proposition states that $\varsigma = 1 / (R_{c-1}(3)-1)^2$ and  therefore every balanced blow-up from an $E\in\mL_{\operatorname{ex}}^{(c)}$ minimizes the density of monochromatic triangles asymptotically. Notice that unlike \cref{th:stability}, we do not yet get a full characterization of those finite graphs minimizing the $\hat K_3$ density or a stability statement.

\begin{proposition}\label{th:extremals_are_Ramsey}
    Assume \boxedchar{A1} and \boxedchar{A2}.
    We have $\varsigma = 1/(R_{c-1}(3)-1)^{2}$. 
\end{proposition}
\begin{proof}
    Let $\phi_0\in\operatorname{Hom}^+(\mA^\varnothing, \mathbb R)$ be an arbitrary but fixed extremal positive homomorphism for $\hat K_3-\varsigma$, that is $\phi_0(\hat K_3-\varsigma)=0$. By \boxedchar{A2} we therefore also have $\phi_0(\hat K_{3,1}+\hat K_{3,3})=0$. Denote by $\mH$ all flags corresponding to summands in $\hat K_{3,1}$ or $\hat K_{3,3}$. Let us sketch the argument that $\mA_\mH^\varnothing \cong \mA^\varnothing / \langle \mH\rangle$. There is an algebra homomorphism $\pi: \mA^\varnothing \to \mA_\mH^\varnothing$ defined by mapping a flag $F\in \mF$ to $F$ if $F\in \mF^\varnothing_\mH$ and otherwise to 0, see \cite{Razborov_2007} Definition 4 and Theorem 2.6. Clearly, it is surjective and $\langle \mH \rangle$ is in its kernel. To see that its kernel is exactly $\langle \mH \rangle$ we notice that if $\pi(f)=0$ then $\pi(f)$ may be written as a linear combination of elements of $\mK^\varnothing_\mH$. But we can pull back this linear combination to a linear combination of $\mK^\varnothing$ and then it follows that $f$ is in $\langle \mH\rangle$.

    Since $\mA_\mH^\varnothing \cong \mA^\varnothing / \langle \mH\rangle$,
    the map $\phi_0$ also induces a positive homomorphism of $\mA^\varnothing_{\mH}$. Hence, by \cref{thm:homfunctional} there exists a sequence $F_i$ of $\alpha_i$-standard flags with $\lim_{i\to\infty}\alpha_i=\varsigma$ in $\mA^\varnothing_{\mH}$, and therefore in $\mA^\varnothing$, satisfying $\lim_{i \to \infty} p(F;F_i) = \phi_0(F)$. By \cref{lem:standard_is_Ramsey}, the $\hat K_3$ density in $F_i$ is asymptotically $1/(R_{c-1}(3)-1)^{2}$.
\end{proof}

Under the assumptions \boxedchar{A1} and \boxedchar{A2}, \cref{th:extremals_are_Ramsey} shows that the balanced blow-ups of the $R_{c-1}(3)$-Ramsey colorings minimize the density of monochromatic triangles. This has several consequences regarding the structure of any $E \in \mL_{\operatorname{ex}}^{(c)}$.

%We prove now what we need for the rest of the stability proof of \cref{th:stability}.
%
%Throughout the rest of this subsection, let the assumptions of \cref{th:stability} hold and let $E$ be an $

\begin{lemma}\label{lem:Degree_at_lest_two}
    Assume \boxedchar{A1} and \boxedchar{A2}. For any $E \in \mL_{\operatorname{ex}}^{(c)}$ there are at least two $c'$-colored edges in $E$ at $v$ for any $c' \in [c]$ that is not the loop color $c_\ell$ and $v\in V(E)$.
\end{lemma}
\begin{proof}
    There exists at least one such edge, as otherwise
    we could increase the number of vertices of $E$ by duplicating $v$ and setting the connecting edge between $v$ and its duplicate to $c_2$. This would contradict the fact that the $E$ is a Ramsey graph.

    $$
       
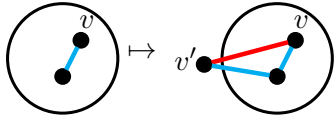
\begin{tikzpicture}[line width=0.4mm,baseline=0ex, scale=0.8]
        \node[circle,fill=white,draw,scale=4] (A) at (0,0) {};
        \node[circle,fill=black,draw,scale=0.5] (Ap2) at (0.3,0.3) {};
        \node[circle,fill=black,draw,scale=0.5] (Ap3) at (0,-0.3) {};

        \draw[cyan, line width=0.6mm] (Ap2) -- (Ap3);
        
        \node at (0.4, 0.6) {$v$};
        \node at (0, -0.6) {};
    \end{tikzpicture}
    \mapsto
    \begin{tikzpicture}[line width=0.4mm,baseline=0ex, scale=0.8]
        \node[circle,fill=white,draw,scale=4] (A) at (0,0) {};
        \node[circle,fill=black,draw,scale=0.5] (Ap2) at (0.3,0.3) {};
        \node[circle,fill=black,draw,scale=0.5] (Ap3) at (0,-0.3) {};
        \node[circle,fill=black,draw,scale=0.5] (B) at (-1.2,-0.1) {};

        \draw[red, line width=0.6mm] (B) -- (Ap2);
        \draw[cyan, line width=0.6mm] (B) -- (Ap3);
        \draw[cyan, line width=0.6mm] (Ap2) -- (Ap3);

        \node at (0.4, 0.6) {$v$};
        \node at (0, -0.6) {};

        \node at (-1.5,0) {$v'$};
    \end{tikzpicture}.
    $$
    
    Assume that only one such edges exists for some vertex $v\in V(G)$. Let $w \ne v$ denote the other vertex with $G(\{v,w\})=c'$. There exists
    a vertex $x$, different from $v$ and $w$, with $c'\ne G(\{w,x\})$. 
    Otherwise, we could duplicate $w$ and choose any color between $w$ and its duplicate that is different from $c'$ without increasing the number of monochromatic triangles.
    $$
    
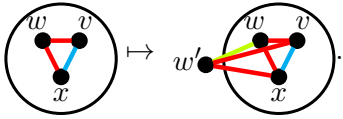
\begin{tikzpicture}[line width=0.4mm,baseline=0ex, scale=0.8]
        \node[circle,fill=white,draw,scale=4] (A) at (0,0) {};
        \node[circle,fill=black,draw,scale=0.5] (Ap1) at (-0.3,0.3) {};
        \node[circle,fill=black,draw,scale=0.5] (Ap2) at (0.3,0.3) {};
        \node[circle,fill=black,draw,scale=0.5] (Ap3) at (0,-0.3) {};

        \draw[red, line width=0.6mm] (Ap1) -- (Ap2);
        \draw[red, line width=0.6mm] (Ap1) -- (Ap3);
        \draw[cyan, line width=0.6mm] (Ap2) -- (Ap3);
        
        \node at (0.4, 0.6) {$v$};
        \node at (-0.4, 0.6) {$w$};
        \node at (0, -0.6) {$x$};
    \end{tikzpicture}
    \mapsto
    \begin{tikzpicture}[line width=0.4mm,baseline=0ex, scale=0.8]
        \node[circle,fill=white,draw,scale=4] (A) at (0,0) {};
        \node[circle,fill=black,draw,scale=0.5] (Ap1) at (-0.3,0.3) {};
        \node[circle,fill=black,draw,scale=0.5] (Ap2) at (0.3,0.3) {};
        \node[circle,fill=black,draw,scale=0.5] (Ap3) at (0,-0.3) {};
        \node[circle,fill=black,draw,scale=0.5] (B) at (-1.2,-0.1) {};

        \draw[red, line width=0.6mm] (Ap1) -- (Ap2);
        \draw[red, line width=0.6mm] (Ap1) -- (Ap3);
        \draw[lime, line width=0.6mm] (B) -- (Ap1);
        \draw[red, line width=0.6mm] (B) -- (Ap2);
        \draw[red, line width=0.6mm] (B) -- (Ap3);
        \draw[cyan, line width=0.6mm] (Ap2) -- (Ap3);

        \node at (0.4, 0.6) {$v$};
        \node at (-0.4, 0.6) {$w$};
        \node at (0, -0.6) {$x$};

        \node at (-1.5,0) {$w'$};
    \end{tikzpicture}.
    $$
    Duplicate $v$ to $v'$ and recolor the edge $\{x,v'\}$ so that it has color $c'$.
    Call the resulting looped graph $E'$. We use the following weights on $E'$:
    \[
        \mathbf w(y) = \frac{1}{v(E)}\cdot
        \begin{cases}
            1 & \text{ if }y\not\in\{v,v'\} \\
            \frac{1}{2} & \text{ otherwise }
        \end{cases}.
    \]
    Thus, $\sum_{i\in [c]} \hat p(K^i_3; E', \mathbf w) = \varsigma$, but $E'$ contains a copy of $\mH(0,2,1)$, contradicting \boxedchar{A2}.
    $$
    
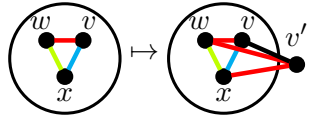
\begin{tikzpicture}[line width=0.4mm,baseline=0ex, scale=0.8]
        \node[circle,fill=white,draw,scale=4] (A) at (0,0) {};
        \node[circle,fill=black,draw,scale=0.5] (Ap1) at (-0.3,0.3) {};
        \node[circle,fill=black,draw,scale=0.5] (Ap2) at (0.3,0.3) {};
        \node[circle,fill=black,draw,scale=0.5] (Ap3) at (0,-0.3) {};

        \draw[red, line width=0.6mm] (Ap1) -- (Ap2);
        \draw[lime, line width=0.6mm] (Ap1) -- (Ap3);
        \draw[cyan, line width=0.6mm] (Ap2) -- (Ap3);

        \node at (0.4, 0.6) {$v$};
        \node at (-0.4, 0.6) {$w$};
        \node at (0, -0.6) {$x$};
    \end{tikzpicture}
    \mapsto
    \begin{tikzpicture}[line width=0.4mm,baseline=0ex, scale=0.8]
        \node[circle,fill=white,draw,scale=4] (A) at (0,0) {};
        \node[circle,fill=black,draw,scale=0.5] (Ap1) at (-0.3,0.3) {};
        \node[circle,fill=black,draw,scale=0.5] (Ap2) at (0.3,0.3) {};
        \node[circle,fill=black,draw,scale=0.5] (Ap3) at (0,-0.3) {};
        \node[circle,fill=black,draw,scale=0.5] (Ap4) at (1.2,-0.1) {};

        \draw[red, line width=0.6mm] (Ap1) -- (Ap2);
        \draw[lime, line width=0.6mm] (Ap1) -- (Ap3);
        \draw[cyan, line width=0.6mm] (Ap2) -- (Ap3);
        \draw[red, line width=0.6mm] (Ap1) -- (Ap4);
        \draw[black, line width=0.6mm] (Ap2) -- (Ap4);
        \draw[red, line width=0.6mm] (Ap3) -- (Ap4);

        \node at (0.4, 0.6) {$v$};
        \node at (1.2, 0.4) {$v'$};
        \node at (-0.4, 0.6) {$w$};
        \node at (0, -0.6) {$x$};
    \end{tikzpicture}
    $$
\end{proof}

\begin{lemma}\label{lem:two_disjoint_monochrom}
     Assume \boxedchar{A1} and \boxedchar{A2}. Let $E \in \mL_{\operatorname{ex}}^{(c)}$ with loop color $c_\ell$. If $f:V(E)\to [c]$ is a function for which there exists $v_0$ with $f(v_0)=c_\ell$ and for each $v\in V(E)\setminus \{v_0\}$ we have
    $f(v)\in [c]\setminus \{c_\ell\}$, then either the set $v_0\ast f^{-1}(c')$ has color $c'$ for every $c'\in[c]$ or there exist two vertices $w\ne x$ so that $f(w)=E(\{w,x\})=f(x)$.
\end{lemma}
\begin{proof}
    Assume there exists a color $c'\in [c]$ so that $v_0\ast f^{-1}(c')$ is not entirely of color $c'$. Then, define a new graph $E'$ on $V(E)\sqcup \{z\}$ so that $E'(z,v) = f(v)$ for all $v\in E$. Consider the following weighting on $E'$ 
    \begin{equation}\label{eq:weighting_on_extension}
        \mathbf w(y) = \frac{1}{v(E)}\cdot
        \begin{cases}
            1 & \text{ if }y\not\in\{v_0,z\} \\
            \frac{1}{2} & \text{ otherwise }
        \end{cases}.
    \end{equation}
    If no two distinct vertices $w,x\in V(E)$ exist with $f(w) = E(\{w,x\}) = f(x)$, we would get $\sum_{i\in [c]} \hat p(K^i_3; E', \mathbf w) = \frac{1}{(R_{c-1}(3)-1)^2} = \varsigma$ but the graph $E'$ contains a copy of $\mH(0,2,1)$, contradicting \boxedchar{A2}.
    $$  
    
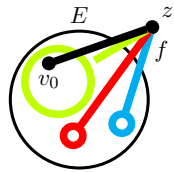
\begin{tikzpicture}[line width=0.4mm,baseline=0ex, scale=0.8]
        \node[circle,black,draw,scale=5] (A) at (0,0) {};
        \node[circle,lime,draw,scale=2.4, line width=1mm] (B1) at (-0.35,0.3) {};
        \node[circle,red,draw,scale=0.7, line width=1mm] (B2) at (-0.1,-0.6) {};
        \node[circle,cyan,draw,scale=0.7, line width=1mm] (B3) at (0.7,-0.4) {};
        \node[circle,fill=black,draw,scale=0.4] (P) at (1.2,1.2) {};
        \node[circle,fill=black,draw,scale=0.4] (Q1) at (-0.5,0.6) {};

        \draw[cyan, line width=1mm] (P) -- (B3);
        \draw[red, line width=1mm] (P) -- (B2);
        \draw[lime, line width=1mm] (P) -- (B1);

        \draw[black, line width=1mm] (Q1) -- (P);

        \node[scale=0.8] at (0,1.4) {$E$};
        \node[scale=0.8] at (1.45,1.45) {$z$};
        \node[scale=0.8] at (1.35,0.8) {$f$};
        \node[scale=0.8] at (-0.5,0.3) {$v_0$};
    \end{tikzpicture}   
    $$
\end{proof}

\begin{lemma}\label{lem:exists_cherry_in_Ramsey}
     Assume \boxedchar{A1} and \boxedchar{A2}. For any $E \in \mL_{\operatorname{ex}}^{(c)}$ with loop color $c_\ell$ and any two distinct vertices $x,y\in V(E)$ and $c'\in [c]\setminus \{c_\ell\}$,
    there exists
    a vertex $w\in E$ 
    satisfying $E(\{x,w\})=E(\{y,w\})=c'$.
\end{lemma}
\begin{proof}
    If there does not exist $w\in E$ 
    satisfying $E(\{x,w\})=E(\{y,w\})=c'$, duplicate $x$ to $x'$ and color the edge $\{x', y\}$ with $c'$.  Call the resulting looped graph $E'$. With the same weighting as in \cref{eq:weighting_on_extension}, we would have the same 
    bound $\frac{1}{(R_{c-1}(3)-1)^2}=\varsigma$ but with a graph containing $\mH(0,2,1)$, contradicting \boxedchar{A2}.
$$
   \begin{tikzpicture}[line width=0.4mm,baseline=0ex, scale=0.8]
        \node[circle,fill=white,draw,scale=4] (A) at (0,0) {};
        \node[circle,fill=black,draw,scale=0.5] (Ap1) at (-0.3,0.3) {};
        \node[circle,fill=black,draw,scale=0.5] (Ap2) at (0.3,0.3) {};
        \node[circle,fill=black,draw,scale=0.5] (Ap3) at (0,-0.3) {};

        \draw[lime, line width=0.6mm] (Ap2) -- (Ap3);
        \draw[cyan, line width=0.6mm] (Ap1) -- (Ap3);
        \draw[red, line width=0.6mm] (Ap1) -- (Ap2);

        \node at (0.4, 0.6) {$y$};
        \node at (-0.4, 0.6) {$x$};
        \node at (0, -0.6) {$w$};
    \end{tikzpicture}
    \mapsto
   \begin{tikzpicture}[line width=0.4mm,baseline=0ex, scale=0.8]
        \node[circle,fill=white,draw,scale=4] (A) at (0,0) {};
        \node[circle,fill=black,draw,scale=0.5] (Ap1) at (-0.3,0.3) {};
        \node[circle,fill=black,draw,scale=0.5] (Ap2) at (0.3,0.3) {};
        \node[circle,fill=black,draw,scale=0.5] (Ap3) at (0,-0.3) {};
        \node[circle,fill=black,draw,scale=0.5] (Ap4) at (0,1.2) {};

        \draw[lime, line width=0.6mm] (Ap2) -- (Ap3);
        \draw[cyan, line width=0.6mm] (Ap1) -- (Ap3);
        \draw[red, line width=0.6mm] (Ap1) -- (Ap2);
        \draw[lime, line width=0.6mm] (Ap2) -- (Ap4);
        \draw[cyan, line width=0.6mm] (Ap3) -- (Ap4);
        \draw[black, line width=0.6mm] (Ap1) -- (Ap4);

        \node at (0.4, 0.6) {$y$};
        \node at (-0.4, 0.6) {$x$};
        \node at (0, -0.6) {$w$};
        \node at (-0.2, 1.5) {$x'$};
    \end{tikzpicture}
$$
\end{proof}

\subsection{Proof of \cref{th:stability}}\label{sec:remove_error_edges}

We will now bring together all of the structural results established in the previous sections to formally prove \cref{th:stability}. That is, assuming \boxedchar{A1}, \boxedchar{A2}, \boxedchar{A3}, we want to show that $\hat K_3-\varsigma$ is stable with respect to $\mG^{(c)}_{\operatorname{ex}}$.
 More specifically, we assume that $\hat{K}_3 - \varsigma$ is a tight positive algebra element satisfying $\hat{K}_3 - \varsigma \ge s(\hat K_{3,1} + \hat K_{3,3})$. Writing
\begin{equation*}
    \varsigma_n = \min_{F \in \mF_n^\varnothing} p(\hat{K}_3; F),
\end{equation*}
where clearly $\lim_{n \to \infty} \varsigma_n = \varsigma$, we want to prove that there exists some $n_0$ so that for any given $\varepsilon > 0$ there exists some $\delta>0$ so that any given $G \in F^\varnothing$ satisfying $p(\hat K_3,G) < \varsigma_n + \delta$ can be turned into an element of $\mG^{(c)}_{\operatorname{ex}}$ as defined in \cref{sec:formal_statement_theorem} by recoloring at most a $\varepsilon$ fraction of all edges.

\medskip

We start by noting that we may always assume
\begin{equation}\label{eq:beta_upper_bound}
    \delta <
    \frac{\varepsilon_4}
    {2\left(\frac{1}{\varepsilon_1(1+\varepsilon_1)}+\frac{\binom{n_1}{6}}{s}\right)}
\end{equation}
which implies that
\boxedchar{A4} is satisfied. Let the set $C_1\cup \dots \cup C_r=W\subseteq V(G)$ as well as $E\in\mL_{\operatorname{ex}}^{(c)}$ be as given by \cref{lem:Claim_7} and $E_0$, $V_0$, $F_i$, and $V_i$ be as given by \cref{eq:E0V0FiVi}. Write
$
    V^\ast=V(G)\setminus \bigcup_i V_i.    
$
and let $E^\ast$ be all pairs $(x,y)$ where $x\in V_i$ and $y\in V_j$  such 
that $G(\{x,y\})\ne E(\{i,j\})$. Note that by \cref{lem:Claim_8}
\begin{equation}
    |V^\ast| \le n - \sum_{i\in [r]}|V_i|\le \varepsilon_7n. \label{eq:Vastupperbnd}
\end{equation}
The goal is to recolor a sufficiently small number of edges in $G$ such that $V^\ast = \emptyset$ as well as $E^\ast = \emptyset$, establishing the claim of \cref{th:stability}. Each of the subsequent claims will recolor some edges without increasing the number monochromatic triangles. After each claim,
we will assume that we have recolored the edges in $G$ accordingly and update both $E^\ast$ and $V^\ast$.%We start by recoloring the edges within each of the parts $V_i$.

\begin{claim}\label{lem:Claim_10}
    We  can turn all $G[V_i]$ into monochromatic cliques with
    the loop color of $E$ by recoloring at most
    \begin{equation*}
        \frac{r}{3}(p(\hat K_3,G)-\varsigma_n)\binom{n}{2}
        \label{eq:Claim_10}
    \end{equation*}
    edges.
\end{claim}
\begin{proof}
    Recoloring any individual edge $\{v,w\} \in G[V_i]$ with $G(\{v,w\}) \neq E(\{i\})$ to the correct color creates no more than
    \begin{align*}
       &|V^\ast|+|V_i| + \sum_{j \neq i} (|V_j|-|F_j|) + |F_j|(1-\delta_0) \le \left(1 - \delta_0\left(1-\frac{1}{r}\right)(1-\varepsilon_7)\right)n 
    \end{align*}
    new monochromatic triangles, where we have used \cref{lem:Claim_8} to lower bound the $|F_j|$.
%
% $$
% \begin{tikzpicture}[line width=0.4mm,baseline=0ex, scale=0.8]
%     \node[circle,,draw,scale=4, dotted] (A) at (0,0) {};
%     \node[circle,,draw,scale=12] (Z) at (-0.3,0) {};
%     \node[circle,fill=white,draw,scale=0.7, line width=1mm] (Ap1) at (-0.3,0.3) {};
%     \node[circle,fill=white,draw,scale=0.7, line width=1mm] (Ap2) at (0.6,0) {};
%     \node[circle,fill=white,draw,scale=0.7, line width=1mm] (Ap3) at (-0.3,-0.45) {};
%     \node[circle,fill=black,draw,scale=0.4] (V1) at (-1.45,0.9) {};
%     \node[circle,fill=black,draw,scale=0.4] (V3) at (-0.75,1.4) {};
%     \node[circle,fill=black,draw,scale=0.4] (V2) at (-0.75,-1.1) {};
%
%     \node[scale=0.8] at (-1.8,1.8) {$V_i$};
%     \node[scale=0.8] at (-1.6,-2) {$V_j$};
%     \node[scale=0.8] at (0.5,1.1) {$W$};
%     \node[scale=0.8] at (1.9,2.4) {$G$};
%     \node[scale=0.8] at (-2.4,0.6) {$V^\ast$};
%     \node[scale=0.8] at (-0.75,1.7) {$w$};
%     \node[scale=0.8] at (-1.45,1.2) {$v$};
%
%     \draw[red, line width=1.2mm] (Ap1) -- (Ap2);
%     \draw[cyan, line width=1.2mm] (Ap1) -- (Ap3);
%     \draw[lime, line width=1.2mm] (Ap2) -- (Ap3);
%
%     \draw[pink] (V1) -- (V3);
%
%     \node[circle,,draw,scale=3] (Vstar) at (-2.1,-0.3) {};
       %
%     \node[circle,,draw, scale=4.5]  at (-0.85,0.85) {};
%     \node[circle,,draw,scale=4.5]  at (-0.6,-1.2) {};
% \end{tikzpicture}
% $$
    %
    On the other hand, the number of destroyed triangles is at least
    \[
        \sum_{\substack{j \neq i\\E(\{i,j\})=G(\{v,w\})}}
        (2\delta_0-1)|F_j|
        \ge 2(2\delta_0-1)\frac{1-\varepsilon_7}{r}n,
    \]
    where we have used \cref{lem:Degree_at_lest_two} so that the summation has at least two summands as well as \cref{lem:Claim_8} to lower bound the $|F_j|$.
    The difference
    between the destroyed and created triangles
    is therefore at least
    \begin{align*}
         &\left((4\delta_0-2)\frac{1-\varepsilon_7}{r}-1 + \delta_0\left(1-\frac{1}{r}\right)(1-\varepsilon_7)\right)n \\
         &=\left(\delta_0-1+\frac{5\delta_0-2}{r}+\left(-\delta_0\left(1-\frac{1}{r}\right) -\frac{4\delta_0-2}{r} \right)\varepsilon_7 \right)n \ge \frac{n}{r}
    \end{align*}
    where the lower bound holds for \whiteboxedchar{sufficiently large $\delta_0$} and \whiteboxedchar{sufficiently small $\varepsilon_7$}. Recoloring more than 
    $r\binom{n}{2}(p(\hat K_3,G)-\varsigma_n)/3$ edges would therefore destroy more than $\binom{n}{3}(p(\hat K_3;G)-\varsigma_n)$ triangles, a contradiction.
    %
    % Then
    % \begin{equation*}
    %     \text{ recolored edges }\times \text{ lower bound } \le \binom{n}{2}(p(\hat K_3,G)-\varsigma_n)
    % \end{equation*}
    % we see that eventually, 
    % Multiplying this with $\binom{n}{2}(p(\hat K_3,G)-\varsigma_n)$, an upper bound on the number of edges $\{v,w\}$, gives the statement.
    % Since this quantitiy may not be repeated more than
    % $\binom{n}{2}(p(\hat K_3,G)-\varsigma_n)$ times, we get the claim.
    % Note that at each edge-recoloring, the sets $V_i$ at most increases in size.
\end{proof}

\begin{wrapfigure}{r}{0pt}
\qquad
\begin{tikzpicture}[line width=0.4mm,baseline=0ex, scale=0.8]
    \node[circle,,draw,scale=4.2, dotted] (A) at (0,0) {};
    \node[circle,,draw,scale=12] (Z) at (-0.3,0) {};
    \node[circle,fill=yellow, fill opacity=0.25, draw,scale=0.7, line width=1mm] (Ap1) at (-0.3,0.3) {};
    \node[circle,fill=yellow, fill opacity=0.25, draw,scale=0.7, line width=1mm] (Ap2) at (0.6,0) {};
    \node[circle,fill=yellow, fill opacity=0.25, draw,scale=0.7, line width=1mm] (Ap3) at (-0.3,-0.45) {};
    \node[circle,fill=black,draw,scale=0.4] (V1) at (-1.45,0.9) {};
    \node[circle,fill=black,draw,scale=0.4] (V2) at (-0.75,-1.1) {};

    \node[scale=0.8] at (-1.8,1.8) {$V_i$};
    \node[scale=0.8] at (-1.6,-2) {$V_j$};
    \node[scale=0.8] at (0.5,1.1) {$W$};
    \node[scale=0.8] at (1.9,2.4) {$G$};
    \node[scale=0.8] at (-1.45,1.2) {$x$};
    \node[scale=0.8] at (-0.9,-1.4) {$y$};
    \node[scale=0.8] at (-0.15,0.65) {$C_i$};
    \node[scale=0.8] at (0.05,-0.7) {$C_j$};

    \draw[red, line width=1.2mm] (Ap1) -- (Ap2);
    \draw[cyan, line width=1.2mm] (Ap1) -- (Ap3);
    \draw[lime, line width=1.2mm] (Ap2) -- (Ap3);
    \draw[red] (V1) -- (V2);
       
    \node[circle,fill=yellow, fill opacity=0.15, draw, scale=4.5]  at (-0.85,0.85) {};
    \node[circle,fill=yellow, fill opacity=0.15, draw,scale=4.5]  at (-0.6,-1.2) {};
\end{tikzpicture}
\bigskip
\end{wrapfigure}
Let us show that we now may assume that
\begin{equation}
    E^\ast \subseteq E_0 \label{eq:f2}.
\end{equation}
For any $(x,y)\in E^\ast$ with $x\in V_i$ and $y\in V_j$ we know that $i \neq j$ by \cref{lem:Claim_10}.
Let us show that $G[\{x,y\}\cup W]$ does not embed into $E$ so that $(x,y) \in E_0$. If it were embeddable, then $x$ would get assigned to the same vertex in $E$ as the class $C_i$ and $y$ to the same vertex as $C_j$ because the edges in $x\ast C_i$ and $y\ast C_j$ are monochromatic
with the loop color of $E$  by \cref{lem:Claim_10}.
By assumption though, $G(\{x,y\})\ne E(\{i,j\})$, which contradicts the embedding.

As a notational convenience,
given a vertex $v\in V(G)$ as well as a set $V_i$ and color $c_1\in [c]$,
denote by $G[v,V_i,c_1]$ the number of $c_1$-colored edges between $v$ and $V_i$ in $G$.

\begin{claim}\label{lem:V_star_is_zero}
    We can ensure $V^\ast = \emptyset$ by recoloring at most
    \begin{equation*}
    2^2 8^2(c-1)^2r^2(p(\hat K_3; G)-\varsigma_n)\binom{n}{2}
    \end{equation*}
   edges.
\end{claim}
\begin{proof}
    Let $v \in V^\ast$ and $i^\ast = i^\ast(v)$ a corresponding index that minimizes $|V_i|$. Let us consider the recoloring obtained by recoloring all edges of $v\ast V_j$ to $E(\{i,j\})$, that is we add $v$ to $V_i$. We will execute this recoloring iteratively until $V^\ast = \emptyset$ and argue that at each step we are reducing the overall number of monochromatic triangles, implying an upper bound on $|V^\ast|$. The only ingredients that will be used in each step are the lower bounds for the $|V_i|$ from \cref{lem:Claim_8}, the upper bound on $|E^\ast|$ due to \cref{eq:f2} and \cref{lem:bound_for_zerosets}
    and the upper bound for $|V^\ast|$ due to \cref{eq:Vastupperbnd}.
    Trivially these will remain true under the recoloring throughout all iterations. We note that $i^\ast$ is not fixed but updated at each iteration.
    
    Let $t_v^{(b)}$ and $t_v^{(a)}$ therefore denote the number of monochromatic triangles incident to $v$ in $G$ \emph{before} and \emph{after} recoloring. We have
    \begin{equation}\label{eq:triangle_upper_bound}
        t_v^{(a)} \leq \binom{|V_{i^\ast}|}{2}   + |E^\ast| + |V^\ast| \, n
        \le
        \left(\frac{1}{2r^2}+\varepsilon_5+\varepsilon_7\right)n^2,
    \end{equation}
    where we have used the upper bounds for $V^*$ and  $E^*$.% and finally the fact that $V_i$ is the smallest set to get $|V_i|\le \frac{n}{r}$.
% \[
% \begin{tikzpicture}[line width=0.4mm,baseline=0ex, scale=0.8]
%     \node[circle,,draw,scale=4, dotted] (A) at (0,0) {};
%     \node[circle,,draw,scale=12] (Z) at (-0.3,0) {};
%     \node[circle,fill=white,draw,scale=0.7, line width=1mm] (Ap1) at (-0.3,0.3) {};
%     \node[circle,fill=white,draw,scale=0.7, line width=1mm] (Ap2) at (0.6,0) {};
%     \node[circle,fill=white,draw,scale=0.7, line width=1mm] (Ap3) at (-0.3,-0.45) {};
%     \node[circle,fill=black,draw,scale=0.4] (V1) at (-2.25,-0.3) {};

%     \node[scale=0.8] at (-1.8,1.8) {$V_i$};
%     \node[scale=0.8] at (-1.6,-2) {$V_j$};
%     \node[scale=0.8] at (0.5,1.1) {$W$};
%     \node[scale=0.8] at (1.9,2.4) {$G$};
%     \node[scale=0.8] at (-2.4,0.6) {$V^\ast$};
%     \node[scale=0.8] at (-2.25,0) {$v$};

%     \draw[red, line width=1.2mm] (Ap1) -- (Ap2);
%     \draw[cyan, line width=1.2mm] (Ap1) -- (Ap3);
%     \draw[lime, line width=1.2mm] (Ap2) -- (Ap3);

%     \node[circle,,draw,scale=3] (Vstar) at (-2.1,-0.3) {};
       
%     \node[circle,,draw, scale=4.5]  at (-0.85,0.85) {};
%     \node[circle,,draw,scale=4.5]  at (-0.6,-1.2) {};
% \end{tikzpicture}
% \]
    We can lower bound $t_v^{(b)}$ by counting those triangles incident to $v$ with both end points in $G\setminus V^\ast$, that is
    \begin{equation}
        \label{eq:basic_lower_bound}
        t_v^{(b)} \ge 
     \sum_{j\le k}
        G[v,V_j,E(\{j,k\})] \, G[v,V_k, E(\{j,k\})]
        \cdot\begin{cases}
            \frac{1}{2} \text{ if } j=k \\
            1 \text{ otherwise }
        \end{cases}-\varepsilon_5n^2
        -r n,
    \end{equation}
    where the term $\varepsilon_5n^2$ accounts for $|E^*|$, the term $-r n$ 
    arises from the fact that we are 
    overcounting, and the coefficient $1/2$ from the fact that
    \[
        \binom{G[v,V_j, E(j)]}{2}
        \ge
        \frac{1}{2} \, G[v,V_j, E(j)]^2 - n.
    \]
    Our task is therefore to lower bound
    \cref{eq:basic_lower_bound} and compare it to the upper bound of $t_v^{(a)}$.
    
    We view the summation in \cref{eq:basic_lower_bound} as a degree 2 polynomial over $cr$ variables $x_{j,d}$ corresponding to $G[v,V_j,d]$ ranging over integers subject to the constraints that all variables are $\ge 0$ and that $|V_j| = \sum_{d} x_{j,d}$. Expressed like this, we are trying to lower bound
    \begin{equation}
        \label{eq:polynomial_to_lower_bound}
        \sum_{j\le k}
        x_{j,E(\{j,k\})} x_{k,E(\{j,k\})}
        \cdot\begin{cases}
            \frac{1}{2} \text{ if } j=k \\
            1 \text{ otherwise }
        \end{cases}.
    \end{equation}
    
    \cref{alg:decrease_triangles} now takes a feasible assignment for \cref{eq:polynomial_to_lower_bound} and produces another feasible assignment with smaller support and attaining a smaller value. The goal is to find a feasible assignment for which we can explicitly compute a lower bound which therefore also lower bounds the original assignment.
    \begin{algorithm}
    \caption{Find a minimal assignment for \cref{eq:polynomial_to_lower_bound}}
    \label{alg:decrease_triangles}
    \begin{algorithmic}
    \REQUIRE $x_{i,d}\in\mathbb N$ with $\sum_{d}x_{i,d} = |V_i|$, loop color $c_\ell$ in $E$, enumeration $i_1,...,i_r$ of $V(E)$.
    % \ENSURE The returned assignment to the variables evaluates to a value in \cref{eq:polynomial_to_lower_bound} that is at least as low as the input.
    
    \FOR{k=1,\dots,r}
        \STATE $d^*\gets \argmin_{d\in[c]} \begin{cases} \frac{1}{2}|V_{i_k}|&\text{if $d=c_\ell$} \\ \sum_{E(\{i_k,j\})=d} x_{j,d}&\text{otherwise} \end{cases}$
        \FOR{ $d\in [c]$}
            \IF{ $d \ne d^*$}
                \STATE $x_{i_k,d} \gets 0$
            \ELSE
                \STATE $x_{i_k,d} \gets |V_{i_k}|$
                \STATE $\alpha(i_k)=d$
            \ENDIF
        \ENDFOR
    \ENDFOR
    \RETURN $\alpha$
    \end{algorithmic}
    \end{algorithm}
    
    Take the given assignment $x_{i,d} = G[v,V_i,d]$ as input to the algorithm.
    By \cref{lem:two_disjoint_monochrom} we have exactly four cases to consider.
    \begin{enumerate}
        \item For at least two different $i,j\in [r]$ we have ${\alpha(i)}=c_\ell = {\alpha(j)}$. In this case \cref{eq:basic_lower_bound} is lower bounded by 
    \begin{equation}
        \label{eq:tvbLower_1}
        \left(\frac{1-\varepsilon_7}{r}\right)^2 n^2 -\varepsilon_5n^2-rn.
    \end{equation}
    %
    % which is bigger than \cref{eq:triangle_upper_bound} by say $\frac{1}{4r^2}n^2$ when \whiteboxedchar{$\varepsilon_5,\varepsilon_7$ are sufficiently small} and \whiteboxedchar{$n_0$ is sufficiently large}.

    % $$
    % \begin{tikzpicture}[line width=0.4mm,baseline=0ex, scale=0.8]
    % \node[circle,fill=white,draw,scale=5] (A) at (0,0) {};

    % \node[circle,fill=black,draw,scale=0.4] (P) at (1.2,1.2) {};

    % \node[circle, fill=black,draw,scale=0.4] (C1) at (-0.2,0.8) {};
    % \node[circle, fill=black,draw,scale=0.4] (C2) at (-0.3,-0.3) {};

    % \draw[black, line width=0.8mm] (P) -- (C1);
    % \draw[black, line width=0.8mm] (P) -- (C2);

    % \node at (1.45,1.45) {$v$};
    % \end{tikzpicture}
    % $$

        \item  There exists no $i$ such that $\alpha(i)=c_\ell$. Then, by maximality of the Ramsey graph, there exists a monochromatic triangle in $E$ extended by $v$ through $\alpha$, giving the lower bound
        
    \begin{equation}
        \label{eq:tvbLower_2}
        \left(\frac{1-\varepsilon_7}{r}\right)^2n^2 - \varepsilon_5n^2-rn.
    \end{equation}
    
%     $$\begin{tikzpicture}[line width=0.4mm,baseline=0ex, scale=0.8]
%     \node[circle,fill=white,draw,scale=5] (A) at (0,0) {};
%     \node[circle,lime, fill=white,draw,scale=1.1, line width=1mm] (B1) at (-0.65,0.3) {};
%     \node[circle, red, fill=white,draw,scale=1.1, line width=1mm] (B2) at (0.3,0.3) {};
%     \node[circle, cyan, fill=white,draw,scale=1.1, line width=1mm] (B3) at (0.2,-0.6) {};
%     \node[circle,fill=black,draw,scale=0.4] (P) at (1.2,1.2) {};

%     \draw[lime, line width=0.8mm] (P) -- (B1);
%     \draw[red, line width=0.8mm] (P) -- (B2);
%     \draw[cyan, line width=0.8mm] (P) -- (B3);

%     \node at (1.45,1.45) {$v$};
% \end{tikzpicture}$$
        \item \label{enum:case_3_alg} There exists an $i$ such that $\alpha(i)=c_\ell$ and there exist two points $x,y\in V(E)$ so that $\alpha(x)=E(x,y)=\alpha(y)$. In this case we again get the lower bound
    
    \begin{equation}
        \label{eq:tvbLower_3}
        \left(\frac{1-\varepsilon_7}{r}\right)^2 n^2-\varepsilon_5n^2-rn.
    \end{equation}
    
%     $$
% \begin{tikzpicture}[line width=0.4mm,baseline=0ex, scale=0.8]
%     \node[circle,fill=white,draw,scale=5] (A) at (0,0) {};
%     \node[circle, fill=black,draw,scale=0.4, line width=1mm] (B2) at (0.3,0.3) {};
%     \node[circle,black, fill=black,draw,scale=0.4, line width=1mm] (B3) at (0.2,-0.6) {};
%     \node[circle,fill=black,draw,scale=0.4] (P) at (1.2,1.2) {};

%     \draw[cyan, line width=0.8mm] (P) -- (B2);
%     \draw[cyan, line width=0.8mm] (P) -- (B3);
%     \draw[cyan, line width=0.8mm] (B2) -- (B3);

%     \node at (1.45,1.45) {$v$};
% \end{tikzpicture}
%     $$
        \item \label{enum:case_4_alg} There exists a vertex $u$ so that $E[u,\alpha^{-1}(d)]$ has color $d$ for all $d\in [c]$.

%     $$
% \begin{tikzpicture}[line width=0.4mm,baseline=0ex, scale=0.8]
%     \node[circle,fill=white,draw,scale=5] (A) at (0,0) {};
%     \node[circle,lime, fill=white,draw,scale=1.1, line width=1mm] (B1) at (-0.65,0.3) {};
%     \node[circle, red, fill=white,draw,scale=1.1, line width=1mm] (B2) at (0.3,0.3) {};
%     \node[circle, cyan, fill=white,draw,scale=1.1, line width=1mm] (B3) at (0.2,-0.6) {};
%     \node[circle,fill=black,draw,scale=0.4] (P) at (1.2,1.2) {};

%     \node[circle, fill=black,draw,scale=0.4] (C1) at (-0.2,0.8) {};

%     \draw[lime, line width=0.8mm] (C1) -- (B1);
%     \draw[red, line width=0.8mm] (C1) -- (B2);
%     \draw[cyan, line width=0.8mm] (C1) -- (B3);
%     \draw[black, line width=0.8mm] (P) -- (C1);
%     \draw[lime, line width=0.8mm] (P) -- (B1);
%     \draw[red, line width=0.8mm] (P) -- (B2);
%     \draw[cyan, line width=0.8mm] (P) -- (B3);

%     \node at (1.45,1.45) {$v$};
%     \node at (-0.2,1.1) {$u$};
% \end{tikzpicture}
%     $$
    \end{enumerate}
    
    The last case does not have an immediate lower bound. Assume therefore that \emph{every enumeration} of $V(E)$ in \cref{alg:decrease_triangles}
    results in  \cref{enum:case_4_alg}. We therefore consider two cases.

    \medskip
    \noindent \textbf{Case 1. }     Assume that for every $w\in V(E)$ there exists an enumeration $i_1,\dots,i_r$ of $V(E)$ so that 
    $i_r=w$ and so that 
    $w$ is not the vertex $u$ from \cref{enum:case_4_alg}. 
For each $w$, by definition of \cref{alg:decrease_triangles},
we may lower-bound $t_v^{(b)}$ by considering the situation when the $r-1$-st iteration of \cref{alg:decrease_triangles} is completed and before we determine the value of all $x_{i_r,d}$ in the $r$-th and last iteration. Since we are in case \cref{enum:case_4_alg}, there exists a vertex $u$ so that
$v\ast V_x$ is assigned the color $E(\{ u,x\})$ for all $x\ne w$. Hence we may lower bound $t_v^{(b)}$ by
\begin{equation}\label{eq:last_step_lower_bound}
    \binom{|V_u|}{2} + \frac{1}{2} G[v,V_w,E(\{w\})]^2 + \sum_{\substack{x\in V(E)\setminus\{u,w\} \\ E(\{w,x\}) = E(\{u,x\})}}|V_x| G[v,V_w,E(\{u,x\})] -\varepsilon_5n^2-rn.
\end{equation}
By \cref{lem:exists_cherry_in_Ramsey}, for every color $d\in [c]\setminus \{E(\{u,w\})\}$, the term $G[v,V_w,d]$ occurs in \cref{eq:last_step_lower_bound}. 
We assume that for all $w\in V(E)$ with respectively $u\in V(E)$ we have that, $G[v,V_w,E(\{u,w\})]\ge \frac{7}{8} |V_w|$. Otherwise, we get the lower bound
\begin{equation} \label{eq:tvbLower_4}
    \text{\cref{eq:last_step_lower_bound}} \ge 
    \left(\frac{1}{2} + \frac{1}{2\cdot 8^2(c-1)^2}\right)\left(\frac{1-\varepsilon_7}{r}\right)^2n^2 
    -\varepsilon_5 n^2 -rn.
\end{equation}
Then, there exists a unique function
$\alpha:V(E)\to[c]$ indicating that for each $w\in V(E)$, the set $v\ast V_w$ has at least $\frac{7}{8}|V_w|$ edges of color $\alpha(w)$. But by the maximality of the Ramsey graph, there will exist $u,w\in V(E)$ so that $\alpha(u)=E(\{u,w\}) = \alpha(w)$. Hence, we may lower bound $t_v^{(b)}$ by
\begin{equation} \label{eq:tvbLower_5}
    \frac{9}{16} \left(\frac{1-\varepsilon_7}{r}\right)^2n^2.
\end{equation}

    \medskip
    \noindent \textbf{Case 2. }   Assume now that there exists a vertex $w$ so that every enumeration $i_1,\dots,i_r$ of $V(E)$ with $i_r=w$ in \cref{alg:decrease_triangles} has $i_r$ as the vertex $u$ from \cref{enum:case_4_alg}. Choose an enumeration $i_1,\dots,i_r$ of $V(E)$ so that $i_r = w$ and $i_{r-1}=x$. Let $c_\ell$ denote the loop color of $E$. Then, considering again the second to last step, we have
\begin{align*}
    t_v^{(b)} \ge \binom{G[v,V_x,c_\ell]}{2} + \sum_{d\in [c]\setminus \{c_\ell\}}
    \sum_{\substack{j\in V(E)\setminus\{x,w\} \\ E(\{x,j\})=d}}
    G[v,V_x,d]|V_j|
\end{align*}
By \cref{lem:Degree_at_lest_two}, for every color $d\in [c]\setminus \{c_\ell\}$ in the first summation, there exists at least one non-zero term in the second summation. This means that we have
\begin{equation}\label{eq:case2_first_special_case}
    \ge \binom{G[v,V_i,c_\ell]}{2} + (|V_i|-G[v,V_i,c_\ell])\frac{1-\varepsilon_7}{r}n
\end{equation}
If we had $G[v,V_i,c_\ell] \le \frac{3}{4}|V_i|$ for some $i$, then the minimum value for \cref{eq:case2_first_special_case} with respect to $G[v,V_i,c_\ell]$ would occur at $G[v,V_i,c_\ell] = \frac{3}{4}|V_i|$ and we would get 
\begin{equation} \label{eq:tvbLower_6}
    t^{(b)}_v\ge \left(\frac{9}{2\cdot 16} + \frac{1}{4}\right)\left(\frac{1-\varepsilon_7}{r}\right)^2n^2
     -\varepsilon_5 n^2 -rn.
\end{equation}
Assume that we always have $G[v,V_i,c_\ell] \ge \frac{3}{4}|V_i|$. Then we would get
\begin{equation} \label{eq:tvbLower_7}
    t_v^{(b)} \ge \sum_{i\in V(E)\setminus \{w\}}\left( \frac{3}{2\cdot 4}|V_i|^2\right)
    -\varepsilon_5 n^2 -rn
     \ge \left(\frac{1-\varepsilon_7}{r}\right)^2n^2
     -\varepsilon_5 n^2 -rn
\end{equation}
since $r \geq 3$.

\medskip

We note that all previous lower bounds for $t_v^{(b)}$, from \cref{eq:tvbLower_1} to \cref{eq:tvbLower_7}, are greater than our upper bound for $t_v^{(a)}$ stated in \cref{eq:triangle_upper_bound} by at least
\begin{equation*}
    \frac{1}{2^2\cdot 8^2(c-1)^2r^2}n^2.
\end{equation*}
when \whiteboxedchar{$\varepsilon_5$ and $\varepsilon_7$ are sufficiently small}.
We can destroy at most 
\begin{equation*}
    |V^*| \, \frac{1}{2^2\cdot 8^2(c-1)^2r^2}n^2 \le (p(\hat K_3; G)-\varsigma_n)\binom{n}{3}
\end{equation*}
triangles, that is we get an upper bound on $|V^\ast|$ depending on $p(\hat K_3; G)-\varsigma_n$. It follows that we have to recolor  at most
\begin{equation*}
    |V^*|n \le 2^2\cdot 8^2(c-1)^2r^2(p(\hat K_3; G)-\varsigma_n)\binom{n}{2}
\end{equation*}
edges.
\end{proof}

\begin{claim}
    After recoloring at most
    \begin{equation*}
         \frac{r+1}{3(2\delta_0-1)}(p(\hat K_3; G)-\varsigma_n)\binom{n}{2}
    \end{equation*}
    edges, the set $E^*$ will consist only of loop colored edges.
    \label{claim:f4}
\end{claim}
\begin{proof}
    Let $(x,y)\in E^\ast$ with $x\in V_i$, $y\in V_j$ and $E(\{x,y\})$ is not the loop color of $E$. By
    \cref{lem:exists_cherry_in_Ramsey} there exists at least one
    $k$ so that $E(\{i,k\}) = E(\{j,k\})=G(\{x,y\})$. By
    definition of $V_i$, both $x\ast F_k$ and $y\ast F_k$
    each have more than $\delta_0 |F_k|$ edges with color $G(\{x,y\})$ in $F_k$.
    Therefore, there are at least 
    \[
        (2\delta_0-1)|F_k|\ge (2\delta_0-1)\frac{1-\varepsilon_7}{r}n\ge (2\delta_0-1)\frac{1}{r+1}n
    \]
    monochromatic
    triangles containing $\{x,y\}$, where we have used \cref{lem:Claim_8} to lower bound $|F_k|$ and assume that \whiteboxedchar{$\varepsilon_7$ is sufficiently small} for the second inequality to hold. Note that each choice of $(x,y) \in
    E^\ast$ dictates its own set of of such monochromatic triangles, so we can sum them up without overcounting. If more than
    $\frac{r+1}{2\delta_0-1}(p(\hat K_3; G)-\varsigma_n)\binom{n}{2}/3$
    of these edges were to exist, they would lead to a decrease of the monochromatic triangle count of more than $(p(\hat K_3,G)-\varsigma_n)\binom{n}{3}$,
    a contradiction. Note that recoloring all of these edges $\{x,y\}$ respectively to $E(\{i,j\})$ \emph{simultaneously} does not create any additional triangles. Indeed, after recoloring, the only triangles left are those with the loop color. But we did not change any edge to the loop color and therefore we did not add any triangles. This establishes the claim.
    % Recolor all these edges in $E^\ast$ simultaneously. In the 
%     process, we do not create any new triangles, giving us the claim.
\end{proof}

% The next \cref{claim:f5} makes repeated use of the removal lemma due to \dots.
\begin{claim}\label{claim:f5}
    There exists a $\delta'$ so that when $p(\hat K_3;G)-\varsigma_n\le \delta'$, at most $\frac{\varepsilon}{3}\binom{n}{2}$ edges in $\bigcup_{i\ne j}V_i\ast V_j$ need to be recolored so that the only monochromatic triangles in $G$ are contained within the $V_i$.
\end{claim}
\begin{proof}
Note that all triangles in $G$ that are not contained entirely within a $V_i$ are either within $V_i\cup V_j$ with $i\ne j$ and otherwise within $V_i\cup V_j\cup V_k$ with $i\ne j\ne k\ne i$. 
Denote by $H$ the graph with $V(H)=V(G)$ so that $e\in E(H)$ if and only if $e$ has the loop color of $E$ in $G$ and is contained in some $V_i\ast V_j$ with $i\ne j$. Note that $H$ is $r$-partite with parts of almost equal size. Applying the triangle removal lemma to $\varepsilon/6$ gives us a $\tilde\delta$ so that when $p(K_3;H)\le \tilde\delta$, all triangles in $H$ can be removed by deleting at most $\frac{\varepsilon}{6}\binom{n}{2}$ edges. Notice that when $p(\hat K_3;G)\le \varsigma_n+\tilde\delta$, we get $p(K_3;H)\le \tilde\delta$. Therefore, by recoloring at most $\frac{\varepsilon}{6}\binom{n}{2}$ edges, the only triangles left in $G$ are those with say two points in a $V_i$ and one point in a $V_j$.

For each $i\ne j$ let $H_{i,j} = H[V_i\ast V_j]$. Applying the removal lemma for paths of length two on three vertices with $\varepsilon / (6\binom{r}{2})$ gives us a $\hat\delta$ so that when $p(P_2,H_{i,j})\le \hat \delta$, at most $\varepsilon / (6\binom{r}{2})$ edges need to be recolored to make $H_{i,j}$ free of any $P_2$. Note that we have
\begin{align*}
    p(P_2;H_{i,j}) \le \frac{(p(\hat K_3;G)-\varsigma_n)\binom{n}{3}}{|H_{i,j}|}
    \le \frac{(p(\hat K_3;G)-\varsigma_n)\binom{n}{3}}
        {\binom{2\lfloor\frac{1-\varepsilon_7}{r}\rfloor n}{3}}
    \le\hat \delta.
\end{align*}
for every $i\ne j$ by \cref{lem:Claim_8}  assuming that
\begin{equation*}
    p(\hat K_3;G) \le \varsigma_n +
        \hat\delta\min_{n\ge n_0}
            \frac{\binom{2\lfloor\frac{1-\varepsilon_7}{r}\rfloor n}{3}}{\binom{n}{3}}.
\end{equation*}
Here we have assumed that \whiteboxedchar{$n_0$ is large enough} such that $\binom{2\frac{1-\varepsilon_7}{r}n}{3} > 0$. Hence we may choose
\begin{equation*}
    \delta ' =  \operatorname{min} \left( \tilde\delta, \hat\delta\inf_{n\ge n_0}
            \frac{\binom{2\lfloor\frac{1-\varepsilon_7}{r}n\rfloor }{3}}{\binom{n}{3}} \right)
\end{equation*}
to prove the claim.
\end{proof}
\begin{claim}\label{claim:f6}
    There exists a $\delta''$ so that when $p(\hat K_3;G)-\varsigma_n\le \delta''$, at most
    $\frac{\varepsilon}{3}\binom{n}{2}$ edges need to be recolored to achieve that $||V_i|-|V_j||\le 1$ when $i\ne j$.
\end{claim}
\begin{proof}
Let $\rho_i = |V_i|/n$. Denote by $\rho\in \mathbb R^r$ the vector $\rho=(\rho_1,\dots,\rho_r)$. Let $R_n\subseteq\mathbb R^r$ be all vectors $s$ so that $\sum_i s_i=1$ and $ns$ is a vector of integers so that no two integers differ by more than 1.
For any given $\rho$, the number of vertices that need to be recolored is at most
\begin{equation}\label{eq:dist_to_optimimum}
    2\min_{s\in R_n} \|\max(\rho-s,0)\|_1\binom{n}{2}
\end{equation}
where “$\max$" is interpreted as acting on each component. 
The triangle density is
    \begin{equation*}
        p(\hat K_3;G) = \frac{
        \sum_{i\in [r]}\binom{n\rho_i}{3}}{\binom{n}{3}}
        =
        Q_n(\rho)
    \end{equation*}
    which is a degree three polynomial in the $\rho_i$ for which $Q_n(s)=\varsigma_n$ when $s\in R_n$. 

We want to compute for a given $1>x>0$
\begin{align*}
    \min_{\min_{s\in R_n}\|\max(\kappa-s,0)\|_1=x} Q_n(\kappa)-\varsigma_n
\end{align*}
where $\kappa$ ranges over vectors in the standard simplex $\Delta^r$. By symmetry we may assume that 
\[
    s = \frac{1}{n}\left(\lceil \frac{n}{r}\rceil, \dots, \lceil \frac{n}{r}\rceil, \lfloor \frac{n}{r}\rfloor, \dots, \lfloor \frac{n}{r}\rfloor\right)
\]
and that $\kappa$ exceeds $s$ on the first $l$ coordinates. 
In this case, we find that the global minimum is at most $x^3/r^2$.
% \begin{align*}
%     &\frac{(n\mod r)\binom{\frac{n}{r}(x+\frac{(n\mod r)-r}{n})+\lceil \frac{n}{r}\rceil}{3} +
%     (r - (n\mod r))
%     \binom{\frac{n}{r} (x+\frac{n\mod r}{n}) + \lfloor \frac{n}{r}\rfloor}{3}}{\binom{n}{3}}
%     -\varsigma_n \\
%     \ge& r\frac{\left(\frac{n}{r}x\right)^3 }{\binom{n}{3}}
%     \ge \frac{x^3}{r^2}.
% \end{align*}
Therefore, when $0<\delta'' < \frac{1}{r^2}\left(\frac{\varsigma_n}{6}\right)^3$, this fulfills the assumptions of our claim: When $p(\hat K_3,G)\le \varepsilon_n+\delta''$ then necessarily $2\min_{2\in R_n}\|\max(\rho-s,0) \|_1\binom{n}{2}\le \frac{\varsigma_n}{6}$ which means that at most $\frac{\varsigma_n}{6}$ many edges need to be recolored to make the $V_i$ almost balanced. 
\end{proof}
It follows that, by picking 
\begin{align*}
    \delta &\le \min \left( \frac{\varepsilon_4}
    {2\left(\frac{1}{\varepsilon_1(1+\varepsilon_1)}+\frac{\binom{n_1}{6}}{s}\right)},
        \frac{\frac{\varepsilon}{9}}{\frac{r}{3}},\frac{\frac{\varepsilon}{9}}{2^2\cdot 8^2(c-1)^2r^2},\frac{\frac{\varepsilon}{9}}{\frac{r+1}{3(2\delta_0-1)}}, \delta', \delta''\right)
\end{align*}
with the individual components due to \cref{eq:beta_upper_bound}, \cref{lem:Claim_10}, \cref{lem:V_star_is_zero}, \cref{claim:f4}, \cref{claim:f5}, and \cref{claim:f6}, $G$ can be be turned into an element of $\mG_{\operatorname{ex}}^{(c)}$ by recoloring at most $\varepsilon\binom{n}{2}$ edges, assuming that $p(\hat{K}_3; G) \leq \varsigma_n + \delta$. \hfill $\blacksquare$

\section{Proofs of the main statements} \label{sec:proof}

All certificates along with a description of their structure and the code required to verify them can be found online at \url{github.com/FordUniver/kps_trianglemult}. The problem simplifications presented in \cref{sec:improvements}, in particular the the one in \cref{sec:color_invariance}, were crucial for the flag algebra calculations to become tenable. The block diagonalization presented in \cref{sec:block_diagonalization} has the additional benefit of making the \emph{exact} (rather than numeric) verification of positive-semidefinitess of the certificate computationally feasible.

We note that we relied on the graph isomorphism algorithm implemented in {\tt sage}~\cite{sagemath} in order to compute all necessary flags and their densities. We used {\tt csdp}~\cite{Borchers_1999}  to find the certificates, which, despite its age, seemed to be the only readily available solver capable of solving SDPs of this scale to a sufficient level of accuracy. We also relied on the exact LP solving capabilities of {\tt SoPlex}~\cite{GleixnerSteffyWolter2015, eifler2023computational, gleixner2012improving, gleixner2020linear} in order to turn the floating-point based output of {\tt csdp} into exact algebraic entities. Finally, we note that we found it crucial to factor out known zero eigenvectors based on the conjectured extremal constructions, besides also enforcing them through linear constraints in the rationalizing LP formulation, in order to obtain a sufficiently well conditioned solution out of {\tt csdp} that maintained positive semidefiniteness after rationalizing the solution.

\begin{proof}[Proof of \cref{eq:fourcolortriangles}]
    The upper bound follows immediately from \cref{eq:ramsey_upper} and by noting that $R(3,3,3) = 17$. A matching lower bound is established through a flag algebra SDP certificate with $N = 5$. Stability is obtained through \cref{th:stability} and by noting that all elements of $\mF^\varnothing_5$ in which $\hat{K}_{3,1}$ or $\hat{K}_{3,3}$ have positive density correspond to non-zero slack in the certificate.
\end{proof}

\begin{proof}[Proof of \cref{prop:m334}]
    The upper bound follows immediately from \cref{{eq:ramsey_upper_offdiag}} and by noting that $R(3,3) = 6$. A matching lower bound is established through a flag algebra SDP certificate with $N = 6$. 
\end{proof}

\section{Discussion and Outlook}\label{sec:discussion}

The computational improvements suggested in \cref{sec:improvements} were crucial in order to derive a certificate for the upper bound and stability statement in \cref{eq:fourcolortriangles}. They are applicable whenever the problem studied exhibits symmetries with respects to the colors, with the reduction of the number of constraints essentially factorial in the number of colors. We therefore hope that they find further use for other problems, for example for improved upper bounds on Ramsey numbers through flag algebras, as recently done by Lidicky and Pfender~\cite{lidicky2021semidefinite}.  We also applied the same techniques to derive (marginally) improved lower bounds for the two-color $K_4$- and $K_5$-Ramsey multiplicity using $N = 9$, with the floating-point output of {\tt csdp} suggesting $m_2(4) \geq 0.02961$ and $m_2(5) \geq 0.001557$. Since no matching or conjectured upper bound exists, we did not turn these results into exact rational values.

The improvements however are largely not applicable when there are no previously ignored symmetries in the problem statement, as is for example the case with the famous $(3,4)$-Tur{\'a}n conjecture. Block diagonalization is still possible, but without first applying \cref{sec:color_invariance} there are significantly fewer symmetries to take into consideration, leading to a less significant reduction in the number of variables. They may also not be helpful for applications beyond graphs~\cite{baber2012turan, balogh2014upper, sliacan2018improving, rue2023rado}, where there can be more drastic jumps on the numbers of constraints as $N$ is increased and even the results stated in \cref{sec:color_invariance} do not necessarily make the next step computationally feasible.

Besides these computational improvements, it is notable that our generalization of the stability argument from~\cite{CummingsEtAl_2013} no longer requires explicit knowledge of the Ramsey colorings underlying the extremal construction. While in our case the colorings were both known and crucial in order to derive an exact rather than a floating point-based flag algebra certificate, \cref{th:stability}, which draws a connection between the Ramsey number $R_{c-1}(3)$ and the Ramsey multiplicity problem $m_c(3)$, in theory opens up an avenue to establish a sort of equivalence of the two problems without first explicitly solving both or even either problem: 

\begin{itemize} \setlength\itemsep{0em}
    \item[1)] We could derive a flag algebra certificate for a particular $c > 4$ establishing $m_c(3)$ and meeting the requirements of \cref{th:stability} without explicit knowledge of the $R_{c-1}(3)$-Ramsey colorings. Note that this would imply the exact value of $R_{c-1}(3) = m_c(3)^{-1/2} - 1$.
    \item[2)] We could show that the Ramsey multiplicity problem satisfies the requirements of \cref{th:stability} for arbitrary $c \geq 3$, in particular that $\hat K_{3,1}$ and $\hat K_{3,3}$ have zero density in an extremal construction, through a purely theoretical argument not relying on the semidefinite programming method and without explicitly determining $m_c(3)$. This would imply that $m_c(3) = (R_{c-1}(3) - 1)^{-2}$ without giving us explicit knowledge of either value.
\end{itemize}

At the risk of extrapolating from a sample size of two, this fact motivates us to go so far as to conjecture the following to be true.
\begin{conjecture}\label{conj:equivalence}
    For any $c \geq 3$, we have $m_c(3) = (R_{c-1}(3) - 1)^{-2}$ and the only extremal constructions are derived from $R_{c-1}(3)$-Ramsey colorings.
\end{conjecture}
Suppose that for some $c\ge 3$ we have actually established that $\hat K_{3,1}$ and $\hat K_{3,3}$ have zero density in any extremal construction. Then, aside from establishing \cref{conj:equivalence} for that particular $c$, this would also imply that any two distinct $c$-color Ramsey-colorings $R_1, R_2$ cannot share any $r-1$-sized subcolorings. This follows since otherwise we could, just as in \cref{sec:Ramsey_mult_graphs}, consider a looped coloring $E$ on $R_1\cup \{v_2\}$ with $E(\{v_2\})=E(\{v_1,v_2\})=E(\{v_1\})$ so that $E[V(R_1)\setminus \{v_1\}]\simeq R_2$. By choosing a weighting on $E$ where every $v\in V(E)$ has weight $\frac{1}{r}$ except for $v_1$ and $v_2$ which get the weight $\frac{1}{2r}$, we would have a weighted looped coloring achieving a monochromatic triangle density of $\varsigma$ with a non-zero $\hat K_{3,1}$ density.

Finally, it should be noted that Fox and Wigerson~\cite{fox2023ramsey} somewhat recently characterized an infinite family of $2$-colorings for which an upper bound equivalent to the one given by \cref{eq:ramsey_upper} is tight, i.e., Tur{\'a}n graphs determine the extremal constructions for the respective Ramsey multiplicity problem. They also obtained results for the case of $c=3$ colors that are conditioned the conjecturec bound $R(t, \lceil t/2 \rceil) \leq 2^{-31} \, R(t,t)$.

\paragraph{Acknowledgements} This work was partially funded by the Deutsche Forschungsgemeinschaft (DFG, German Research Foundation) under Germany’s Excellence Strategy – The Berlin Mathematics Research Center MATH+ (EXC-2046/1, project ID: 390685689).

% \bibliography{bib}
%

\end{document}